\documentclass{article}

\newcommand{\papertitle}{A local Torelli theorem for log symplectic manifolds}

\usepackage{titling}

\usepackage{amsfonts,amsmath,amssymb,amsthm}
\usepackage{mathscinet}
\usepackage{mathrsfs}

\usepackage[font=small,labelfont=bf]{caption}
\numberwithin{equation}{section}

\usepackage{todonotes}

\usepackage{xcolor,hyperref}
\definecolor{darkred}{rgb}{.8,0,0}
\definecolor{tocolor}{rgb}{.1,.1,.1}
\definecolor{urlcolor}{rgb}{.2,.2,.6}
\definecolor{linkcolor}{rgb}{.1,.1,.5}
\definecolor{citecolor}{rgb}{.4,.2,.1}
\definecolor{gray}{rgb}{.8,.8,.8}

\hypersetup{
	backref=true,
	colorlinks=true,
	urlcolor=urlcolor,
	linkcolor=linkcolor,
	citecolor=citecolor,
	pdfauthor = {Mykola Matviichuk and Brent Pym and Travis Schedler},
	pdftitle = {\papertitle}
	}
	
\newcommand{\email}[1]{\href{mailto:#1}{#1}}

\usepackage{aliascnt}

\newcommand{\thdef}[2]{
	\newaliascnt{#1}{theorem}  
	\newtheorem{#1}[#1]{#2}
	\aliascntresetthe{#1}  
	\newtheorem*{#1*}{#2}
	\expandafter\newcommand\expandafter{\csname #1autorefname\endcsname}{#2}
}

\newtheorem{theorem}{Theorem}[section]
\newtheorem*{theorem*}{Theorem}
\thdef{conjecture}{Conjecture}
\thdef{lemma}{Lemma}
\thdef{corollary}{Corollary}
\thdef{proposition}{Proposition}
\theoremstyle{definition}
\thdef{definition}{Definition}

\theoremstyle{remark}
\thdef{remarkx}{Remark}
\thdef{examplex}{Example}

\newenvironment{example}
  {\pushQED{\qed}\examplex}
  {\popQED\endexamplex}

\newenvironment{remark}
  {\pushQED{\qed}\remarkx}
  {\popQED\endremarkx}
  
\newcounter{runexcount}

\newcommand{\runex}{\stepcounter{runexcount}Running example, part \arabic{runexcount}}

\newcommand{\defn}[1]{\textbf{\emph{#1}}}

\usepackage{xypic}

\usepackage{tikz}
\usetikzlibrary{decorations.markings}

\newcommand{\sagepentagon}[2][0.4]{
\pgfdeclarelayer{background layer}
\pgfdeclarelayer{foreground layer}
\pgfsetlayers{background layer,main,foreground layer}
\begin{tikzpicture}[baseline=-1ex,rotate=-360/5,scale=#1,line join = round,decoration={
    markings,
    mark=at position 0.55 with {\arrow{>}}}
    ] 
\draw[white] (162:1.51) node {};
\draw[white] (342:1.1) node {};
\coordinate (v0) at (90:1.5);
\coordinate (v1) at (162:1.5);
\coordinate (v2) at (234:1.5);
\coordinate (v3) at (306:1.5);
\coordinate (v4) at (18:1.5);
\begin{pgfonlayer}{main}
\foreach \x in {v0,v1,v2,v3,v4}
{
	\foreach \y in {v0,v1,v2,v3,v4}
	\draw[black!50!white] (\x) -- (\y);
}
\end{pgfonlayer}
#2
\begin{pgfonlayer}{foreground layer}
\foreach \x in {v0,v1,v2,v3,v4}
{
	\draw[fill,gray] (\x) circle (0.03);
}
\end{pgfonlayer}
\end{tikzpicture}
}

\newcommand{\smoothedge}[5][blue]{
\begin{scope}
\def\Ea{#2}
\def\Eb{#3}
\def\Ra{#4}
\def\Rb{#5}
\draw[very thick,#1] (\Ea) -- (\Eb);
\foreach \x in {\Ra,\Rb} {
		\begin{pgfonlayer}{background layer}
		\begin{scope}
			\clip (\Ea) -- (\x) -- (\Eb);
			\draw[opacity=0.5,fill,#1] (\x) circle (0.7);
		\end{scope}
		\end{pgfonlayer}
}
\end{scope}
}

\newcommand{\zeroedge}[3][white]{\draw[dashed,#1!50!white] (#2) -- (#3);}

\newcommand{\abrac}[1]{\left\langle#1\right\rangle}

\newcommand{\rbrac}[1]{\left(#1\right)}

\newcommand{\set}[2]{\left\{#1 \,\middle|\, #2 \right\}}

\newcommand{\CC}{\mathbb{C}}
\newcommand{\PP}{\mathbb{P}}

\newcommand{\ZZ}{\mathbb{Z}}
\newcommand{\QQ}{\mathbb{Q}}

\newcommand{\iu}{\sqrt{-1}}

\newcommand{\ord}{m}
\newcommand{\orddel}{\ord(\edge,\Delta)}

\newcommand{\torloop}{\ell}
\newcommand{\torcyc}{\widehat{\torloop}}

\newcommand{\X}{\mathsf{X}}
\newcommand{\tX}{\widetilde{\X}}
\newcommand{\Xs}{\X^{\mathrm{s}}}

\newcommand{\bdX}[1][]{\partial^{#1}\mathsf{X}}
\newcommand{\Xdel}{{\X_\Delta}}
\newcommand{\Xodel}{{\Xo_\Delta}}
\newcommand{\bdXdel}[1][]{\partial^{#1}\mathsf{X}_\Delta}
\newcommand{\bdXdelr}{\partial\Xdelr}
\newcommand{\Xdelr}[1][]{\mathsf{X}^{\mathrm{res}}_\Delta}

\newcommand{\fX}[1]{\bdX[#1]}
\newcommand{\Xo}{\X^\circ}
\newcommand{\Xtwo}{\X^{\le 2}}

\newcommand{\edge}{e}
\newcommand{\Xe}[1][]{{\X_{\edge_{#1}}}}
\newcommand{\Xoe}[1][]{{\Xo_{\edge_{#1}}}}
\newcommand{\bdXe}[1][]{\partial^{#1}\mathsf{X}_\edge}

\newcommand{\bouqe}[1][]{{\Theta_{\edge_{#1}}}}

\newcommand{\ObsTri}[1]{\Theta_{#1}}
\newcommand{\STri}[1]{Obs(#1)}

\newcommand{\F}{\mathsf{F}}

\newcommand{\G}{\mathsf{G}}

\newcommand{\dc}[2][]{\Sigma_{#1}(#2)}
\newcommand{\dcX}[1][]{\dc[#1]{\X}} 
\newcommand{\dcP}[2][]{\dc[#1]{\PP^{#2}}} 
\newcommand{\ornt}{\mathrm{or}}
\newcommand{\ordel}{\ornt_\Delta}

\newcommand{\idel}{i_{\Delta}}
\newcommand{\jdel}{j_{\Delta}}
\newcommand{\jdelf}{j_{\Delta*}}
\newcommand{\idelf}{i_{\Delta*}}
\newcommand{\idelb}{i_{\Delta}^*}

\newcommand{\Bdel}{B_\Delta}
\newcommand{\Be}{B_\edge}

\newcommand{\U}{\mathsf{U}}
\newcommand{\Utwo}{\U^{\le 2}}
\newcommand{\Us}{\U^{\mathrm{s}}}
\newcommand{\Ur}{\U^{\mathrm{r}}}
\newcommand{\Uo}{\U^\circ}

\newcommand{\Udel}{\mathsf{U}_\Delta}
\newcommand{\Udeltwo}{\Udel^{\le 2}}

\newcommand{\V}{\mathsf{V}}

\newcommand{\tVe}{\widetilde{\mathsf{V}}_\edge}
\newcommand{\tV}{\widetilde{\mathsf{V}}}
\newcommand{\tVo}{\widetilde{\mathsf{V}}^\circ}
\newcommand{\Vo}{\mathsf{V}^\circ}

\newcommand{\Y}{\mathsf{Y}}
\newcommand{\Z}{\mathsf{Z}}
\newcommand{\bdZ}[1][]{\partial^{#1}\Z}
\newcommand{\bZ}[1][]{\overline{\Z}}
\newcommand{\Zo}{\mathsf{Z}^\circ}
\newcommand{\bL}{\mathsf{L}}
\newcommand{\W}{\mathsf{W}}
\newcommand{\We}{\mathsf{W}_\edge}

\newcommand{\bS}{\mathsf{S}}
\newcommand{\B}{\mathsf{B}}

\newcommand{\Se}{\mathsf{S}_\edge}

\newcommand{\Semu}{\mathsf{S}_{\edge,m}}

\newcommand{\Kleaf}[2][1]{\Sigma_{#1}(#2)}
\newcommand{\Smth}[1]{\Sigma^{\mathrm{sm}}(#1)}

\newcommand{\Def}[2][]{\mathsf{Def}_{#1}(#2)}

\newcommand{\tb}[2][]{\mathsf{T}_{#1}#2}
\newcommand{\ctb}[1]{\mathsf{T}^*#1}
\newcommand{\nb}[1]{\mathsf{N}#1}
\newcommand{\nbdel}{\mathsf{N}_{\Delta}}
\newcommand{\nbdelo}{\mathsf{N}_{\Delta}^\circ}
\newcommand{\bdnbdel}{\partial\nbdel}
\newcommand{\logctb}[1]{\ctb{#1}(\log\partial#1)}

\newcommand{\I}{\mathscr{I}}

\DeclareMathOperator{\coker}{coker}
\DeclareMathOperator{\rank}{rank}
\DeclareMathOperator{\img}{image}
\DeclareMathOperator{\res}{Res}
\DeclareMathOperator{\Sym}{Sym}
\DeclareMathOperator{\gr}{Gr}
\newcommand{\id}{\mathrm{id}}

\newcommand{\pis}{\pi^\sharp}
\newcommand{\pidel}{\pi_\Delta}
\newcommand{\pisdel}{\pis_\Delta}

\newcommand{\cT}[1]{\mathcal{T}_{#1}} 
\newcommand{\cN}[1]{\mathcal{N}_{#1}} 
\newcommand{\cNdel}{\cN{\Delta}}
\newcommand{\coN}[1]{\mathcal{N}^\vee_{#1}} 
\newcommand{\coNdel}{\coN{\Delta}}

\newcommand{\cTlog}[1]{\mathcal{T}_{#1}(-\log \partial #1)} 
\newcommand{\Tpol}[2][\bullet]{\wedge^{#1} \cT{#2}} 
\newcommand{\der}[2][\bullet]{\mathscr{X}^{#1}_{#2}} 
\newcommand{\derlog}[2][\bullet]{\der[#1]{#2}(-\log\partial #2)} 
\newcommand{\forms}[2][\bullet]{\Omega^{#1}_{#2}} 
\newcommand{\logforms}[2][\bullet]{\Omega^{#1}_{#2}(\log \partial#2)} 

\newcommand{\can}[1]{\mathcal{K}_#1} 
\newcommand{\acan}[1]{\mathcal{K}^\vee_#1} 

\newcommand{\cO}[1]{\mathcal{O}_{#1}} 
\newcommand{\cI}{\mathcal{I}} 
\newcommand{\m}{\mathfrak{m}} 

\newcommand{\sD}[2][]{\mathcal{D}^{#1}_{#2}} 
\newcommand{\cMpi}[1][\bullet]{\mathcal{M}^{#1}_\pi} 
\newcommand{\sQ}{\mathcal{Q}}
\newcommand{\cE}{\mathcal{E}}
\newcommand{\sL}{\mathcal{L}}

\newcommand{\sLdel}{\mathcal{L}_\Delta}
\newcommand{\tsLdel}{\widetilde{\mathcal{L}}_\Delta}
\newcommand{\sLdelnr}{\mathcal{L}_\Delta^{\mathrm{nr}}}

\newcommand{\sLe}[1][]{\mathcal{L}_{\edge_{#1}}}

\newcommand{\sLenr}{\mathcal{L}_\edge^{\mathrm{nr}}}

\newcommand{\grpol}[2][\bullet]{\mathcal{C}^{#1}_{\pi,#2}}
\newcommand{\grdel}{\grpol{\Delta}}

\newcommand{\coH}[2][\bullet]{\mathsf{H}^{#1}(#2)}
\newcommand{\Hlgy}[2][\bullet]{\mathsf{H}_{#1}(#2)}
\newcommand{\scoH}[2][\bullet]{\mathcal{H}^{#1}(#2)}
\newcommand{\Hpi}[2][\bullet]{\mathsf{H}^{#1}_\pi(#2)} 
\newcommand{\sect}[1]{\mathsf{\Gamma}\rbrac{#1}}

\newcommand{\HomZ}[1]{\mathsf{Hom}_{\ZZ}(#1)}

\newcommand{\dd}{\mathrm{d}}
\newcommand{\dlog}[1]{\frac{\mathrm{d}#1}{#1}}
\newcommand{\dpi}[1][]{\dd_{\pi_{#1}}}
\newcommand{\lie}[1]{\mathscr{L}_{#1}}
\newcommand{\hook}[1]{\iota_{#1}}
\newcommand{\cvf}[1]{\partial_{#1}}
\newcommand{\logcvf}[1]{#1\partial_{#1}}

\newcommand{\dcat}[2][]{\mathsf{D}_{#1}(#2)}

\newcommand{\fibprod}{\mathop{\times}}

\begin{document}

\title{\papertitle}
\author{Mykola Matviichuk\thanks{McGill University, \email{mykola.matviichuk@mcgill.ca}} \and Brent Pym\thanks{McGill University, \email{brent.pym@mcgill.ca}} \and Travis Schedler\thanks{Imperial College London, \email{t.schedler@imperial.ac.uk}}}
\maketitle
\vspace{-0.5cm}
\abstract{
  We establish a local model for the moduli space of holomorphic symplectic structures with logarithmic poles, near the locus of structures whose polar divisor is normal crossings.  In contrast to the case without poles, the moduli space is singular: when the cohomology class of a symplectic structure satisfies certain linear equations with integer coefficients, its polar divisor can be partially smoothed, yielding adjacent irreducible components of the moduli space that correspond to possibly non-normal crossings structures.  These components are indexed by combinatorial data we call smoothing diagrams, and amenable to algorithmic classification.     Applying the theory to four-dimensional projective space, we obtain a total of 40 irreducible components of the moduli space, most of which are new.  Our main technique is a detailed analysis of the relevant deformation complex (the Poisson cohomology) as an object of the constructible derived category.   }

{\scriptsize
\setcounter{tocdepth}{2} 
\tableofcontents 
}

\section{Introduction}
\vspace{-0.15cm}
\subsection{Overview}
\vspace{-0.15cm}
Log symplectic manifolds are generalizations of holomorphic symplectic manifolds $(\X,\omega)$, where the symplectic two-form $\omega$ is allowed to have logarithmic poles on an anticanonical divisor $\bdX$. The open complement $\Xo := \X\setminus \bdX$ is then symplectic, and the divisor $\bdX$ is foliated by symplectic leaves of smaller dimension. Natural examples include Hilbert schemes of log Calabi--Yau surfaces~\cite{Bottacin1998,Goto2002}, moduli spaces of monopoles~\cite{Atiyah1988,Goto2002}, certain moduli spaces of bundles over elliptic curves~\cite{Feigin1989,Feigin1998,Gualtieri2013}, irreducible components of good degenerations of compact hyperk\"ahler manifolds~\cite{Harder2020}, and even-dimensional toric varieties (together with any torus-invariant symplectic structure on the open torus).

In the case of compact hyperk\"ahler manifolds, where $\bdX$ is empty, celebrated results such as the Bogomolov unobstructedness theorem~\cite{Bogomolov1978} and the local~\cite{Beauville1983} and global~\cite{Huybrechts2012,Markman2011,Verbitsky2013,Verbitsky2020} Torelli theorems give a great deal of information about the classification.  In particular, the moduli space is smooth, and locally parameterized by the second cohomology of $\X$.  Similar results hold for varieties $\X$ with symplectic singularities~\cite{Namikawa2011}.

For log symplectic manifolds with nonempty polar divisor, analogous unobstructedness results were recently established in the holomorphic~\cite{Ran2017, Ran2020,Ran2020a}, $C^\infty$~\cite{Marcut2014} and generalized complex~\cite{Cavalcanti2018, Cavalcanti2020,Goto2015} contexts, under some assumptions on the smoothness of $\bdX$ and the local behaviour of $\omega$.  These results indicate  that in the moduli space of all compact log symplectic manifolds,  there exist certain irreducible components parameterizing log symplectic structures whose divisor $\bdX$ is normal crossings, and these components are locally isomorphic to a neighbourhood of the origin in the cohomology group $\coH[2]{\Xo;\CC}$ of the open symplectic part. 

In light of Hironaka's desingularization theorem, it is natural to try to understand log symplectic manifolds with more complicated singularities by desingularizing them to obtain a normal crossings divisor, and then appealing to the aforementioned results.  Unfortunately, the desingularization typically destroys the nondegeneracy of the two-form, so a new approach is needed.

In this paper, we instead use normal crossings structure to \emph{approximate} more general log symplectic manifolds. The key new idea is that if $\bdX$ has normal crossings and we tune the periods $\int_{\Hlgy[2]{\Xo;\ZZ}}\omega$ appropriately, then new deformations of $(\X,\omega)$ appear for which $\bdX$ has more complicated singularities; see e.g.~\autoref{tab:3cpt-smoothings}.  We show that these deformations give new irreducible components of the moduli space that intersect the normal crossings locus along explicit subspaces of $\coH[2]{\Xo;\CC}$, as in \autoref{fig:3cpt-line-arr}, and that they are enumerated by combinatorial data we call ``smoothing diagrams'', as in \autoref{fig:D5}, which can be classified via elementary algorithmic linear algebra.  This yields an explicit model for the formal neighbourhood of the normal crossings structures in the full moduli space, and as a byproduct, an optimal unobstructedness result that unifies and extends the previously known results. 
  We illustrate the effectiveness of our results with the concrete example of $\X = \PP^4$, producing an order-of-magnitude increase in the number of known irreducible components of the moduli space (enumerated diagrammatically in \autoref{tab:confs_P4}).

\subsection{Summary of the results}

In more detail, let $(\X,\omega)$ be a log symplectic manifold with normal crossings polar divisor $\bdX$.  We review and develop the basic structure theory of such manifolds in \autoref{sec:log-symp}, and in particular we establish the following local normal form in \autoref{prop:stable-form}: if $p \in \X$ is a point where exactly $k$ irreducible components of the degeneracy divisor meet, then up to stabilization (i.e., taking products with symplectic balls), the log symplectic form can be written in suitable coordinates $y_1,\ldots,y_k,z_1,\ldots,z_k$ as
\[
\omega = \sum_{1 \le i<j \le k} B_{ij} \dlog{y_i}\wedge\dlog{y_j} + \sum_{1 \le i \le k} \dlog{y_i} \wedge \dd z_i.
\]
Here $(B_{ij})_{ij} \in \CC^{k\times k}$ is an arbitrary skew-symmetric matrix of constants, called the \defn{biresidues} of $\omega$.  Up to an overall factor of $(2\pi \mathrm{i})^2$, these numbers are simply the periods of the symplectic form along small two-tori encircling components of the divisor. In particular, they depend only on the cohomology class $[\omega] \in \coH[2]{\Xo;\CC}$, and a choice of ordering of the irreducible components intersecting at $p$.  (In the body of the paper, we adopt a more natural system for tracking such orderings using orientations of simplices in the dual CW complex of $\bdX$.)

The deformations  of $(\X,\omega)$ can equivalently be viewed as deformations of the Poisson bivector $\pi = \omega^{-1}$, or more generally, as deformations of $(\X,\pi)$ as a generalized complex manifold in the sense of Gualtieri~\cite{Gualtieri2011} and Hitchin~\cite{Hitchin2003}.  (The two notions are equivalent when $h^1(\cO{\X}) = h^2(\cO{\X})=0$.)    We denote by $\Def{\X}$ the corresponding space of formal deformations.  This space is naturally a 2-stack, defined as the Maurer--Cartan functor for a suitable homotopy Gerstenhaber ($G_\infty$) algebra, and we treat it as such; see \autoref{sec:L-infinity} for the precise definition.  For simplicity,  we will consider only the underlying miniversal deformation space in this introduction.

The space of deformations of the holomorphic symplectic manifold $(\Xo,\omega|_{\Xo})$ is canonically identified with a formal neighbourhood of the cohomology class of $\omega$ in the vector space $\coH[2]{\Xo;\CC}$.  There is a natural restriction map $\Def{\X} \to \coH[2]{\Xo;\CC}$ induced by the corresponding restriction of polyvector fields, which has a canonical section corresponding to deformations of $(\X,\omega)$ for which the polar divisor deforms locally trivially.  We will describe  $\Def{\X}$ as a  ``constructible fibration'' over $\coH[2]{\Xo;\CC}$---the formal completion of a union of varieties that are fibre bundles over their images in $\coH[2]{\Xo;\CC}$.  The fibres parameterize deformations in which some strata of the normal crossings divisor are partially smoothed.

The simplest example of such a smoothing occurs for the manifold $\X=\CC^2$, equipped with log symplectic form
\[
\omega_0 := B \dlog{y_1}\wedge\dlog{y_2} 
\]
for some $B \in \CC^*$.  The polar divisor is the union of the coordinate axes $\{y_1y_2=0\}$, and the origin is its only singularity (a node). The miniversal deformation of $\omega$ is given by the family of 2-forms
\begin{align}
\omega(\tilde B,\epsilon)  := \frac{\tilde B \dd y_1 \wedge \dd y_2}{y_1y_2-\epsilon} \label{eq:C2-universal}
\end{align}
where $\tilde B, \epsilon \in \CC$, with $\tilde B \ne 0$.   When $\epsilon = 0$, the polar divisor is unchanged and the deformation is completely determined by the biresidue $\tilde B$, or equivalently the cohomology class of the form on $(\CC^*)^2$, but when $\epsilon$ is nonzero, the deformed polar divisor $y_1y_2=\epsilon$ is smooth.  Correspondingly $\Def{\CC^2} \to \coH[2]{(\CC^*)^2;\CC}$ is a line bundle having $\epsilon$ as a global coordinate on the fibres.  Note that while this line bundle is evidently trivial at the level of miniversal deformation spaces, to describe the quotient stack, we must account for the symmetries, which are given by the two-torus that acts by dilations of the coordinates $y_1,y_2$.  This torus acts  nontrivially on the fibres of $\Def{\CC^2}$, so that deformations with different values of $\epsilon \neq 0$ are equivalent.  Consequently, the bundle is nontrivial at the level of stacks.

This example motivates the following definition for an arbitrary normal crossings log symplectic manifold $(\X,\omega)$:
\begin{definition}[see \autoref{def:smoothable}]
  A codimension-two normal crossings stratum $\Y \subset \X$ is (infinitesimally) \defn{smoothable} if there exists a first-order deformation of $(\X,\omega)$ that, near any point of $\Y$, takes the form \eqref{eq:C2-universal} in the directions transverse to $\Y$.
\end{definition}
Our results imply that in most cases (including all cases where the normal crossings are simple), it is actually possible to smooth every smoothable stratum to all orders in the deformation parameter, leaving all other smoothable strata intact.  Note that it is important in this definition that the stratum be smoothed at first order: we will explain below that, under special conditions, some non-smoothable strata can be smoothed at higher order. 

We prove that the smoothability of a stratum is a purely topological property.  More precisely, in \autoref{sec:arrangement}, we define, for every codimension-two stratum $\Y$, a canonical  submanifold $\bS_\Y \subset \coH[2]{\Xo;\CC}$, such that $\Y$ is smoothable for $(\X,\omega)$ if and only if $[\omega] \in \bS_\Y$.  This submanifold $\bS_\Y$ is an open set in a countable union of affine linear subspaces of $\coH[2]{\Xo;\CC}$ defined over $\ZZ$, and it depends only on the topology of the pair $(\X,\bdX)$.  The condition $[\omega]\in \bS_\Y$ is tested  concretely by taking periods of $\omega$ around tori in $\Xo$ lifting loops in $\Y$, generalizing the biresidues $B_{ij}$ above; see \autoref{prop:smoothable-bires}.  

 The mutual intersections of the submanifolds $\bS_\Y$ give a canonical stratification of $\coH[2]{\Xo;\CC}$ by locally closed submanifolds corresponding to normal crossings log symplectic structures with a fixed collection of smoothable strata.  These strata are naturally indexed by combinatorial data we call \defn{smoothing diagrams}; see \autoref{sec:arrangement} for the definition and \autoref{fig:D5} below for some examples.  Our main result, which is established in stages throughout Sections \ref{sec:smoothable} and \ref{sec:L-infinity}, is the following Torelli-style description of the deformation space.  Since we are working with formal deformations, $\coH[2]{\Xo;\CC}$ and its stratification $\{\bS_\Y\}$ should be completed at $[\omega]$, which we suppress for simplicity of notation.

\begin{theorem}
\label{thm:torelli-bundle}
If the cohomology class $[\omega]\in\coH[2]{\Xo;\CC}$ satisfies an explicit open ``nonresonance'' condition, then $\Def{\X} \to \coH[2]{\Xo;\CC}$ forms a constructible vector bundle over the stratification $\{\bS_\Y\}$, given by a fibre product of line bundles $\mathsf{L}_\Y \to \bS_\Y$, one for each smoothable stratum $\Y$.
\end{theorem}
Note that the nonresonant condition is often vacuous, such as for $\X=\mathbb{P}^{2n}$, in which case $\Def{\X}$ forms a constructible vector bundle over $\coH[2]{\Xo;\CC}$ with no caveats.
In general, however, for finely tuned periods, the fibres can become singular:
\begin{theorem}\label{thm:torelli-fibre}
  Let $\F$ be the fibre of $\Def{\X}$ over $[\omega] \in \coH[2]{\Xo;\CC}$.  Then the Zariski tangent space $\tb{\F}$ splits canonically as a direct product of one-dimensional subspaces indexed by the smoothable strata of $(\X,\omega)$, and $\F$ embeds in $\tb{\F}$ as the solution set of a system of equations of the form $f=g$, where $f,g$ are monomial functions.  
\end{theorem}
Put differently, the relative tangent spaces of $\Def{\X} \to \coH[2]{\Xo;\CC}$ form the constructible bundle mentioned in \autoref{thm:torelli-bundle}, with $\Def{\X}$ itself embedding as the solutions of the monomial equations as in  \autoref{thm:torelli-fibre}, which are vacuous over the nonresonant locus.  Although the fibres can, in general, be singular and even reducible, it turns out that as long as $\X$ has only finitely many codimension two strata, $\Def{\X}$ always has a stratification  by principal torus bundles over subvarieties of $\coH[2]{\Xo;\CC}$, which are defined by polynomial relations between periods and their logarithms; see \autoref{sec:irred-decomp}.  In this way, even in the resonant case, we reduce the determination of the deformation space to explicit calculations involving periods of the log symplectic form, which are largely (but not completely) linear-algebraic in nature.

While this model for the deformation space is quite explicit and computable, we would like to emphasize that the corresponding Poisson deformations may be highly nontrivial, with interesting singularities in codimension three and higher.  The simplest nontrivial case is the deformations of a log symplectic structure near a point on a codimension-four symplectic leaf where three components of the degeneracy divisor meet with normal crossings; we use it as a ``running example'' throughout the paper, and it eventually plays a key role in the proof of Theorems \ref{thm:torelli-bundle} and \ref{thm:torelli-fibre}. In \autoref{sec:runex-def}, we give a complete description of its universal deformation, which is obtained by adding certain monomial terms to the bivector when the biresidues are suitably tuned.  This yields the model for the moduli space described in \autoref{fig:3cpt-line-arr}, and the list of possible singularities of the deformed divisor in \autoref{tab:3cpt-smoothings}, consisting of a transverse $A_1$ curve singularity (node), the series of stem singularities of type $T_{p,q,r}$ with $p+q+r=\infty$ from \cite[Section 4.5]{Arnold1993}, a bifurcation into a pair of Whitney umbrellas joined along their handles, and in the most exceptional case, the one-parameter families of simple elliptic singularities $\tilde E_6,\tilde E_7,\tilde E_8$ of \cite{Saito1974}, whose associated Poisson structures were studied by the second author in \cite{Pym2017}.

\begin{figure}
\begin{center}
\includegraphics[scale=1]{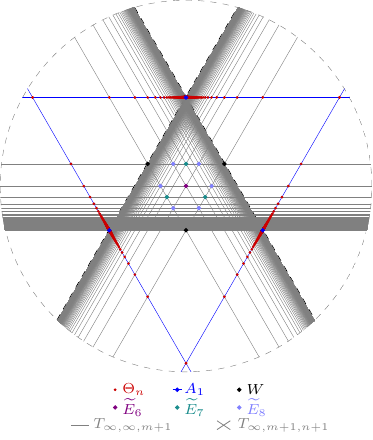}
\end{center}
\caption{The moduli space of log symplectic germs on $\mathbb{C}^4$ with degeneracy divisor $xyz=0$, considered up to isomorphism and rescaling, is the projective plane $\PP\coH[2]{(\CC^*)^3} \cong \PP^2$, modulo the natural action of $S_3$.  It is illustrated above using the biresidues $b_1,b_2,b_3$ as homogeneous coordinates, so that the central point is $[1:1:1]$, and the three dashed lines indicate the points where some $b_i=0$.  This copy of $\PP^2/S^3$ give an irreducible component of the full moduli space, and its formal neighbourhood is a countable union of vector bundles, one for each stratum of the arrangement of solid lines.  At a point where $k$ solid lines meet, the fibre  has dimension $k$, and its generic element is a log symplectic structure whose singularity type is indicated by the legend and described in \autoref{tab:3cpt-smoothings}; unless $k=3$ such structures are all isomorphic.  }
\label{fig:3cpt-line-arr}
\end{figure}

The series $\ObsTri{n}$  in \autoref{tab:3cpt-smoothings} plays a particularly important role in the deformation theory.  It exhibits the following ``resonance'' phenomenon: if both smoothable strata are smoothed simultaneously at first order in deformation theory, then the remaining  stratum is automatically smoothed at order $n+2$, resulting in a smooth divisor.  For a general log symplectic manifold $(\X,\omega)$, the presence of such local resonances along codimension-three strata enforce global consistency conditions relating the smoothings of different pairs of smoothable strata, which can prevent the simultaneous independent smoothing of all smoothable strata. This is the source of the equations defining the fibres in \autoref{thm:torelli-fibre}.

\begin{table}
\caption{Classification of log symplectic smoothings near a codimension-three stratum}
\label{tab:3cpt-smoothings}\def\arraystretch{1.4}
\centerline{
\begin{tabular}{c|c|c}
Name & Equation for divisor  & Singularity \\  \hline \hline
generic   & $xyz$ & 3d normal crossings \\
\hline 
$A_1$ &  $xyz+x$ & 2d normal crossings \\
\hline 
$T_{\infty,\infty,m+1}$  & $xyz+x^{m+1}$  & stem of degree two\\
\hline 
$\ObsTri{n}$   & $xyz+x + y^{n+1}$ & smooth  \\
\hline 
$W $   & $xyz+x^2 + y^2$ & two Whitney umbrellas \\
\hline 
 $T_{\infty,m+1,n+1}$    & $xyz+x^{m+1} + y^{n+1}$ & stem of degree one \\ \hline
 $\tilde E_{6}$    & $ xyz + \lambda(x^{3} + y^3 + z^{3}) $ & 
 \\
 $\tilde E_{7}$  & $ xyz + \lambda(x^{2} + y^4 + z^{4}) $  & elliptic with parameter $\lambda \in \CC$\\
$\tilde E_{8}$  & $ xyz + \lambda(x^{2} + y^{3} + z^6) $ & 
\end{tabular}
}
\end{table}
As a nontrivial example of how \autoref{thm:torelli-bundle} can be applied, we consider the case $\X = \PP^{2n}$ equipped with a torus-invariant log symplectic structure; it was shown in \cite{Lima2014} (see also \cite[Section 5]{Pym2018a}) that these are the only simple normal crossings log symplectic structures on smooth Fano varieties $\X$ with $\dim \X = 2n \ge 4$ and $b_2(\X) = 1$.  In this case, \autoref{thm:torelli-bundle} reduces the description of $\Def{\X}$ to some elementary linear algebra of skew-symmetric matrices,  giving a new combinatorial machine for producing many new examples of log symplectic structures. While the relevant equations for a given example can easily be written down and solved by hand, the classification of all possibilities is somewhat tedious and is best approached algorithmically.  We wrote a simple computer script to enumerate the possibilities, which yields the following result for $\PP^4$:
\begin{theorem}[see \autoref{sec:P2n-diagrams} and \autoref{ex:P2n-moduli}]\label{thm:P4}
The moduli space of Poisson structures on $\PP^4$ has 40 irreducible components that intersect the locus of simple normal crossings log symplectic structures.  These components are in bijection with the isomorphism classes of smoothing diagrams shown in \autoref{tab:confs_P4}.
\end{theorem}

\begin{figure}
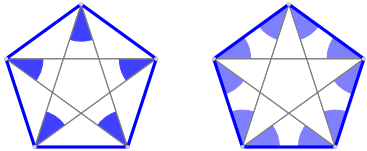
\begin{center}
\sagepentagon[0.7]{\smoothedge{v3}{v4}{v1}{v1} 
\smoothedge{v2}{v3}{v0}{v0} 
\smoothedge{v1}{v2}{v4}{v4} 
\smoothedge{v0}{v1}{v3}{v3} 
\smoothedge{v4}{v0}{v2}{v2} 
}
\sagepentagon[0.7]{\smoothedge{v3}{v4}{v0}{v2} 
\smoothedge{v2}{v3}{v1}{v4} 
\smoothedge{v1}{v2}{v3}{v0} 
\smoothedge{v0}{v1}{v2}{v4} 
\smoothedge{v4}{v0}{v1}{v3} 
}
\end{center}
\caption{There are two nontrivial smoothing diagrams for $\X = \PP^4$ whose automorphism group is the dihedral group $D_5$.  They correspond to irreducible components of the moduli space of Poisson structures whose generic members are Feigin--Odesskii's elliptic Poisson structures~\cite{Feigin1989,Feigin1998} of type $q_{5,1},q_{5,2}$.}\label{fig:D5}
\end{figure}

We remark that Poisson structures on $\PP^n$ for $n\le 3$ have been classified: it is straightforward to see that the moduli space has one irreducible component for $n\le2$ and much harder to see that there are exactly six irreducible components when $n=3$, as shown in \cite{Cerveau1996,Loray2013}.  For $n \ge 4$, only a handful of components were previously known, and our classification includes all examples of log symplectic structures on $\PP^4$ of which we are currently aware.  For instance, the two most symmetric smoothing diagrams (shown in \autoref{fig:D5}) correspond to Poisson structures on moduli spaces of bundles over elliptic curves discovered by Feigin--Odesskii~\cite{Feigin1989,Feigin1998}.  Some other components correspond to structures discovered by Okumura~\cite{Okumura2020} by blowing up linear subspaces. However, the vast majority of the components correspond to structures that appear to be new.  It would be interesting to know if every log symplectic structure on $\PP^{4}$ lies in one of these 40 components of the moduli space, i.e.~admits a toric degeneration, and also to apply this technique to other toric Fano varieties.

\subsection{Method of proof}

Our main technical tool in the paper is the deformation complex of the Poisson manifold $(\X,\pi)$, given by the sheaf $\der{\X}$ of holomorphic polyvector fields, equipped with the Schouten bracket $[-,-]$ and the Poisson differential $\dpi = [\pi,-]$.  For general Poisson manifolds, the calculation of the resulting ``Poisson cohomology'' is difficult.  For instance, even in the relatively simple case of normal crossings log symplectic manifolds considered here, the stalk cohomology is infinite-dimensional for certain non-generic values of the biresidues $B_{ij}$, which complicates the calculation of the global hypercohomology.  

For this reason, we employ at various stages a natural refinement of the log symplectic condition, called ``holonomicity'', which was introduced by the second- and third-named authors in~\cite{Pym2018}.  This condition is phrased using D-modules and is essentially equivalent to the finite-dimensionality of the stalk cohomology; see \autoref{sec:holonomic}.  In the normal crossings case at hand, we show that holonomicity, and in turn, the Poisson cohomology, is controlled by a certain collection of symplectic leaves of the Poisson manifold, which we call ``characteristic'' (those to which Weinstein's modular vector field~\cite{Weinstein1997} is tangent).  Moreover, we reduce this abstract condition to a simple integer-linear condition on the biresidues.  (We expect that this picture extends beyond the normal crossings case; see \autoref{conj:char-holonomic}.)

More precisely, there is a canonical increasing filtration $\W_\bullet\der{\X}$ dual to Deligne's familiar weight filtration on logarithmic differential forms in mixed Hodge theory~\cite{Deligne1971a} and related to a filtration on meromorphic forms used for similar purposes in \cite{Ran2020}.  We prove the following result which characterizes the holonomic structures, and gives in that case a complete description of the associated graded $\gr \der{\X}$ as an object in the constructible derived category $\dcat{\X}$.  This can be viewed as information about the full ``derived'' moduli space of holonomic Poisson manifolds, extending the classical deformation functor $\Def{\X}$.
\begin{theorem}[see \autoref{thm:hol-complex}]\label{thm:HPhol}
The following are equivalent:
\begin{enumerate}
\item $(\X,\pi)$ is holonomic in the sense of \cite{Pym2018}.
\item The stalk cohomology of $(\der{\X},\dpi)$ is finite-dimensional
\item Every point in $\X$ has a neighbourhood intersecting only finitely many characteristic symplectic leaves.
\item For every stratum of odd codimension, the vector $(1,\ldots,1)$ is linearly independent from the rows of the corresponding biresidue matrix.
\end{enumerate}
In this case, the associated graded $\gr \der{\X}$ is a direct sum of extensions of rank-one local systems supported on characteristic symplectic leaves, shifted in degree by their codimensions.  Furthermore, $\der{\X}\cong\gr\der{\X}$ in  $\dcat{\X}$ when certain nonresonance conditions on the biresidues are satisfied.
\end{theorem}

The local systems constituting $\gr{\der{\X}}$ have explicit descriptions; in particular their monodromies are given explicitly by periods of the log symplectic form (\autoref{cor:no-monodromy}).    The hypercohomology of $\gr{\der{\X}}$ is then readily computed in terms cohomology groups of strata relative to certain ``resonant'' boundary components, with coefficients in the corresponding local systems (\autoref{prop:res-cohlgy}).  

In particular, every normal crossings log symplectic manifold is holonomic away from a codimension three subset, and $\gr_2\der{\X}$ consists of  local systems on codimension-two strata that parameterize their local smoothings. A stratum is then globally smoothable if and only if its corresponding zeroth relative cohomology group is nontrivial.  The proof of Theorems \ref{thm:torelli-bundle} and \ref{thm:torelli-fibre} is achieved by a local-to-global argument, in which we use normal forms established in \autoref{sec:log-symp} to establish the local formality of $L_\infty$ structure on the deformation complex, giving an explicit description of the local deformation functors in terms of the corresponding Maurer--Cartan 2-groupoids.  We then glue the local Maurer--Cartan stacks over a judiciously chosen open cover to obtain the desired description of $\Def{\X}$.

\paragraph{Acknowledgements:} We thank Marco Gualtieri, Julien Keller, Ziv Ran and Chris Rogers for conversations and correspondence.  B.P.~gratefully acknowledges the support of Seggie Brown bequest at the University of Edinburgh; a faculty startup grant at McGill University; and the Natural Sciences and Engineering Research Council of Canada (NSERC), through Discovery Grant RGPIN-2020-05191.  B.P.~and T.S.~also thank the Hausdorff Institute for Mathematics for their support and hospitality during the Junior Trimester Program on ``Symplectic Geometry and Representation Theory'', where preliminary discussions on the present research took place. T.S.~would also like to thank the Max Planck Institute for Mathematics for its generous support during work on this project.

\section{Calculus on normal crossings divisors}
\label{sec:nc-divisors}
Throughout this paper, $\X$ will denote a connected complex manifold. We denote by $\cO{\X}$, $\cT{\X}$, $\forms{\X}$ and $\der{\X} = \wedge^\bullet \cT{\X}$ the sheaves of  holomorphic functions, vector fields, differential forms and polyvector fields, respectively.

We shall fix once and for all a normal crossings divisor on $\X$, which we denote by $\bdX \subset \X$.  Note that $\bdX$ is not a topological boundary of $\X$ in any sense; rather it is the boundary of the open dense set
\[
\Xo := \X \setminus \bdX
\]
in $\X$.  However, the notation will be convenient in what follows. 

For the sake of exposition, we shall assume throughout that $\bdX$ is simple, meaning that all its irreducible components are smooth, and the intersection of any collection of distinct irreducible components is transverse.  

\begin{remark} Any normal crossings divisor is locally simple.  Since our main results (including Theorems \ref{thm:hol-complex}, \ref{thm:Xtwo} and \ref{lem:codim3-balls-suffice}) are ultimately deduced from local properties, they hold essentially without modification in the non-simple case; the only slightly subtle points necessary for this generalization are addressed in Remarks \ref{rmk:non-simple-or} and \ref{rmk:non-simple-symplectic} below. 
\end{remark}

If $p \in \bdX$ is a point where exactly $k$ components meet, there exist an open neighbourhood $\U$ of $p$ and a system of analytic coordinates $(y_1,\ldots,y_k,z_1,\ldots,z_j)$ on $\U$ such that
\[
\bdX \cap \U = \{y_1\cdots y_k = 0\}.
\]
We refer to such coordinates as \defn{normal crossings coordinates}.  

In this section we recall some basic constructions associated to the pair $(\X,\bdX)$; with the possible exception of \autoref{sec:nc-polyvect}, this material is standard and is mainly for the purposes of establishing notations and examples that will be used throughout the paper.

\subsection{Strata and the dual complex}
\label{sec:dual-complex}

The normal crossings divisor $\bdX\subset \X$ gives rise to a filtration of $\X$ by closed subvarieties
\[
\X \supset \bdX \supset \fX{2} \supset \cdots
\]
where for $j \ge 0$, the subvariety $\fX{j} \subset \X$ is the locus of points where $\bdX$ has multiplicity $\ge j$, or equivalently, $\fX{j}$ is the union of all $j$-fold intersections of irreducible components of $\bdX$.  In particular, $\fX{j+1}$ is the singular locus of $\fX{j}$ for $j \ge 1$.  By transversality, each irreducible component of $\fX{j}$ is a smooth closed submanifold of codimension $j$.  Indeed, in local normal crossings coordinates as above, the subvariety $\fX{j}$ is the locus where at least $j$ of the coordinates $y_1,\ldots,y_k$ vanish, giving the union of ${ k \choose j}$ linear subspaces of codimension $j$.

Recall that the \defn{dual complex} of the pair $(\X,\bdX)$ is the CW complex $\dcX$,
whose $k$-simplices correspond to the irreducible components of $\bdX[k+1]$, and whose inclusions of faces are dual to the inclusions of subvarieties; see, e.g.~\cite{Danilov1975}.  We shall adopt the following notations. 

\begin{definition}
We denote by $\dc[k]{\X}$ the set of $(k-1)$-dimensional simplices in $\dcX$; they are in bijection with irreducible components of $\bdX[k]$.  If $\Delta \in \dc[k]{\X}$ is a $(k-1)$-simplex, we denote by
\[
\idel : \Xdel \to \X
\]
the inclusion of the irreducible component $\Xdel \subset \fX{k}$ indexed by $\Delta$. This component contains an induced normal crossings divisor
\[
\bdXdel := \Xdel \cap \fX{k+1}
\]
whose complement is the locally closed submanifold
\[
\Xodel := \Xdel \setminus \bdXdel.
\]
We refer to $\Xodel$ as a \defn{stratum of codimension }$k$.  We denote by
\[
\Delta_0 :=  \Delta \cap \dcX[0] 
\]
the set of vertices of $\Delta$; they correspond to the irreducible components of $\bdX$ that intersect along $\Xodel$.
\end{definition}

We now give several examples; the first of these forms the basis of a ``running example'' that will be developed throughout the paper.
\begin{example}[\runex]\label{ex:3cpt-dual-complex}
Let $\X = \CC^4$ with coordinates $(y_1,y_2,y_3,z)$.  Consider the divisor $\bdX \subset \X$ given by the union of three coordinate hyperplanes $y_i = 0$.  Any pair of these hyperplanes intersects in a unique codimension-two stratum, and intersects the remaining hyperplane along the unique codimension-three stratum given by the $z$-line $y_1=y_2=y_3=0$.  The dual complex is therefore a standard two-simplex:
\[
\dcX = 
\vcenter{\hbox{ \begin{tikzpicture}[scale=1.2,decoration={
    markings,
    mark=at position 0.5 with {\arrow{>}}}
    ] 
    \draw[fill=gray] (-30:1) -- (90:1) -- (210:1) -- (-30:1);
\draw (-30:1) -- (90:1);
\draw (90:1) -- (210:1) ;
\draw (210:1) -- (-30:1);
\draw (90:1.2) node {$\scriptstyle(y_1)$};
\draw (-30:1.2) node {$\scriptstyle(y_2)$};
\draw (210:1.2) node {$\scriptstyle(y_3)$};
\draw (30:0.9) node {$\scriptstyle(y_1,y_2)$};
\draw (150:0.9) node {$\scriptstyle(y_3,y_1)$};
\draw (-90:0.7) node {$\scriptstyle(y_3,y_2)$};
\draw (0,0) node {$\scriptstyle(y_1,y_2,y_3)$};
\end{tikzpicture}}}
\]
where we have labelled each simplex $\Delta$ in $\dcX$ by the ideal of functions that define the corresponding submanifold $\Xdel$.  For instance, the edge $\Delta$ labelled $(y_1,y_2)$ corresponds to the codimension-two stratum
\[
\Xodel = \set{(0,0,y_3,z)}{y_3\neq 0}\cong \CC^*\times \CC.
\]
and the remaining codimension-two strata are obtained by cyclically permuting $y_1,y_2,y_3$. 
\end{example}

\begin{example}
Let $\X = \CC^n$ with the boundary divisor $\bdX$ given by the union of the $n$ coordinate hyperplanes.  The intersection of any $k$-tuple of components of $\bdX$ has exactly one irreducible component, and hence the dual complex $\dcX$ is an $(n-1)$-simplex. 
\end{example}

\begin{example}
Let $\X = \PP^n$ with its standard toric boundary divisor $\bdX$ given by the union of the $n+1$ coordinate hyperplanes.  Then the dual complex $\dcX$ is given by gluing $n+1$ distinct $(n-1)$-simplices along common codimension-one faces, giving the boundary of an $n$-simplex, which is homeomorphic to the $(n-1)$-sphere $S^{n-1}$. If $\Delta \in \dc[k]{\X}$ then the stratum $\Xodel$ is isomorphic to the torus $(\CC^*)^{n-k}$.  
\end{example}

Note that if $\Delta$ is a simplex, the normal sheaf $\cNdel$ of the corresponding stratum $\Xodel$ splits canonically as a direct sum of line bundles
\begin{align}
\cNdel \cong \bigoplus_{i \in \Delta_0} \cN{\Delta,i} \label{eq:nb-split}
\end{align}
corresponding to the normal bundles of the irreducible components of $\bdX$ that intersect along $\Xodel$; here, as above, $\Delta_0$ is the set of vertices of $\Delta$.  The total space $\nbdel := \mathsf{Tot}(\cNdel)$ of this locally free sheaf therefore contains a normal crossings divisor $\bdnbdel$ given by the union of the corank-one subbundles obtained from the splitting; this gives the normal cone of $\bdX$ over $\Xodel$.  The following ``rectification'' lemma is standard and will be useful in what follows. 

\begin{lemma}\label{lem:rectify}
Suppose that $\Xodel$ admits a holomorphic tubular neighbourhood in $\X$.  Then $\Xodel$ admits a holomorphic tubular neighbourhood in which $\bdX \subset \X$ is identified with its normal cone $\bdnbdel \subset \nbdel$. 
\end{lemma}

\begin{proof}
Choosing an arbitrary holomorphic tubular neighbourhood, we may assume without loss of generality that $\X$ is an open neighbourhood of the zero section in $\nbdel$, and that the linearization of this identification is the identity map on the normal bundle of $\Xodel$.  Thus the normal cone of $\bdX$ in $\nbdel$ is identified with $\bdnbdel$ as a subvariety in $\nbdel$.  

Fix some component $\Y$ of $\bdX$ that contains $\Xodel$, and let $\Z$ be the corresponding component of $\bdnbdel$.  Shrinking $\X$ if necessary, we may assume that  $\Y$ projects isomorphically onto its image in $\Z$ under the canonical splitting $\nbdel = \nb{\Y} \oplus \Z$. It may therefore be viewed as a graph of a map $\phi$ from an open set in $\Z$ to $\nb{\Y}$, which vanishes with multiplicity two along the zero section.  Define a new map $\psi : \X \to \nbdel$ by $\psi(y,z)=(y-\phi(z),z)$ with respect to the splitting. Then $\psi$ is tangent to the identity along the zero section, so after shrinking $\X$ we may assume that it is a biholomorphism onto its image, defining a new tubular neighbourhood embedding.  By construction, $\psi(\Y) = \Z$ and $\psi(\Z') =\Z'$ for any component $\Z'\subset \bdnbdel$ other than $\Z$.  Hence we may repeat this construction for each irreducible component in turn, arriving at the rectified tubular neighbourhood, as desired.
\end{proof}

\subsection{Logarithmic forms}
Recall from \cite{Deligne1971a,Saito1980} that a differential form $\omega$ on $\X$ is \defn{logarithmic with respect to $\bdX$} if the following conditions hold:
\begin{enumerate}
\item $\omega$ is holomorphic on $\Xo = \X\setminus \bdX$
\item $\omega$ and $\dd\omega$ have at worst first-order poles on $\bdX$
\end{enumerate}
We denote by $\logforms{\X} \cong \wedge^\bullet\logforms[1]{\X}$ the sheaf of logarithmic forms. In local normal crossings coordinates $(y_1,\ldots,y_k,z_1,\ldots,z_j)$ we have 
\[
\logforms{\X} \cong  \cO{\X}\abrac{\dlog{y_1},\ldots,\dlog{y_k},\dd z_1,\ldots,\dd z_j}
\]
Recall that this sheaf comes equipped with a canonical increasing filtration, called the \defn{weight filtration}, given by the subsheaves
\[
\forms{\X} = \W_0\logforms{\X} \subset \W_1\logforms{\X} \subset \cdots \subset \W_n \logforms{\X} = \logforms{\X}
\]
where $n=\dim \X$, defined by
\[
\W_j\logforms{\X} := \logforms[\le j]{\X} \wedge \forms{\X}. 
\]
Equipped with the wedge product and de Rham differential,  $\W_\bullet \logforms{\X}$ becomes a sheaf of filtered differential graded (dg) commutative algebras.  Recall that the hypercohomology $\coH{\logforms{\X},\dd}$ is canonically isomorphic to $\coH{\Xo;\CC}$, the cohomology of the open complement $\Xo=\X\setminus\bdX$.

For a $(k-1)$-simplex $\Delta  \in \dc[k]{\X}$, let  $\ornt_\Delta$ be the orientation of $\Delta$, i.e., the one-dimensional vector space of pairs of a complex number and an ordering of the irreducible components of $\bdX$ intersecting along $\Xodel$, modulo permuting the components and multiplying the complex number by the sign of the permutation.

Taking residues of logarithmic forms along the irreducible components containing $\Xdel$, one forms the classical Deligne--Poincar\'e \defn{residue map}:
\begin{align*}
\xymatrix{
\res : \W_k\logforms{\X}\ar[r] &  i_{\Delta*}\forms{\Xdel} \otimes \ornt_\Delta[-k]
}  
\end{align*}
which descends to an isomorphism of complexes
\begin{align}
\gr_k^\W \logforms{\X} \cong \bigoplus_{\Delta \in \dc[k]{\X}} i_{\Delta*}\forms{\Xdel} \otimes \ornt_\Delta[-k]. \label{eq:gr-logforms}
\end{align}
Here and throughout, the symbol $[n]$ for $n \in \ZZ$ is the usual degree-shift operator, defined so that $(\mathcal{C}^\bullet[n])^j = \mathcal{C}^{n+j}$ for any complex $\mathcal{C}^\bullet$.  Thus, up to a choice of orientation for all $k$-simplices, the right hand side of \eqref{eq:gr-logforms} is the direct sum of the ordinary de Rham complexes of the submanifolds $\Xdel$, shifted up in degree by their codimensions. 

\begin{remark}\label{rmk:non-simple-or}
If the normal crossings divisor $\bdX$ is not simple, then it may not be possible to find a globally consistent ordering of the components along $\Xodel$; in this case $\ordel$, rather than a vector space, forms a rank-one local system on $\Xodel$. 
\end{remark}

\subsection{Polyvector fields}
\label{sec:nc-polyvect}

Dual to the standard constructions for forms above is a structure on polyvector fields, which we now describe; the proofs of the lemmas in this section are straightforward and are left to the reader.

 We denote by $\cTlog{\X} \subset \cT{\X}$ the locally free sheaf of \defn{logarithmic vector fields}, i.e.~vector fields tangent to $\bdX$ and by
\[
\derlog{\X} := \wedge^\bullet \cTlog{\X} \subset \der{\X}
\]
the sheaf of \defn{logarithmic polyvector fields}.  In normal crossings coordinates $(y_1,\ldots,y_k,z_1,\ldots,z_j)$, we have
\[
\derlog{\X} = \cO{\X}\abrac{y_1\cvf{y_1},\ldots,y_k\cvf{y_k},\cvf{z_1},\ldots,\cvf{z_j}}.
\]
We remark that the inclusion $\derlog{\X} \subset \der{\X}$ is $\cO{\X}$-dual to the inclusion $\forms{\X} \subset \logforms{\X}$, and correspondingly we have a dual weight filtration on polyvectors:
\[
\derlog{\X} = \W_0 \der{\X} \subset \W_1\der{\X} \subset \cdots \subset \W_n \der{\X} = \der{\X}
\]
defined by the formula
\[
\W_j \der{\X} := \der[\le j]{\X} \wedge \derlog{\X}.
\]
Note that we added a shift of $n$ to the standard dualization. This is
convenient in view of the following elementary result.
\begin{lemma}\label{lem:weight-gerst}
The weight filtration is compatible with the wedge product and Schouten--Nijenhuis bracket of polyvector fields:
\[
\W_j\der{\X}\wedge \W_k\der{\X} \subset \W_{j+k} \der{\X} \qquad [\W_{j}\der{\X},\W_k\der{\X}] \subset \W_{j+k}\der{\X}.
\]
Thus $\W_\bullet \der{\X}$ is a sheaf of filtered Gerstenhaber algebras.
\end{lemma}

The associated graded $\gr^\W \der{\X}$ has the following description, which is dual to the corresponding identification $\eqref{eq:gr-logforms}$ for logarithmic forms. Given a simplex $\Delta \in \dc[k]{\X}$, consider the standard exact sequence for the normal bundle $\cN{\Xdel}$:
\[
\xymatrix{
0\ar[r] & \cT{\Xdel} \ar[r] & i^*_\Xdel \cT{\X} \ar[r] & \cN{\Xdel} \ar[r] & 0.
}
\]
Taking exterior powers of this sequence and composing with the restriction map $\der{\X}\to i_{\Delta*}i_{\Delta}^* \der{\X}$, i.e.~the unit map of the adjunction $(i_{\Delta}^*,i_{\Delta*})$, we obtain a canonical map of graded coherent sheaves 
\[
\xymatrix{
\der{\X} \ar[r] & i_{\Delta*}\der{\Xdel} \otimes \det\cN{\Xdel}[-k]
}
\]
which restricts to a canonical ``coresidue'' morphism
\begin{align}
 \W_k \der{\X} \to i_{\Delta*} \derlog{\Xdel} \otimes \det \cN{\Xdel}[-k]. \label{eq:Tpol-res}
\end{align}
\begin{lemma}\label{lem:gr-poly}
The coresidue map \eqref{eq:Tpol-res} induces an isomorphism 
\begin{align*}
\gr_k^\W \der{\X} &\cong \bigoplus_{\Delta \in \dc[k]{\X}} i_{\Delta*} \derlog{\Xdel} \otimes \det \cN{\Xdel}[-k].
\end{align*}
\end{lemma}

\section{Log symplectic geometry}
\label{sec:log-symp}

\label{sec:log-symp}

\subsection{Log symplectic structures}

We continue to fix a connected complex manifold $\X$ and a simple normal crossings divisor $\bdX\subset \X$ with dual complex $\dcX$ as in \autoref{sec:nc-divisors}.  We will now incorporate a Poisson structure.  We shall assume that the reader is familiar with basic aspects of Poisson geometry, the details of which are covered, for instance, in \cite{Dufour2005,Laurent-Gengoux2013,Polishchuk1997}.

We recall the following notion from \cite{Goto2002}, which makes sense more generally when $\bdX$ is an arbitrary hypersurface.
\begin{definition}
A \defn{log symplectic form} on the pair $(\X,\bdX)$ is a global closed logarithmic two-form $
\omega \in \coH[0]{\logforms[2,cl]{\X}}$ that defines a nondegenerate pairing on the coherent sheaf $\cTlog{\X}$.
\end{definition}

Note that the existence of a log symplectic form implies that the dimension of $\X$ is even.  Inverting $\omega$, we obtain a Poisson bivector
\[
\pi \in \coH[0]{\derlog[2]{\X}}
\]
that is generically nondegenerate, and has $\bdX$ as its degeneracy divisor. That is,  $\bdX$ is the scheme-theoretic vanishing locus of the Pfaffian
\[
\pi^{\tfrac{1}{2}\dim \X} \in \coH[0]{\det{\cT{\X}}(-\bdX)},
\]
so that $\bdX$ is an anti-canonical divisor in $\X$.  In this way, we obtain a bijection between the set of log symplectic structures on $(\X,\bdX)$ and the set of generically nondegenerate Poisson structures on $\X$ with degeneracy divisor $\bdX$.

From now on, we fix a log symplectic structure $\omega$ on $(\X,\bdX)$, or equivalently the corresponding Poisson bivector $\pi$.   Note that $\pi$ is automatically tangent to all strata of $\bdX$, so that for each simplex $\Delta$ of $\dcX$, the restriction
\[
\pidel := \pi|_{\Xdel} \in \coH[0]{\Xdel;\derlog[2]{\Xdel}}
\]
defines a Poisson structure on $\Xdel$, making $\Xdel$ into a closed Poisson submanifold.  Equivalently, $\Xdel$ is preserved by the flow of all Hamiltonian vector fields.   In fact, since $\bdX$ is determined by $\pi$, the strata are preserved by all infinitesimal symmetries, not just the Hamiltonian ones.

Suppose that $\Delta$ is a simplex of $\dcX$, and denote by $\Delta_0 \subset \Delta$ its vertices, so that elements in $\Delta_0$ are in bijection with the irreducible components of $\bdX$ that intersect along $\Xdel$.  Since the log symplectic form $\omega$ is closed, so is its residue $\res{\omega} \in \gr_2\logforms{\X}$.  Restricting the latter to $\Xdel$, we obtain a constant section
\[
B_\Delta  \in   \wedge^2 \CC^{\Delta_0} \subset  \sect{\Xdel,\wedge^2 \cO{\Xdel}^{\Delta_0}}
\]
which we call the \defn{biresidue of $\omega$ along $\Xdel$}.  Note that if we choose an identification $\Delta_0 \cong \{1,\ldots,k\}$, then the biresidue is simply a skew-symmetric matrix of constants $B_\Delta = (B_{ij}) \in \CC^{k\times k}$.  We shall indicate the biresidue pictorially by orienting the edges of the dual complex and labelling them with the corresponding biresidue, as in the following example.

\begin{example}[\runex]
Let $\X=\CC^4$ with boundary divisor $\bdX = \{y_1y_2y_3=0\}$ as in \autoref{ex:3cpt-dual-complex}. Consider the logarithmic two form 
\begin{align}
\omega = b_1 \dlog{y_2}\wedge \dlog{y_3} + b_2 \dlog{y_3}\wedge \dlog{y_1} + b_3 \dlog{y_1}\wedge \dlog{y_2}+ \dd z \wedge \sum_i \dlog{y_i} \label{eq:3cpt-logform}
\end{align}
where $b_i \in \CC$ are constants.  It has a biresidue along each codimension two stratum, i.e.~each edge of the triangle, corresponding to the coefficients of the forms $\dlog{y_i}\wedge\dlog{y_j}$.  We represent this pictorially as follows:
\begin{align}
\res \omega &= \vcenter{\hbox{ \begin{tikzpicture}[scale=0.7,decoration={
    markings,
    mark=at position 0.5 with {\arrow{>}}}
    ] 
    \draw[fill=gray] (-30:1) -- (90:1) -- (210:1) -- (-30:1);
\draw[postaction={decorate}] (-30:1) -- (90:1);
\draw[postaction={decorate}] (90:1) -- (210:1) ;
\draw[postaction={decorate}] (210:1) -- (-30:1);
\draw (90:1.3) node {$(y_2)$};
\draw (-30:1.5) node {$(y_1)$};
\draw (210:1.5) node {$(y_3)$};
\draw (30:0.9) node {$b_3$};
\draw (150:0.9) node {$b_1$};
\draw (-90:0.9) node {$b_2$};
\end{tikzpicture}}} \label{eq:3cpt-bires}
\end{align}
If  $\Delta \in\dc[3]{\X}$ is the 2-simplex corresponding to the unique codimension-three stratum, and we order the vertices $(1,2,3)$ in the obvious way, then the biresidue $B_\Delta \in \wedge^2 \CC^{\Delta_0}$ is represented by the skew-symmetric matrix
\[
B_\Delta = \begin{pmatrix}
0 & b_3 & -b_2 \\
-b_3 & 0 & b_1 \\
b_2 & -b_1 & 0
\end{pmatrix}
\]
containing the biresidues along all codimension two strata.

One easily checks that $\omega$ is nondegenerate if and only if the constant $c := b_1+b_2+b_3$ is nonzero.  In this case, it defines a log symplectic structure on $\X = \CC^4$, whose inverse is the logarithmic bivector
\[
\pi = -\frac{1}{c} \rbrac{ \sum_{i} \logcvf{y_i}\wedge\logcvf{y_{i+1}}+ \cvf{z} \wedge \sum_i {b_i\logcvf{y_i}} }
\]
The restriction  of $\pi$ to any stratum $\Xodel$ is easily determined by setting the relevant coordinates to zero. For instance, $\pi$ vanishes identically along the codimension-three stratum $y_1=y_2=y_3=0$, i.e., if $\Delta \in\dc[3]{\X}$ is the unique 2-simplex as above, then $\pidel = 0$.  
\end{example}

\begin{example}\label{ex:magnetic}
Let $\bdZ \subset \Z$ be a simple normal crossings divisor in a complex manifold $\Z$ with complement $\Zo := \Z \setminus \bdZ$, and let $\X := \logctb{\Z}$ be the logarithmic cotangent bundle, i.e.~the total space of the locally free sheaf $\logforms[1]{\Z}$.  Then $\X$ carries a canonical log symplectic form $\omega$ which agrees with the standard symplectic structure on the open dense set $\Xo = \ctb{\Zo} \subset \X$.  There is a canonical projection $\phi : \X \to \Z$, and the degeneracy divisor on $\X$ is the pullback $\bdX = \phi^*\bdZ$.

If $B \in \coH[0]{\logforms[2,\mathrm{cl}]{\Z}}$ is any closed logarithmic two-form, then the form $\omega + \phi^*B$ is again a log symplectic structure on $\X$ with the same polar divisor $\bdX \subset \X$, yielding a deformation of the Poisson bivector.  In this case, $B$ is called the \defn{magnetic term}.  
\end{example}

\begin{example}
As a special case of the previous example, suppose that $\Z=\CC^k$ with coordinates $y_1,\ldots,y_k$ and that $\bdZ \subset \Z$ is the union of the coordinate hyperplanes. If we denote by $p_j := y_j\cvf{y_j}$ the corresponding coordinates on the logarithmic cotangent bundle, then the canonical Poisson structure and log symplectic form are given by
\begin{align*}
\pi &= \sum_{j=1}^k \logcvf{y_j} \wedge \cvf{p_j}  & \omega &= \sum_{j=1}^k \dd p_j \wedge \dlog{y_j}.
\end{align*}
Every magnetic term is cohomologous in $\logforms{\Z}$ to one of the form
\[
B = \sum_{1 \le i < j \le k} B_{ij} \dlog{y_i}\wedge\dlog{y_j}
\]
where $(B_{ij}) \in \CC^{k\times k}$ is a skew-symmetric matrix.  We then calculate the deformed bivector
\begin{align}
\pi' := (\omega+\phi^*B)^{-1} &= \sum_j \logcvf{y_j} \wedge \cvf{p_j} + \sum_{i,j}B_{ij}\cvf{p_i}\wedge\cvf{p_j},  \label{eq:local-magnetic}
\end{align}
giving the corresponding Poisson structure on $\logctb{\Z}$.
\end{example}

Every symplectic manifold is locally isomorphic to a cotangent bundle via the Darboux theorem.  Similarly, every log symplectic structure on $(\X,\bdX)$ is locally equivalent to a log cotangent bundle with magnetic term, but the relevant notion of equivalence is weaker than isomorphism:
\begin{definition}
Two germs $\Y,\Y'$ of Poisson manifolds are \defn{stably equivalent} if there exist symplectic germs $\bS,\bS'$ such that $\Y \times \bS \cong \Y'\times \bS'$.
\end{definition}
A fundamental result in Poisson geometry, known as Weinstein's splitting theorem~\cite{Weinstein1983}, states that every Poisson germ is stably equivalent to a germ of a Poisson structure that vanishes at a point, the latter being uniquely determined up to isomorphism.  We apply this result to establish the following local model:

\begin{proposition}\label{prop:stable-form}
Suppose that $\pi$ is a log symplectic Poisson structure on the normal crossings divisor $(\X,\bdX)$, and suppose that $\Xodel \subset \X$ is a stratum of codimension $k$, corresponding  to a simplex $\Delta$.  Then near any point $p \in \Xodel$, the germ $(\X,\pi)_p$ is stably equivalent to the germ \eqref{eq:local-magnetic}, where $(B_{ij})$ is the matrix of biresidues along $\Xdel$ relative to an ordering of $\Delta_0$. 
\end{proposition}

\begin{proof}
First note that near a point of corank $2j$, the Pfaffian of the Poisson tensor must vanish to order at least $j$.  Since $\bdX$ has multiplicity $k$ at $p$, we conclude that the corank of $\pi$ at $p$ is at most $2k$. Thus, by taking products with symplectic germs or splitting off transverse factors using Weinstein's splitting theorem, we may  assume without loss of generality that the dimension of $\X$ is twice the number of irreducible components of $\bdX$ containing $p$.  Thus the normal crossings germ $(\X,\bdX)_p$ is equivalent to the germ at the origin of the logarithmic cotangent bundle of $\Z=\CC^k$ with divisor $\bdZ \subset \Z$ given by the union of the coordinate hyperplanes.  Evidently the biresidues of $\omega$ are the same as those of the normal form \eqref{eq:local-magnetic}. A Moser argument then shows that their germs are isomorphic; see \cite[Example 3.9]{Pym2017}. 
\end{proof}

\begin{remark}
It was observed in \cite[Example 3.9]{Pym2017} that the problem of parameterizing the isomorphism classes of normal crossings log symplectic germs of a given dimension is somewhat subtle, due to constraints on the rank imposed by the dimension and the number of irreducible components.  The notion of stable equivalence gives a cleaner formulation, and this weaker notion will suffice for our purposes. 
\end{remark}

\subsection{Symplectic leaves}

Recall that every Poisson manifold comes equipped with a canonical foliation whose leaves are symplectic.  The tangent sheaf of this foliation is the image of the \defn{anchor map}
\[
\pis : \forms[1]{\X} \to \cT{\X}.
\]
given by contraction of forms into the bivector $\pi$.  In the log symplectic case, this map extends to an isomorphism $\logforms{\X} \cong \cTlog{\X}$, which we denote by the same symbol.

It was observed in~\cite{Polishchuk1997} that that if $\pi$ is log symplectic on $(\X,\bdX$), then the rank of $\pi$ is constant along every stratum $\Xodel \subset \X$, i.e., every symplectic leaf in $\Xodel$ has the same dimension.  This is also evident from the stable local normal form \eqref{eq:local-magnetic}, which furthermore makes it clear that the rank along $\Xodel$ is determined by the residue of $\omega$ near $p$.

This can be seen more directly, as follows~\cite[Section 5.3.2]{Pym2018a}.  Suppose that  $\Delta \subset \dcX$ is a simplex, and $\Delta_0$ is its set of vertices.  Then by taking residues of logarithmic one-forms along the components of $\bdX$ labelled by $\Delta_0$, we obtain a short exact sequence
\[
\xymatrix{
0 \ar[r] & \logforms[1]{\Xdel} \ar[r] & \idelb\logforms[1]{\X} \ar[r] & \cO{\Xdel}^{\Delta_0}  \ar[r] & 0
}
\]
Note that the bundle $\cO{\Xdel}^{\Delta_0}$ is trivial, with basis given by the elements of $\Delta_0$; hence it is canonically self-dual, and we may alternatively view the biresidue $\Bdel \in \wedge^2 \CC^{\Delta_0}$ as a constant skew-adjoint endomorphism of this bundle.   We then have the following commutative diagram of vector bundles on $\Xdel$, with exact rows:
\begin{align*}
\xymatrix{
0 \ar[r] 	&	\logforms[1]{\Xdel} \ar[r]\ar[d]_{\pisdel}		& \idelb\logforms[1]{\X} \ar[r] \ar@/_/[d]_{\pis} 	& \cO{\Xdel}^{\Delta_0}  \ar[r] & 0  \\
0  			& \cTlog{\Xdel} \ar[l] 								& \idelb\cTlog{\X} \ar[l] \ar@/_/[u]_{\omega^\flat}		& \cO{\Xdel}^{ \Delta_0} \ar[l] \ar[u]_{\Bdel} & 0.  \ar[l] 
}
\end{align*}
A quick diagram chase then establishes the following:
\begin{lemma}\label{lem:const-rank}
The normal bundle of the symplectic leaves in $\Xodel$ is canonically isomorphic to $\coker \Bdel$.  In particular, all symplectic leaves in $\Xodel$ have the same dimension.
\end{lemma}

\begin{lemma}\label{lem:bires-nondegen}
The Poisson submanifold $(\Xdel,\pidel)$ is log symplectic if and only if the bilinear form $B_\Delta$ is nondegenerate.  
\end{lemma}

\begin{example}
If $\Delta \in \dcX[2]$ is a 1-simplex (an edge), then the corresponding codimension-two stratum $\Xdel$ is log symplectic if and only if the biresidue $\Bdel$ is nonzero. 
\end{example}

\begin{example}\label{ex:tetra-nondegen}
If $\Delta \in \dcX[4]$ is a 3-simplex (a tetrahedron), the biresidue $\Bdel$ may depicted as follows:
\begin{center}
\begin{tikzpicture}[scale=1,decoration={
    markings,
    mark=at position 0.55 with {\arrow{>}}}]
\coordinate (A) at (210:1) {};
\coordinate (B) at (0,0) {};
\coordinate (C) at (-30:1)  {};
\coordinate (D) at (90:1) {};
\draw[fill=gray,opacity=0.6] (A) -- (B) -- (C) -- (A);
\draw[gray,fill,opacity=0.8] (A) -- (C) -- (D) -- (C);
\draw[fill=gray,opacity=0.2] (A) -- (B) -- (D) -- (A);
\draw[fill=gray,opacity=0.4] (C) -- (B) -- (D) -- (C);
\foreach \i in {0,1,2}
{
	\draw (120*\i-90:0.75) node {$\scriptstyle a_{\i}$};
	\draw (120*\i+65:0.4) node {$\scriptstyle b_{\i}$};
}
\draw[postaction=decorate] (A) -- (B);
\draw[postaction=decorate] (C) -- (B);
\draw[postaction=decorate] (D) -- (B);
\draw[postaction=decorate] (A) -- (C);
\draw[postaction=decorate] (C) -- (D);
\draw[postaction=decorate] (D) -- (A);
\end{tikzpicture}
\end{center}
The corresponding codimension-four stratum $\Xdel$ is then log symplectic if and only if $a_0b_0+a_1b_1+a_2b_2 \ne 0$. 
\end{example}

\begin{remark}\label{rmk:non-simple-symplectic}
It follows from \autoref{lem:bires-nondegen} that if $\bdX$ is a non-simple normal crossings divisor and $\Xdel$ is a symplectic stratum, then the orientation local system $\ordel$ is trivial. In particular, for a codimension-two stratum this implies that the divisor is simple in a tubular neighbourhood of $\Xodel$.  
\end{remark}

\subsection{Characteristic symplectic leaves}

Recall from \cite{Brylinski1999,Polishchuk1997,Weinstein1997} that if $\mu \in \forms[\dim \X]{\X}$ is a volume form on a Poisson manifold, then there is canonical vector field $\zeta \in \cT{\X}$, called the \defn{modular vector field}, which plays the role of the divergence of the Poisson tensor with respect to $\mu$.  It is defined by the property that 
\[
\lie{\pis(\dd f)} \mu = -\zeta(f) \cdot \mu
\]
for all $f \in \cO{\X}$.  One can show that if $\bS\subset \X$ is a symplectic leaf, then $\zeta$ is either tangent to $\bS$ at every point, or transverse to $\bS$ at every point.  Moreover, if we change the volume form, then the modular vector field changes by a Hamiltonian vector field, so that the condition of tangency is independent of the choice of $\mu$.  

\begin{definition}\label{def:char-leaves}
We say that a symplectic leaf $\bS$ of a Poisson manifold is \defn{characteristic} if the modular vector field with respect to some (and hence any) locally defined volume form is tangent to $\bS$.
\end{definition}

\begin{example}\label{ex:char-surf}
Let $\X$ is a Poisson surface with degeneracy divisor $\bdX$, then the open leaf $\Xo = \X \setminus \bdX$ is characteristic, and a point $p \in \bdX$ is a characteristic symplectic leaf if and only if it is a singular point of $\bdX$; see \cite[Example 3.3]{Pym2018}. 
\end{example}

\begin{example}\label{ex:char-stratum}
The modular vector fields is an infinitesimal symmetry of the Poisson structure and is therefore tangent to all normal crossings strata.  In particular, if some stratum $\Xodel$ is a symplectic leaf, it is necessarily characteristic.
\end{example}

The use of the word ``characteristic'' is motivated by the importance of these leaves in the study of Poisson cohomology; see \autoref{conj:char-holonomic}.  We warn the reader that this notion is sensitive to the ambient Poisson manifold: if $\Y \subset \X$ is a Poisson submanifold, then a symplectic leaf $\bS \subset \Y$  that is characteristic for $\Y$ need not be a characteristic leaf for $\X$. For instance, $\bS$ is always characteristic in itself, but \autoref{ex:char-surf} shows that it need not be characteristic in $\X$.

In the case of a normal crossing log symplectic structure, we may determine the characteristic symplectic leaves using the biresidue as follows.

\begin{proposition}\label{prop:res-adapt}
Suppose that $k \ge 0$, and let $\Delta \in \dc[k]{\X}$ be a simplex.  Then the following statements are equivalent:
\begin{enumerate}
\item $\Xodel$ contains a characteristic symplectic leaf of $(\X,\pi)$.
\item Every symplectic leaf in $\Xodel$ is characteristic for $(\X,\pi)$.
\item The constant vector $(1,\ldots,1) \in \CC^{\Delta_0}$ lies in the image of the biresidue form $\Bdel \in \wedge^2 \CC^{\Delta_0}$.
\end{enumerate}
\end{proposition}
Here, a vector $v$ lies in the image of a form $\alpha \in \wedge^2 V$ if
there exists an element $\xi \in V^*$ such that $(\xi \otimes 1)(\alpha) = v$. Thus 3) can alternatively be phrased by writing $\Bdel$ as a matrix using the basis $\Delta_0$ and requiring that the constant vector is in the image.
\begin{proof}[Proof of Proposition \ref{prop:res-adapt}]
The condition that a leaf $\bS$ be characteristic can be checked locally near a point of $\bS$, and is moreover invariant under stable equivalence of germs.  Hence by \autoref{prop:stable-form}, we may assume without loss of generality that the Poisson structure $\pi$ is given near a point $p \in \Xodel$ by the stable normal form \eqref{eq:local-magnetic}.  Using the standard expression for the modular vector field $\zeta$ of $\pi$ with respect to the coordinate volume form~\cite[Lemma 4.5]{Polishchuk1997}, we compute
\[
\zeta = \sum_{j=1}^k \cvf{p_i}
\]
Meanwhile, the Poisson structure on $\Xodel$ is a constant bivector:
\[
\pidel =
  \sum_{1 \le i < j \le k} B_{ij} \cvf{p_i}\wedge\cvf{p_j}.
\]
Its symplectic leaves are the translates of the linear subspace given by the image of the matrix $(B_{ij})$ of $\pidel$ in the basis $\cvf{p_i}$.  Meanwhile $\zeta$ is represented in this basis by the vector $(1,\ldots,1)$.  The equivalence of the statements 1 through 3 follows immediately.
\end{proof}

By considering the linear algebra of $k\times k$ skew-symmetric matrices for small values of $k$, one easily determines when a stratum of small codimension admits characteristic leaves.  We summarize the result as follows.
\begin{corollary}\label{ex:tri-holonomic} Let $\Xodel$ be a codimension-$k$ stratum.  Then the following statements hold. 
\begin{itemize}
\item If $k=0$, then the whole stratum $\Xodel = \Xo$ is a characteristic leaf.
\item If $k > 0$ and $B_\Delta = 0$, then no leaves in $\Xodel$ are characteristic
\item If $k=2$, and $B_\Delta \ne 0 \in \CC$, then the whole stratum $\Xodel$ is a characteristic leaf.  
\item If $k=3$, $B_\Delta\ne0$, and we represent the biresidue as follows:
\begin{align*}
B_\Delta &= \vcenter{\hbox{ \begin{tikzpicture}[scale=0.6,decoration={
    markings,
    mark=at position 0.5 with {\arrow{>}}}
    ] 
    \draw[fill=gray] (-30:1) -- (90:1) -- (210:1) -- (-30:1);
\draw[postaction={decorate}] (-30:1) -- (90:1);
\draw[postaction={decorate}] (90:1) -- (210:1) ;
\draw[postaction={decorate}] (210:1) -- (-30:1);
\draw (30:0.9) node {$b_2$};
\draw (150:0.9) node {$b_0$};
\draw (-90:0.9) node {$b_1$};
\end{tikzpicture}}}, \label{eq:3cpt-bires}
\end{align*}
then the leaves in $\Xodel$ are characteristic if and only if $b_0+b_1+b_2=0$.
\end{itemize}
\end{corollary}

\begin{corollary}\label{cor:incident-nonhol}
Let $\Xodel$ be a stratum of codimension $k$ whose Poisson structure is degenerate, and whose leaves are all characteristic in $\X$.  Then $\Xodel$ lies in the boundary of a stratum of  codimension $k-1$ whose leaves are characteristic in $\X$.
\end{corollary}

\begin{proof}
By \autoref{prop:res-adapt}, this amounts to the statement that if $B$ is a degenerate skew matrix of size $k$ whose image contains the constant vector $(1,\ldots,1) \in \CC^k$, then $B$ has a skew submatrix of size $k-1$ whose image contains $(1,\ldots,1)\in \CC^{k-1}$.
\end{proof}

\subsection{Tubular neighbourhoods}

Suppose that $\Xdel \subset \X$ is a closed stratum of a normal crossings log symplectic manifold.  Then we may construct a natural Poisson structure on its normal bundle $\nbdel$ by deformation to the normal cone, as follows.

To begin, note that if $\cI\subset \cO{\X}$ is the ideal sheaf of $\Xdel$, then near $\Xdel$, the ideal $\cI$ is the intersection of the ideals $\cO{\X}(-\Y_i)$ where $\Y_i \subset \bdX$ are the irreducible components containing $\X$.  Each of these components is a Poisson subvariety, so that the corresponding ideals are Poisson.  It follows that $\{\cI,\cI\}\subset\cI^2$.  Hence the associated graded ring
\[
\bigoplus_{j \ge 0} \cI^j/\cI^{j+1} \cong \Sym_{\cO{\Xdel}} \coNdel
\]
carries a canonical homogeneous Poisson bracket.  But this ring is precisely the fibrewise-polynomial functions on the normal bundle $\nbdel$, and hence we obtain a canonically defined Poisson structure on $\nbdel$, which we call the \defn{homogenization} of $(\X,\pi)$ along $\Xdel$.  Note that the homogenization is invariant under the natural action of the torus $(\CC^*)^{\Delta_0}$ on the normal bundle induced by the splitting \eqref{eq:nb-split} into  line subbundles.  Moreover, the inclusion $\Xdel \to \nbdel$ and the projection $\nbdel \to \Xdel$ are automatically Poisson maps, dual to the inclusion and projection of the pair of Poisson algebras $(\cO{\Xdel},\Sym_{\cO{\Xdel}} \coNdel)$.  

\begin{example}\label{ex:symplectic-tube}
When $\Xodel$ is symplectic, the homogeneous log symplectic form $\omega_0$ on $\nbdel$  takes on a particularly simple shape, as follows.  By taking the $\omega_0$-orthogonal  we split the tangent sheaf as a direct sum of horizontal and vertical components, defining a flat connection on the vector bundle $\nbdel$ with structure group given by the torus $(\CC^*)^{\Delta_0}$. In this decomposition, the homogenized log symplectic form may be written canonically as 
\[
\omega_0 = \sum_{i,j \in \Delta_0} B_{ij}\dlog{y_i}\wedge\dlog{y_j} + \rho^* \omega_{\Delta}
\]
where $y_i$ are any flat fibre-linear coordinates, $(B_{ij})$ is the biresidue along $\Xdel$, and $\omega_\Delta$ is the log symplectic form on $\Xdel$, which has been pulled back by the bundle projection $\rho : \nbdel \to \Xdel$. 
\end{example}

\begin{proposition}\label{prop:homogenization}
Suppose that the stratum $\Xodel \subset \X$ admits a holomorphic tubular neighbourhood.  Then $\Xodel$ admits a holomorphic Poisson tubular neighbourhood, i.e.~a holomorphic tubular neighbourhood embedding that identifies $\pi$ with its homogenization. 
\end{proposition}

\begin{proof}
Without loss of generality, we may assume that $\X$ is a neighbourhood of the zero section in $\nbdel$ and that $\Xodel=\Xdel$. 
Furthermore, by \autoref{lem:rectify} we may assume that $\bdX = \bdnbdel \cap \X$.  Taking Taylor expansions along the fibres, we obtain a canonical decomposition $\omega = \sum_{k=0}^\infty \omega_k$ where $\omega_k$ is a logarithmic two-form whose germ along the zero section has weight $k$ for the $\CC^*$-action.   In particular, the homogenization of $\omega$ is equal to $\omega_0$, and the form $\omega_k$ is exact for all $k > 0$. The analytic family of forms $\omega(t) = \sum_{k=0}^\infty t^k \omega_k$ therefore has constant cohomology class and interpolates between $\omega(0)=\omega_0$ and $\omega(1) =\omega$.  Moreover, the restriction $\omega(t)|_{\Xdel} \in \logforms[2]{\nbdel}|_{\Xdel}$ is constant in $t$.  In particular, the germ of $\omega(t)$ along the zero section is nondegenerate for all $t$. Applying the Moser argument, we produce a time dependent logarithmic vector field that vanishes identically along the zero section, giving an isotopy of $(\nbdel,\bdnbdel)$ that takes  $\omega_0$ to $\omega$ and fixes $\Xdel$ pointwise.
\end{proof}

\section{Poisson cohomology}
\label{sec:HP}

\subsection{The Lichnerowicz complex}
\label{sec:chlgy-basics}

Throughout the present brief \autoref{sec:chlgy-basics}, $(\X,\pi)$ denotes an arbitrary complex Poisson manifold; we make no assumption of nondegeneracy, but will return  to the log symplectic case in \autoref{sec:HP-logsymp}.

As observed by Lichnerowicz~\cite{Lichnerowicz1977}, the adjoint action of the bivector $\pi$ on $\der{\X} = \Tpol{\X}$  defines a differential
\[
\dpi := [\pi,-] : \der{\X} \to \der[\bullet +1]{\X}.
\]
We call the resulting complex $(\der{\X},\dpi)$ the \defn{Lichnerowicz complex}. Its hypercohomology
\[
\Hpi{\X} := \coH{\der{\X},\dpi}
\]
is called the \defn{Poisson cohomology of $\X$}; see, e.g.~\cite{Dufour2005,Laurent-Gengoux2013} for a detailed exposition.

We will also need the more general notion of Poisson cohomology with coefficients in a module, as follows.  Recall that a \defn{Poisson module} is a coherent sheaf $\cE$ equipped with a contravariant connection operator $\nabla : \cE \to \cT{\X}\otimes \cE$ that extends to a differential $\dpi^\nabla$ on $\der{\X}\otimes \cE$, making the latter into a differential graded module over $(\der{\X},\dpi)$.  We then define the Poisson cohomology with coefficients in $\cE$ as the hypercohomology $\Hpi{\X;\cE} := \coH{\der{\X}\otimes \cE,\dpi^\nabla}$.  For details of this construction, see~e.g.~\cite[Section 1]{Polishchuk1997}. We shall only need the case in which $\cE=\sL$ is a line bundle.  In this case, if $s \in \sL$ is a local trivialization, we have
\[
\nabla s = \zeta \otimes s
\]
where $\zeta \in \cT{\X}$ is a Poisson vector field, i.e., $\lie{\zeta}\pi = 0$. 

For convenience, we will say that a contravariant
connection $\nabla$ is tangent to a submanifold $\Y$ if the image is,
i.e., if $\nabla s|_{\Y} \in \cE \otimes \cT{\Y}$ for all local sections $s$ of $\cE$.
\begin{example}\label{ex:modular-rep}
The canonical bundle $\can{\X} = \det\forms[1]{\X}$ carries a natural Poisson module structure defined by the formula $\nabla_{\dd f} \mu =- \lie{\pis(\dd f)} \mu$ for all $f \in \cO{\X}$ and all locally defined volume forms $\mu \in \can{\X}$.  The connection vector field is the modular vector field with respect to $\mu$.  Thus a symplectic leaf $\bL$ is characteristic in the sense of \autoref{def:char-leaves} if and only if the connection operator $\nabla : \can{\X} \to \cT{\X}\otimes \can{\X}$ is tangent to $\bL$. 
\end{example}

\begin{example}\label{ex:det-normal}
Suppose that that $i : \Y \hookrightarrow \X$ is a closed Poisson submanifold of codimension $k$, with normal bundle $\cN{\Y}$.  Suppose that the Poisson connection on the canonical bundle $\can{\X}$ is tangent to $\Y$.  Then it descends to a Poisson module structure on the line bundle $i^*\can{\X}$ over $\Y$.  It follows that $\det{\cN{\Y}} \cong i^*\acan{\X}\otimes \can{\Y}$ carries a canonical Poisson module structure in this case.  One can then verify that the natural projection
\[
\der{\X} \to i_*(\der{\Y}\otimes \det{\cN{\Y}}[-k])
\]
is a morphism of complexes.  
\end{example}

In general the local structure of a Poisson module can be quite complicated, but there is a simple case in which the local structure is trivial:
\begin{lemma}\label{lem:locally-trivial}
Suppose that $\pi$ has constant rank and that $(\sL,\nabla)$ is a Poisson line bundle whose connection is tangent to the symplectic leaves.  Then $\sL$ is locally trivial as a Poisson module.  In particular, $\scoH[0]{\der{\X}\otimes \sL}$ is a free module over the sheaf $\scoH[0]{\der{\X}} \subset \cO{\X}$ of Casimir functions.
\end{lemma}

\begin{proof}[Proof]
  Choose a local trivialization $s \in \sL$. Then the connection vector field $\zeta$ is a Poisson vector field tangent to the symplectic leaves.  Since $\pi$ has constant rank in a neighbourhood of $s$, such a vector field is automatically locally Hamiltonian; hence by shrinking the domain we may assume that $\zeta = \pis \dd h$ for some $h \in \cO{\X}$.  But then we have $\nabla(e^hs) = 0$, so that the flat section $e^hs \in \sL$ is a local trivialization of $(\sL,\nabla)$.
\end{proof}

\subsection{Cohomology of log symplectic structures}
\label{sec:HP-logsymp}
From now on, we fix a normal crossing divisor $(\X,\bdX)$, and a Poisson structure $\pi$ induced by a log symplectic form as in \autoref{sec:log-symp}, and consider the structure of its Lichnerowicz complex.

Note that since $\pi \in \derlog[2]{\X} = \W_0 \der{\X}$, it follows from \autoref{lem:weight-gerst} that the differential $\dpi$ on $\der{\X}$ preserves the weight filtration on $\der{\X}$.  In particular, the weight zero part, given by the logarithmic polyvectors $(\derlog{\X},\dpi)$, is a subcomplex which we call the \defn{logarithmic Lichnerowicz complex}.  More generally, if $\cE$ is a Poisson module whose connection $\cE \to \cT{\X}\otimes \cE$ factors through $\cTlog{\X}\otimes \cE$, then we may consider the logarithmic Lichnerowicz complex $\derlog{\X}\otimes \cE$ with coefficients in $\cE$.

In light of \autoref{lem:gr-poly}, summands in the associated graded complex $(\gr \der{\X},\dpi)$ are indexed by simplices in the dual complex $\dc{\X}$.  More precisely, if $\Delta \in \dc[k]{\X}$ is a $(k-1)$-simplex, let 
\[
\grpol{\Delta} := \idelb\gr_k\der{\X}[k]  \cong \derlog{\Xdel}\otimes \det{\cN{\Xdel}}
\]
Then $\idelf\grpol{\Delta}[-k]$ is the unique summand of $\gr_k \der{\X}$ whose support is the codimension-$k$ submanifold $\Xdel$.  Note that since the modular vector field is a symmetry of $\pi$, it is automatically tangent to all strata of $\bdX$.  In light of \autoref{ex:det-normal}, we conclude that $\grpol{\Delta}$ is simply the logarithmic Lichnerowicz complex of the Poisson line bundle $\det{\cN{\Xdel}}$ on the Poisson manifold $(\Xdel,\pidel)$.  The behaviour of this complex is controlled, in large part, by whether or not the symplectic leaves in $\Xodel$ are characteristic, giving another condition on strata that is equivalent to those listed in \autoref{prop:res-adapt}:
\begin{proposition}\label{prop:coh-adapt}
Suppose that $\pi$ is a log symplectic Poisson structure on the simple normal crossings divisor $(\X,\bdX)$, and that $\Delta$ is a simplex of $\dcX$.  Then the complex $\grpol{\Delta}$ has nontrivial cohomology sheaves if and only if the symplectic leaves in $\Xodel$ are characteristic for $(\X,\pi)$.  In this case, the stalk cohomology is finite-dimensional if and only if $\Xodel$ is symplectic.
\end{proposition}

\begin{proof}
We have the isomorphism $\det{\cN{\Xdel}} \cong \idelb\acan{\X}\otimes \can{\Xodel}$ as Poisson modules, and by \autoref{lem:const-rank} the rank of $\pi|_{\Xodel}$ is constant; hence by the Darboux--Weinstein local form, the induced Poisson structure on $\Xodel$ is constant in suitable coordinates.  It follows that every symplectic leaf in $\Xodel$ is characteristic for $\Xodel$, i.e., the connection on $\can{\Xodel}$ is tangent to all symplectic leaves in $\Xodel$.    We conclude that the connection on $\det{\cN{\Xdel}}$ is tangent to a leaf $\bL \subset \Xodel$ if and only if the same holds for $\acan{\X}$, i.e., the modular vector field of $(\X,\pi)$ is tangent to $\bL$, so that $\bL$ is a characteristic symplectic leaf of $(\X,\pi)$.  In light of \autoref{prop:res-adapt}, there are therefore only two cases to consider: either all leaves in $\Xodel$ are characteristic, or none of them are.

Suppose first that all leaves are characteristic.  Then by \autoref{lem:locally-trivial}, the zeroth cohomology of $\grpol{\Delta}|_{\Xodel}$ is a locally free module over the Casimir functions on $\Xodel$.  In particular it is nonzero, and finite-dimensional if and only if $\Xodel$ is symplectic.

For the converse, suppose that that the leaves are not characteristic.  We will show that the cohomology sheaves vanish by constructing a contracting homotopy, as follows.  (The construction was inspired by the argument in \cite{Marcut2014} in the case of smooth hypersurfaces.)  By \autoref{prop:res-adapt}, there exists a vector  $v \in \ker B_\Delta \subset \CC^{\Delta_0}$ such that $\abrac{v,(1,\ldots,1)} = 1$, where $(1,\ldots,1) \in \CC^{\Delta_0}$ is the constant vector, and $\abrac{-,-}$ is the self-duality pairing on $\CC^{\Delta_0}$ determined by the basis $\Delta_0$.  Viewing $v$ as a section of $\idelb\cTlog{\X}$ and applying the log symplectic form $\omega$, we obtain a closed logarithmic one-form $\alpha \in \logforms[1]{\Xdel}$ lying in the kernel of $\pisdel$ and having the following additional property: if $\mu \in \acan{\X}|_{\Xodel}$ is a section, then $\nabla_\alpha\mu = \mu$.  Contraction with $\alpha$ then defines an operator  $h = \iota_{\alpha}$ of degree $-1$ on $\grpol{\Delta}$ such that $\dpi^\nabla h + h\dpi^\nabla = 1$, giving the desired contracting homotopy.
\end{proof}

\subsection{Contributions from log symplectic strata}
\label{sec:log-dR-strata}

Of particular relevance to us are  the contributions to the Poisson cohomology corresponding to submanifolds $\Xdel$ that are themselves log symplectic.  We now explain how to compute these contributions in topological terms.

Suppose that $\Delta \subset\dcX$ is a simplex for which $\Xdel$ is log symplectic.  Then under the resulting isomorphism $\cTlog{\Xdel} \cong \logforms[1]{\Xdel}$ given by the dual $(\pisdel)^\vee = - \pisdel$ of the anchor map, the Poisson module structure on $\det{\cNdel}$ becomes a flat logarithmic connection, and we obtain a canonical isomorphism
\[
\grdel  \cong (\logforms{\Xdel}\otimes \det{\cN{\Xdel}},\dd^\nabla)
\]
with the logarithmic de Rham complex of this flat connection; it is  simply the determinant of the $(\CC^*)^{\Delta_0}$-connection on the normal bundle induced by the log symplectic structure as in \autoref{ex:symplectic-tube}.

The de Rham complexes of flat logarithmic connections have been well studied; see, e.g.~the classic text~\cite{Deligne1970} of Deligne. The key point for us is that the cohomology is isomorphic to a certain sheaf cohomology group of the corresponding local system
\[
\sLdel := \ker \nabla|_{\Xodel} \subset \det{\cN{\Xodel}}
\]
on the open part $\Xodel \subset \Xdel$, relative to certain boundary components.  To express this precisely, we must first explain how to determine the invariants that classify the connection up to isomorphism (the monodromy and residue) in terms of periods of the log symplectic form.

\subsubsection{Monodromy}
\label{sec:monodromy}
Over the stratum $\Xodel$, the $(\CC^*)^{\Delta_0}$-connection $(\cNdel,\nabla)$ is flat and holomorphic, and is therefore determined up to isomorphism by its monodromy around loops $\torloop \subset \Xodel$.  These monodromies are elements of the torus $(\CC^*)^{\Delta_0}$ and may be expressed purely in terms of the periods of  $\omega$ along cycles in the open leaf $\Xo$, as follows.

Observe that if $\torloop$ is any oriented loop, the bundle $\nbdelo|_{\torloop}$ is topologically trivial, but not canonically so.  If we choose a trivialization, then we may construct, for each $i \in \Delta_0$, a positively oriented loop $\gamma_i$ in the fibre that wraps once around the origin in the line subbundle $\nb_{\Delta,i} \subset \nbdel$ and is constant on the other factors. The product $\torloop \times \gamma_i$ then defines an oriented two-torus in $\nbdelo$ which we may push forward along an arbitrary  $C^\infty$ tubular neighbourhood embedding $(\nbdel,\bdnbdel) \hookrightarrow (\X,\bdX)$ to define a homology class
\[
\torcyc_i \in \Hlgy[2]{\Xo;\ZZ}.
\]
We may then define the period
\[
\int_{\torcyc_i} \omega \in \CC.
\]
In this way, our choice of trivialization of the bundle $\nbdelo|_\torloop$ gives a vector
\begin{align}
P_\torloop := \rbrac{\int_{\torcyc_i}\omega}_{i \in \Delta_0} \in \CC^{\Delta_0} \label{eq:period-vector}
\end{align}
Now the biresidue $\Bdel \in \wedge^2 \CC^{\Delta_0}$  is nondegenerate since $\Xodel$ is symplectic, so that the vector $\Bdel^{-1} P_{\torloop} \in \CC^{\Delta_0}$ is defined.  This allows us to state the following result.
\begin{proposition}\label{prop:monodromy}
The monodromy of the flat torus bundle $(\cNdel,\nabla)$ around a loop $\torloop$ in $\Xodel$ is given by
\begin{align}
\exp\rbrac{\tfrac{-1}{2 \pi \iu} \Bdel^{-1} P_\torloop} \in (\CC^*)^{\Delta_0} \label{eq:period-monodromy}
\end{align}
where $P_\torloop$ is the period vector \eqref{eq:period-vector} with respect to any trivialization of $\nbdelo|_{\torloop}$.
\end{proposition}

\begin{remark}
The value of $P_\torloop$ depends on the choice of the cycles $\torcyc_i$, but the monodromy \eqref{eq:period-monodromy} is independent of this choice. 
\end{remark}

\begin{proof}
Note that the computation may be done in an arbitrary neighbourhood of $\torloop$ in $\X$.  Hence we may assume without loss of generality that $\Xodel$ is a Stein manifold that retracts onto $\torloop$, and in particular $\Xodel$ admits a holomorphic tubular neighbourhood in $\X$ by a theorem of Siu~\cite{Siu1976}.  Thus by \autoref{prop:homogenization} we may reduce the problem to the case in which $\X = \nbdel$ is the normal bundle, and $\omega$ is a torus-invariant log symplectic structure.  Since $\Xodel$ is Stein and $\coH[2]{\Xodel;\ZZ} =0$ every line bundle over $\Xodel$ is trivial, and hence we may choose a holomorphic trivialization of $\nbdel$ over $\Xodel$. Relative to such a trivialization, we may write
\[
\omega = \sum_{i,j \in \Delta_0} B_{ij}\dlog{y_i}\wedge\dlog{y_j} + \sum_{i \in \Delta_0} A_i \wedge \dlog{y_i} + \omega'
\]
where $y_i$ are fibre-linear coordinates, $A_i \in \coH[0]{\Xodel;\forms[1]{\Xodel}}$ are closed one-forms and $\omega' \in \coH[0]{\Xodel;\forms[2]{\Xodel}}$ is the symplectic form on $\Xodel$.  As discussed in \autoref{ex:symplectic-tube}, the flat connection on the normal bundle is precisely the symplectic connection induced by $\omega$, for which the decomposition into vertical and horizontal pieces is $\omega$-orthogonal.  From this we deduce that the connection is given, in this trivialization, by the vector-valued operator
\[
\nabla = \dd + \Bdel^{-1} \cdot (A_i)_{i \in \Delta_0}
\] 
Furthermore the class $\torcyc_i \in \Hlgy[2]{\nbdelo;\ZZ}$ is represented by $\torloop \times \gamma_i$ where $\gamma_i$ is a loop in the fibre defined by $|y_i|=1$ and $y_j=1$ for $j\ne 1$.  The holonomy around $\torloop$ is therefore given by the exponential of the expression
\begin{align*}
-\Bdel^{-1}\rbrac{\int_{\torloop} A_i}_{i} &=   -\Bdel^{-1} \rbrac{ \tfrac{1}{2\pi \iu} \int_{\torloop \times \gamma_i} A_i \wedge \dlog{y_i}}_i \\
&= \tfrac{-1}{2\pi \iu} \Bdel^{-1} \rbrac{\int_{\torcyc_i} \omega}_i \\
&= \tfrac{-1}{2\pi \iu} \Bdel^{-1}P_\torloop,
\end{align*}
as claimed.
\end{proof}

The holonomy of the determinant bundle is the product of the holonomies of the line bundles $\cN{\Delta,i}$, so the local system $\sLdel \subset \det{\cNdel}|_{\Xodel}$ appearing in the Poisson cohomology is determined up to isomorphism by the following corollary.
\begin{corollary}\label{cor:no-monodromy}
The monodromy of the local system $\sLdel$ of flat sections of $\det{\cNdel}$ around a loop $\torloop$ in $\Xodel$ is given by
\[
\exp\rbrac{\tfrac{-1}{2 \pi \iu} (1,...,1)\cdot \Bdel^{-1} \cdot P_\torloop}  \in \CC^*
\]
where $(1,\ldots,1) \in \CC^{\Delta_0}$ is the constant vector.  In particular, the local system $\sLdel$ is trivial if and only if
\[
(1,\ldots,1)\Bdel^{-1}P_\torloop \in  (2 \pi \iu)^{2} \ZZ
\]
for all loops $\torloop$.
\end{corollary}

\subsubsection{Residue}

Recall that the \defn{residue} of a logarithmic connection $\nabla$ on a line bundle along an irreducible component $\Y$ of its polar divisor is the number
\[
\res_{\Y}(\nabla) \in \CC
\]
defined by choosing any local trivialization of the bundle, writing the connection in the form $\nabla = \dd + \alpha$ for a closed logarithmic one-form $\alpha$, and setting $\res_{\Y}(\nabla):=\res_{\Y}(\alpha)$.  The result is independent of the choice of trivialization, and in the case at hand it may be computed directly from the biresidues of the log symplectic form:

\begin{lemma}\label{lem:block-res}
Let $\Xdel$ be a log symplectic stratum of codimension $c$, suppose that $\X_{\Delta'}$ is an irreducible component of $\bdXdel$, and write the biresidue along the codimension-$(c+1)$ stratum $\X_{\Delta'}$ in block form
\[
B_{\Delta'} = \begin{pmatrix}
B_{\Delta} & A \\
- A^T & 0
\end{pmatrix}.
\]
Then the residue of the flat $(\CC^*)^{\Delta_0}$-connection $\nabla$ on the $\cN{\Xdel}$ is given by
\[
\res_{\X_{\Delta'}}(\cNdel,\nabla) = B_{\Delta}^{-1} A \in \CC^{\Delta_0}.
\]
In particular, the residue of the determinant connection is
\[
\res_{\X_{\Delta'}} (\det{\cNdel},\nabla) =  (1,\ldots,1)\cdot \Bdel^{-1}\cdot  A
\]
\end{lemma}

\begin{proof}
Note that since the residue is constant along $\X_{\Delta'}$, it is enough to calculate it at a single point $p \in \Y$ which does not intersect any other boundary components of $\Xdel$.  Choosing a Poisson tubular neighbourhood of $\Xdel$ in a neighbourhood of $p$ and a coordinate $z$ on $\Xdel$ vanishing on $\Y$, we may express the log symplectic structure in the form
\[
\sum_{i,j\in\Delta_0} (\Bdel)_{ij} \dlog{y_i}\wedge \dlog{y_j} + \sum_{i \in \Delta_0} \rbrac{a_i \dlog{z}} \wedge \dlog{y_i} +  \omega_\Delta
\]
where $\omega_\Delta$ is the log symplectic form on $\Xdel$, and where $a_i\in\CC$ are the biresidues forming the block $A = (a_i)_{i}$ in the biresidue matrix $B_{\Delta'}$ along $\Y$.  As explained in the proof of \autoref{prop:monodromy}, the $\CC^{\Delta_0}$-valued connection form on $\Xdel$ in this chart is given by $B_{\Delta}^{-1}(A_i)_i$, where now $A_i = a_i \dlog{z}$.  Taking the residue at $z=0$ gives the result.
\end{proof}

\begin{example}\label{ex:res23}
Suppose that $\Xdel$ is a log symplectic stratum of codimension two, and let $\X_{\Delta'} \subset \bdXdel$ be a boundary component.  Dually, we have a triangle $\Delta' \subset \dcX$ containing the edge $\Delta$. The biresidue along $\X_{\Delta'}$ can be represented graphically as
\[
B_{\Delta'} =\vcenter{\hbox{ \begin{tikzpicture}[scale=0.7,decoration={
    markings,
    mark=at position 0.5 with {\arrow{>}}}
    ] 
\draw[fill=gray] (-30:1) -- (90:1) -- (210:1) -- (-30:1);
\draw[postaction={decorate}] (90:1) -- (-30:1);
\draw[postaction={decorate}] (210:1)  -- (90:1);
\draw[postaction={decorate}] (-30:1) -- (210:1);
\draw (30:0.9) node {$a_1$};
\draw (150:0.9) node {$a_2$};
\draw (-90:0.8) node {$B_\Delta = b$};
\end{tikzpicture}}}
\]
or in matrix form as
\[
B_{\Delta'} = \begin{pmatrix}
B_{\Delta} & A \\ -A^T & 0
\end{pmatrix} = \begin{pmatrix}
0 & b & -a_1 \\
-b & 0 & a_2 \\
a_1 & -a_2 & 0
\end{pmatrix}.
\]
Then $b \ne 0$, and the residue of the determinant connection is given by
\[
\res_{\X_{\Delta'}}(\det \cNdel,\nabla) = -\frac{a_1+a_2}{b}.
\]
Note that this specific expression for the residue depends on the chosen orientation for the edges. 
\end{example}

\subsubsection{Resonance and cohomology}
The residue controls the local behaviour of the flat sections of $(\det\cNdel,\nabla)$, as follows: if $y$ is a local defining equation for an irreducible component $\Y \subset \bdXdel$, then on the open part $\Y^\circ\subset \Y$ where no other boundary components meet, a local flat section has the form $y^{-r}s$ where $s \in \det \cNdel$ is a local trivialization and $r = \res_\Y \nabla$ is the residue.  Thus the local monodromy is trivial if and only if $r \in \ZZ$, in which case the section is meromorphic of order $r$ along $\Y$.  In particular, the section extends to  a holomorphic section of $\det\cNdel$ over  $\Y^\circ$ if and only if $r \in \ZZ_{\le 0} $.  There is a standard name for the case $r \in \ZZ_{>0}$, where the local monodromy is trivial but the flat section has a pole:

\begin{definition}
A point $p \in \bdXdel$ is \defn{resonant} for $\nabla$ if for every irreducible component $\Y \subset \bdXdel$ containing $p$, the residue along $\Y$ is a positive integer:
\[
\res_\Y(\nabla) \in \ZZ_{>0}
\]
We say that a log symplectic stratum $\Xdel \subset \X$ is \defn{nonresonant} if no point of $\bdXdel$ is resonant for the corresponding logarithmic flat connection $\nabla$ on $\det \cN{\Xdel}$, i.e., $\res_\Y(\nabla) \notin \ZZ_{>0}$ for any irreducible component $\Y$ of $\bdXdel$.
\end{definition}

The computation of the logarithmic de Rham cohomology is simplest in the nonresonant case.  Applying \cite[Corollaire 6.10]{Deligne1970}, we obtain the following computation of the composition factor $\grdel$ of the Lichnerowicz complex:
\begin{proposition}\label{prop:nonres-cohlgy}If $\Xdel$ is nonresonant, then  the hypercohomology of $\grdel$ is given by the ordinary sheaf cohomology of the local system $\sLdel$ on the characteristic symplectic leaf $\Xodel$:
\[
\coH{\grdel} \cong \coH{\Xodel;\sLdel}.
\]
In fact, there is a canonical isomorphism
\[
\grdel \cong Rj_* \sLdel \in \dcat{\Xdel}
\]
in the constructible derived category, where $j : \Xodel \hookrightarrow \Xdel$ is the inclusion.  
\end{proposition}

In the presence of resonance, the computation of the cohomology is slightly more involved, but still essentially topological, as follows.  Let us denote by $\Xdelr \subset \Xdel$ the open set obtained as the union of $\Xodel$ and the set of boundary points where $\nabla$  is resonant; equivalently $\Xdelr$ is the complement of the nonresonant boundary components in $\Xdel$. Let $\bdXdelr := \bdXdel \cap \Xdelr$ be its boundary divisor.  Since $\sLdel$ has trivial monodromy around $\bdXdelr$, it extends canonically to a rank-one local system $\tsLdel$ on $\Xdelr$. However, while the sections of $\sLdel$ extend to nonvanishing sections of the sheaf $\tsLdel$, the corresponding sections of $\det{\cN{\Xdel}}$ have poles and therefore they do not extend to sections of $\scoH[0]{\grdel}$ along $\Xdelr$.  This has the effect of introducing relative cohomology:
\begin{proposition}\label{prop:res-cohlgy}
For an arbitrary log symplectic stratum $\Xdel$, the hypercohomology  of $\grdel$ is given by the relative cohomology
\[
\coH{\Xdel;\grdel} \cong \coH{\Xdel,\bdXdelr;\tsLdel}.
\]
In fact, there is a canonical isomorphism
\[
\grdel \cong Rj_{2*}(j_{1!}\sLdel) \in \dcat{\X}
\]
in the derived category, where we have factored the inclusion $j : \Xodel \hookrightarrow \Xdel$ as a composition of open inclusions
\[
\xymatrix{
\Xodel \ar@{^(->}[r]^-{j_{1}} & \Xdelr \ar@{^(->}[r]^-{j_2} & \Xdel.
}
\]
\end{proposition}

\begin{proof}[Sketch of proof]
When all boundary components are resonant, so that $j = j_1$ and $j_2 = \id$, this statement follows from \autoref{prop:nonres-cohlgy} by Verdier duality, as explained in \cite[Remark 1.8]{Esnault1986}.  The general case follows by observing that near a point on an arbitrary stratum of $\Xdel$, one can find a holomorphic gauge transformation taking the connection to one of the form $\nabla = \dd + \sum a_i \dlog{y_i}$ where $a_i \in \CC$ are the residues.  In particular, the triple $(\Xdel,\bdXdel,\nabla)$ decomposes locally as a Cartesian product of a purely nonresonant connection and a purely nonresonant connection, so that the cohomology sheaves of the complex $\grdel$ may be computed using the K\"unneth formula, giving the stated result. 
\end{proof}

\begin{definition}\label{def:nonres-push}
We call the complex
\[
\sLdelnr := Rj_{2*}(j_{1!}\sLdel) \in \dcat{\Xdel}
\]
appearing in \autoref{prop:res-cohlgy} the \defn{nonresonant extension of $\sLdel$}.  In particular, when $\Xdel$ is nonresonant, $\sLdelnr = R\jdelf\sLdel$ is the usual derived
direct image of the local system $\sLdel$ along the inclusion $j_\Delta : \Xodel  \to \X$.
\end{definition}

\begin{example}[\runex]\label{ex:runex-strata-contrib}
For the log symplectic form \eqref{eq:3cpt-logform}, the open stratum $\Xo$ (corresponding to the empty simplex $\varnothing)$ is homotopic to a three-torus, so that 
\[
\coH{\grpol{\varnothing}} = \coH{\Xo;\CC} \cong \coH{(S^1)^3;\CC} \cong \CC[e_1,e_2,e_3]
\]
where $e_i = [\pis(\dlog{y_i})]$ are the classes of the log Hamiltonian vector fields of the coordinate functions.

Meanwhile,  a codimension two stratum $\Xodel$ is symplectic if and only if the corresponding biresidue is nonzero.  Consider, without loss of generality, the stratum with biresidue $b_1$.  This stratum is isomorphic to $\CC^*\times \CC$ with coordinates $(y_1,z)$.  If $b_1 \ne 0$, it carries a local system $\sLdel \subset \det{\cN{\Xodel}}$ generated by the (possibly multivalued) bivector $y_1^m \cvf{y_2}\wedge\cvf{y_3}$ where $m = (b_2+b_3)/b_1$ is minus the residue as in \autoref{ex:res23}.  The monodromy of the local system $\sLdel$ around a positively oriented generating loop is therefore $\exp(2\pi\mathrm{i}m)$. In particular the monodromy is trivial if and only if $m \in \ZZ$ and the stratum is resonant if and only if  $m \in \ZZ_{<0}$.  When the monodromy is nontrivial, the cohomology of $\sLdel$ vanishes.  Similarly, when the stratum is resonant, the relative cohomology from \autoref{prop:res-cohlgy} vanishes, since the inclusion $\CC \cong  \bdXdel =\bdXdelr \hookrightarrow \Xdelr = \Xdel  \cong \CC^2$ is a homotopy equivalence.  Finally when $m \in \ZZ_{\ge 0}$ we obtain the cohomology of $\CC^*\times \CC \sim S^1$ with coefficients in a trivial local system. It is straightforward to determine the generators explicitly, which gives
\[
\Hpi{\grpol{\Delta}} = \coH{\Xdel;\sLdelnr} \cong \begin{cases}
\coH{S^1;\CC} \cong \CC f_1 \oplus \CC f_1e_1 & m \in \ZZ_{\ge 0} \\
0 & \textrm{otherwise}
\end{cases}
\]
where $f_1 = [y_1^m \cvf{y_2}\wedge\cvf{y_3}]$.  The remaining strata are analyzed similarly, by cyclically permuting the variables. 
\end{example}

\subsection{Holonomicity and total cohomology}
\label{sec:holonomic}
We now assemble the results of the previous subsections to deduce a complete calculation of the Poisson cohomology, under natural nondegeneracy and nonresonance assumptions.

The relevant nondegeneracy condition is the notion of holonomicity for a (possibly non-normal crossings) log symplectic manifold $(\X,\pi)$, introduced in \cite{Pym2018}; it is a condition on the germs of $\pi$ at each point $p \in \X$.  Formally, $(\X,\pi)$ is \defn{holonomic} if the complex of D-modules $\der{\X}\otimes \sD{\X}$ induced by the Lichnerowicz complex is holonomic in the usual sense.  In \emph{op.~cit.}~it was shown that the conormal bundle of a symplectic leaf is contained in the singular support of $\der{\X}\otimes \sD{\X}$ if and only if it is characteristic in the sense of \autoref{def:char-leaves}, and furthermore the following chain of implications was explained:
\begin{align*}
(\X,\pi) \textrm{ is holonomic } & \implies (\der{\X},\dpi)[\dim\X] \textrm{ is a perverse sheaf} \\
&\implies \textrm{the stalk cohomology }\scoH{\der{\X},\dpi}_p \textrm{ is finite-} \\
&\qquad\ \  \textrm{dimensional  for all }p \in \X \\
&\implies (\X,\pi) \textrm{ has only finitely many characteristic }\\
&\qquad\ \ \textrm{symplectic leaves near any point} 
\end{align*}
We remark that the first implication follows from standard facts about D-modules, namely Kashiwara's constructibility theorem and Roos' bound on the projective dimension of holonomic modules, and meanwhile the second implication is true by definition. Thus the novel ingredient from Poisson geometry is the third implication.

We expect that these statements essentially characterize holonomicity, so that the implications above are in fact biconditional:
\begin{conjecture}\label{conj:char-holonomic}
If $(\X,\pi)$ has only finitely many characteristic symplectic leaves near any point, then $(\X,\pi)$ is holonomic.  In this case, the Lichnerowicz complex $(\der{\X},\dpi)$ has a finite exhaustive filtration whose composition factors are products of indecomposable extensions of finite-rank local systems on characteristic symplectic leaves.
\end{conjecture}

\autoref{thm:hol-complex} below confirms this conjecture in the case when the boundary divisor $\bdX$ of the log symplectic form is normal crossings.  If we further assume that all characteristic leaves are nonresonant, it gives a complete calculation of the Poisson cohomology in topological terms. 

 To state the result precisely, we introduce some notation: for even integer $k \ge 0$, let $\Kleaf[k]{\pi}$ denote the set of all $(k-1)$-simplices $\Delta \subset \dcX$ such that the corresponding submanifold $\Xdel \subset \X$ is log symplectic, i.e., the biresidue $B_\Delta$ is a nondegenerate skew form.
\begin{theorem}\label{thm:hol-complex}
Let $(\X,\bdX,\pi)$ be a normal crossings log symplectic manifold.  Then the following statements hold.
\begin{enumerate}
\item $(\X,\pi)$ is holonomic at a point $p \in \X$ if and only if $p$ has a neighbourhood intersecting only finitely many characteristic symplectic leaves.  Equivalently, for every stratum $\Xodel$ of odd codimension such that $p \in \Xdel$, the constant vector $(1,\ldots,1)$ is not in the image of $\Bdel$.
\item If $\pi$ is holonomic, then we have a canonical isomorphism in $\dcat{\X}$:
\[
\gr^\W_k( \der{\X},\dpi) \cong  \begin{cases} 
0 & k \textrm{ is odd} \\
\bigoplus_{\Delta \in \Kleaf[k]{\pi}} \sLdelnr[-k] & k  \textrm{ is even}.
\end{cases}
\]
\item If $\pi$ is holonomic, and all characteristic symplectic leaves are nonresonant, then the weight filtration splits canonically in $\dcat{\X}$:
\[
\der{\X} \cong \gr^\W \der{\X} \cong \bigoplus_{k \ge 0} \bigoplus_{\Delta \in \Kleaf[k]{\pi}} R\jdelf \sLdel[-k] 
\]
where $\jdel : \Xodel \to \X$ is the inclusion.
\end{enumerate}
\end{theorem}

\begin{proof}
  That the local finiteness of the set of characteristic leaves is necessary for holonomicity was proven in \cite{Pym2018} as stated above.  To see that this is equivalent to the statement about biresidues along odd-codimension strata, note that the number of characteristic leaves is locally infinite if and only if there exists a stratum whose Poisson structure is degenerate and whose leaves are characteristic. If such a stratum has even codimension, then by \autoref{cor:incident-nonhol} there is also a stratum of odd codimension with the same property.  Hence for local finiteness, it is enough to check that the leaves in all strata of odd codimension are non-characteristic, which reduces to the stated condition on biresidues by \autoref{prop:res-adapt}.
  
We now prove that the local finiteness  implies holonomicity. Since the problem is local,  we may assume without loss of generality that $\pi$ has only finitely many characteristic leaves globally, not just in a neighbourhood of some point. Consider the complex of $\sD{\X}$-modules $\cMpi := \der{\X} \otimes \sD{\X}$.  It inherits a weight filtration $\W_\bullet \cMpi$ from that of $\der{\X}$, and since $\sD{\X}$ is flat over $\cO{\X}$, we have $\gr^\W \cMpi \cong (\gr^\W\der{\X})\otimes \sD{\X}$ as complexes of $\sD{\X}$-modules.  Since extensions of holonomic D-modules are holonomic, by an induction on weight it suffices to show that the complex $\gr^\W \der{\X} \otimes \sD{\X}$ is holonomic.  From the description of $\gr^\W\der{\X}$ and the fact that the direct image preserves holonomicity, it suffices to prove that for every $k \ge 0$ and every $(k-1)$-simplex $\Delta$,  the complex $\grdel \otimes \sD{\Xdel}$ is holonomic.

By \autoref{prop:coh-adapt} there are two possibilities: either all symplectic leaves in $\Xodel$ are non-characteristic, in which case $\grdel  \cong 0$, or all symplectic leaves in $\Xodel$ are characteristic, and hence there must be at most finitely many of them by assumption.  Since $\Xodel$ is smooth and connected, the only possibility is that $\Xodel$ itself is symplectic, in which case $\Xdel$ is log symplectic by \autoref{lem:bires-nondegen}, and $\grdel$ is a logarithmic de Rham complex as in \autoref{sec:log-dR-strata}, and is therefore holonomic.  In this way we establish both the holonomicity of $\der{\X}$ and the description of the associated graded (parts 1 and 2 of the theorem).

For part 3, suppose that all characteristic symplectic leaves are nonresonant, and consider the open set $\U_k \subset \X$ obtained by removing all strata of codimension greater than $k$. Let $j : \U_k \to \X$ be the inclusion.  We will prove that the natural composition
\[
\W_k\der{\X} \to \der{\X} \to Rj_*j^*\der{\X} \cong Rj_* \der{\U_k}
\]
is a quasi-isomorphism; hence the inclusion of $\W_k\der{\X}$ is split by the map $\der{\X} \to Rj_*\der{\U_k} \cong \W_k\der{\X}$.  Since $\W_k\der{\U_k} = \der{\U_k}$, it suffices to show that the adjunction map
\[
\W_k\der{\X} \to Rj_*j^*\W_k\der{\X}
\]
is a quasi-isomorphism, and by induction on $k$ this reduces to showing that
\[
\gr_k\der \X \cong Rj_*j^*\gr_k \der{\X},
\]
or equivalently that
\[
\sLdelnr \cong Rj_*j^*\sLdelnr
\]
for all $\Delta \in \Kleaf[k]{\pi}$. But $j^*\sLdelnr = \sLdel$ by definition, and hence this is equivalent to the statement that $\sLdelnr = Rj_*\sLdel$, i.~e., the stratum is nonresonant.
\end{proof}

\begin{remark}\label{rmk:weight-split}
The proof of part 3 shows, more generally, that if all characteristic leaves of codimension $\le k$ are nonresonant, then $\W_k\der{\X}$ is a direct summand in $\der{\X}$ in $\dcat{\X}$, with projection $\der{\X} \to \W_k\der{\X}$ induced by the restriction to the open set $\X \setminus \bdX[k+1]$. 
\end{remark}

Taking hypercohomology and applying the theorem, we obtain the following information about the Poisson cohomology $\Hpi{\X}$ with its induced weight filtration
\[
\W_k \Hpi{\X} := \img\rbrac{{\coH{\W_k\der{\X}} \to \Hpi{\X}}}.
\]

\begin{corollary}\label{cor:htpy-equiv}
Let $(\X,\pi)$ be a normal crossings holonomic Poisson manifold, and suppose that $\U\subset \X$ is an open set for which the inclusion $(\U,\partial\U) \subset (\X,\partial \X)$ is a homotopy equivalence of normal crossings divisors.  Then the restriction $\Hpi{\X}\to\Hpi{\U}$ is an isomorphism.
\end{corollary}

\begin{corollary}\label{cor:weight-split}
If $(\X,\pi)$ is holonomic and all characteristic symplectic leaves are nonresonant, there is a canonical isomorphism
\[
\W_j\Hpi{\X} \cong \bigoplus_{k \le j} \prod_{\Delta \in \Kleaf[k]{\pi}} \coH[\bullet-k]{\Xodel;\sLdel}.
\]
where the right hand side is the sum of the ordinary sheaf cohomology of local systems on characteristic symplectic leaves of codimension $\le j$.
\end{corollary}

Note that by \autoref{ex:tri-holonomic}, every log symplectic manifold with normal crossings degeneracy divisor is holonomic away from a codimension-three subset.  We therefore obtain the following result, valid even when $(\X,\pi)$ is not holonomic:
\begin{corollary}\label{cor:low-deg}
If $\pi$ is a log symplectic Poisson structure on the normal crossings divisor $(\X,\bdX)$, then
\begin{align}
\Hpi[k]{\X} \cong \begin{cases}
\coH[k]{\Xo;\CC}  & k=0,1 \\
\coH[2]{\Xo;\CC} \oplus \prod_{\Delta \in \Kleaf[2]{\pi}} \coH[0]{\Xdel;\sLdelnr} &  k=2
\end{cases} \label{eq:hol-coh}
\end{align}
\end{corollary}

\begin{example}[\runex]\label{ex:runex-cohomology}
Let $\omega$ be the log symplectic form \eqref{eq:3cpt-logform} on $\X = \CC^4$ in our running example.  The characteristic leaves are completely determined by \autoref{ex:tri-holonomic}, and in particular we see that since $b_0+b_1+b_2\ne0$, there are only finitely many characteristic leaves.  Hence this Poisson structure is holonomic, with no contributions from strata of codimension greater than two. Thus the weight filtration splits, giving an isomorphism $\Hpi{\X} \cong \coH{\Xo;\CC} \oplus \prod_{\Delta \in \dc[2]{\X}} \coH{\Xdel;\sLdelnr}[-2]$, where the individual summands are computed in \autoref{ex:runex-strata-contrib}.  We conclude that the Poincar\'e polynomial of the Poisson cohomology $\Hpi{\X}$ is given by $P(t) = (1+t)^3 + kt^2(1+t)$ where $k$ is the number of nonresonant codimension-two strata with trivial monodromy.  
\end{example}

\begin{example} 
Consider the following log symplectic form on $\X=\mathbb{C}^{n+k}$:
\[
\omega
=\sum_{1\le i<j\le n} B_{ij}\dlog{y_i}\wedge \dlog{y_j} + \sum_{i=1}^n\sum_{j=1}^k a_{ij} \dd p_j \wedge  \dlog{y_i} + \sum_{1\le i<j\le k} c_{ij} \dd p_i \wedge  \dd p_j
\]
Assume that $\omega$ is holonomic.  The simplices $\Delta$ in the dual complex are in bijection with subsets $\Delta_0 \subset \{1,2,\ldots,n\}$; the corresponding subvarieties are given by $\X_\Delta=\set{y_i=0}{i\in \Delta_0}$.  The weight filtration may be split using the action of the torus $(\CC^*)^n$ which rescales the variables $y_i$.  Using \autoref{prop:res-cohlgy}, and that all strata are homotopic to tori, one can check that the summand in $\Hpi{\X}$ corresponding to a stratum $\Xodel$ is nonzero if and only if $\Xodel$ is symplectic and the corresponding logarithmic connection is nonresonant, with trivial monodromy.  By \autoref{lem:bires-nondegen}, \autoref{prop:monodromy} and \autoref{lem:block-res}, this is equivalent to requiring that $B_\Delta$ is invertible, and that $t_k :=-\sum_{i,j\in \Delta_0} (B_\Delta^{-1})_{ij} B_{jk}$ is a nonnegative integer for each $k\notin \Delta_0$.   We therefore obtain
\begin{equation}\label{eq:poisson_coh_Cn}
\Hpi{\X} \cong \bigoplus_{\Delta} \coH{\Xodel}[-\,\mathrm{codim}\, \Xodel]
\end{equation}
where the sum is over simplices $\Delta$ satisfying the linear-algebraic conditions above, and $\coH{\Xodel} \cong \wedge^\bullet \CC^{n-\# \Delta_0}$ is the cohomology of a torus.

In the special case when $\pi$ is in the stable normal form \eqref{eq:local-magnetic}, the Poisson cohomology classes in the $\Delta$ summand in \eqref{eq:poisson_coh_Cn} have the following representatives:
\[
\tau \wedge \exp\left( \sum_{i\in \Delta_0} \alpha_i p_i\right)\left(\prod_{k\notin \Delta_0} y_k^{t_k}\right)  \bigwedge_{i\in \Delta_0} \partial_{y_i}  ,
\]
where $\alpha_i = \sum_{j\in \Delta_0}(B_\Delta^{-1})_{ji}$, $i\in I$, and $\tau$ is a constant polyvector field lying in the algebra $\wedge^\bullet \W_{\Delta}$, where $\W_{\Delta}\subset  \mathrm{span}\{\cvf{p_1},\ldots,\cvf{p_n}\} \cong \mathbb{C}^n$ is any fixed complement to the column space of the matrix $(B_{ij})_{1 \le i \le n,j \in \Delta_0}$. 
\end{example}

\section{First order deformations}

\label{sec:smoothable}

\subsection{Smoothable strata}

We now turn to the problem of describing the deformations of a normal crossings log symplectic manifold $(\X,\bdX,\pi)$.  The isomorphism classes of first-order deformations of $(\X,\pi)$, are parameterized by the second Poisson cohomology $\Hpi[2]{\X}$. Note that \autoref{cor:low-deg}
gives a decomposition
\begin{align*}
\Hpi[2]{\X} \cong \coH[2]{\Xo;\CC} \oplus \prod_{\edge \in \Kleaf[2]{\pi}} \coH[0]{\Xe; \sLenr} ,
\end{align*}
where the product is over all characteristic symplectic leaves of codimension two; these correspond to edges $\edge$ in the dual complex $\dcX$ whose biresidue is nonzero and each contribute a factor given by the vector space $\coH{\Xe;\sLenr} \subset \coH[0]{\Xoe;\sL}$ of cohomology of the rank-one local system $\sLe$ relative to the resonant boundary of the stratum $\Xe$ as in \autoref{prop:res-cohlgy} and \autoref{def:nonres-push}. This decomposition has the following deformation-theoretic interpretation:
\begin{itemize}
\item The summand $\coH[2]{\Xo;\CC}$ corresponds to deformations of $(\X,\pi)$ for which $\bdX$ deforms locally trivially (i.e.~deformations that remain normal crossings with locally unchanged number of components); 
\item For each $\edge \in \Kleaf[2]{\pi}$, the space $\coH[0]{ \Xe;\sLenr } $ has dimension at most one; it corresponds to first-order deformations in which the singularities of $\bdX$ along the codimension-two stratum $\Xoe$ are smoothed. 
\end{itemize}

We therefore introduce the following

\begin{definition}\label{def:smoothable}
A \defn{smoothable stratum} is a codimension two characteristic symplectic leaf $\Xoe \subset\X$ such that $\coH[0]{ \Xe;\sLenr }  \ne 0$, or equivalently $\dim \coH[0]{\Xe;\sLenr} =1$. We refer to the 1-simplices of $\dcX$ dual to smoothable strata as \defn{smoothable edges}.  We will indicate smoothable edges pictorially by colouring them blue as in \eqref{eq:triangle-bires} below. 
\end{definition}

Rephrasing the above in light of this definition, we have the following calculation of the first-order deformation space.
\begin{proposition}\label{prop:HP2-decomp}
We have a canonical isomorphism
\[
\Hpi[2]{\X} \cong \coH[2]{\Xo;\CC} \oplus \prod_{\substack{\textrm{smoothable}\\\textrm{edges }\edge }} \coH[0]{ \Xoe;\sLe}.
\]
In particular, when the second Betti number $b_2(\Xo)$ and the number of smoothable strata are finite, we have
\[
\dim \Hpi[2]{\X} = b_2(\Xo) + (\#\textrm{ of smoothable strata}).
\]
\end{proposition}

Observe that  an edge $\edge \in \Kleaf[2]{\pi}$ is smoothable  if and only if it is nonresonant and the local system $\sLe$ has trivial monodromy.  In particular, all residues of the logarithmic connection on $\Xe$ are nonpositive integers.  Applying \autoref{cor:no-monodromy} and \autoref{ex:res23}, we obtain the following characterization of smoothability:
\begin{proposition}\label{prop:smoothable-bires}
Let $\edge$ be an oriented edge of $\dcX$.  Then the codimension-two stratum $\Xoe$ is smoothable with respect to the log symplectic form $\omega$ if and only if the following statements hold:
\begin{enumerate}
\item\label{cnd:nondegen} (Nondegeneracy) The biresidue along $\Xe$  is nonzero:
\begin{align}
\Be := \res_{e}(\omega)  \ne 0 \in \CC \label{eq:smoothable-nondegen}
\end{align}
\item\label{cnd:nonres} (Regularity of flat sections along the boundary) For every triangle $\Delta$ containing $e$, with biresidues
\begin{align}
\Bdel =\vcenter{\hbox{ \begin{tikzpicture}[scale=0.7,decoration={
    markings,
    mark=at position 0.5 with {\arrow{>}}}
    ] 
\draw[fill=gray] (-30:1) -- (210:1) -- (90:1) -- (-30:1);
\draw[postaction={decorate}] (-30:1) -- (90:1);
\draw[postaction={decorate}] (90:1)  -- (210:1);
\draw[thick,blue,postaction={decorate}]  (210:1) -- (-30:1);
\draw (30:0.9) node {$a_1$};
\draw (150:0.9) node {$a_2$};
\draw[blue] (-90:0.8) node {$B_e$};
\end{tikzpicture}}} \label{eq:triangle-bires}
\end{align}
we have the integrality condition
\begin{align}
a_1+a_2 \in  \Be\cdot \ZZ_{\ge0}.\label{eq:bires-smoothable}
\end{align} 
\item\label{cnd:nomon} (No monodromy) For some (and hence any) collection of loops $\{\torloop^j\}_{j \in \I}$ in $\Xoe$ that generates $\Hlgy[1]{\Xoe;\ZZ}$, we have the integrality condition
\begin{align}
\int_{\torcyc_1^j-\torcyc_2^j}\omega \in (2\pi\iu)^2 \Be \cdot \ZZ \qquad \textrm{for all }j\in \I. \label{eq:smooth-monodromy}
\end{align}
where $\torcyc^j_{i}$ are the associated tori in $\Xo$ as in \autoref{sec:monodromy}.
\end{enumerate}
\noindent Moreover, if $\Hlgy[1]{\Xe;\ZZ}=0$, then the  second condition implies the third.
\end{proposition}
For the last claim in the proposition, a Mayer-Vietoris argument shows that $\Hlgy[1]{\Xe;\ZZ}=0$ implies that every element in $\Hlgy[1]{\Xoe; \ZZ}$ is homologous to a $\ZZ$-linear combination of several small loops in $\Xoe$ around its boundary components.

Note that if $\edge$ is an edge of $\dcX$ corresponding to any nondegenerate codimension two stratum (i.e.~characteristic leaf of codimension two), of $(\X,\omega)$, and $\Delta \supset \edge$ is a triangle with biresidue as depicted in \eqref{eq:triangle-bires}, then the quantity
\[
m(\edge,\Delta) := \frac{a_1+a_2}{\Be} \in \CC
\]
is independent of the orientation of $\Delta$; when it is an integer it is exactly the order of vanishing of a flat section of $\sLe$ along the component of $\bdXe$ corresponding to $\Delta$.  We refer to $m(e,\Delta)$ as the \defn{order of the pair $(\edge,\Delta)$} and depict it graphically by inscribing $m$ in the vertex of $\Delta$ opposite $e$, as follows:
\[
\Bdel =\vcenter{\hbox{ \begin{tikzpicture}[scale=0.7,decoration={
    markings,
    mark=at position 0.5 with {\arrow{>}}}
    ] 
\draw[fill=gray] (-30:1) -- (210:1) -- (90:1) -- (-30:1);
\draw[postaction={decorate}] (-30:1) -- (90:1);
\draw[postaction={decorate}] (90:1)  -- (210:1);
\draw[thick,blue,postaction={decorate}]  (210:1) -- (-30:1);
\draw (30:0.9) node {$a_1$};
\draw (150:0.9) node {$a_2$};
\draw[blue] (-90:0.8) node {$\Be$};
\draw[blue] (90:0.5) node {$\scriptstyle m$};
\end{tikzpicture}}}
\]
With this notation understood, the first Chern class of $\det{\cN{\Xe}}$ for a smoothable stratum $\Xe$ may be computed in two ways.  On the one hand, it is the divisor class of a global flat section.  On the other hand, by the adjunction formula, it is the pullback of the classes of the components of $\bdX$ that intersect along $\Xe$.  This immediately yields the following. 
\begin{lemma}\label{lem:smooth-c1}
Suppose that $\Xoe$ is a smoothable stratum, and let $\Y_0,\Y_1 \subset \bdX$ be the boundary components that intersect along $\Xoe$.  Then we having the following identity in $\coH[2]{\Xe;\ZZ}$:
\begin{align}
\sum_{\Delta \supset \edge} m(e,\Delta) [\Xdel] = i_e^*( [\Y_0] + [\Y_1])  \label{eq:smooth-c1}
\end{align}
\end{lemma}

\begin{example}\label{ex:P2n-single-edge}
Let $\X = \PP^{2n}$ and consider the anticanonical divisor $\partial \PP^{2n}$ given by the union of the coordinate hyperplanes.  Then all strata are linear subspaces of $\X$ and the dual complex has a unique simplex for every possible tuple of vertices. In particular, every codimension-two stratum $\Xe$ is the intersection of two linear subspaces, and each of its boundary components is also linear.  Hence all fundamental classes appearing in \eqref{eq:smooth-c1} are equal to the hyperplane class $H$, and we obtain the simpler numerical formula
\begin{align}
\sum_{\Delta \supset \edge} \orddel = 2
\end{align}
Therefore if $\edge$ is a smoothable edge, there are two possibilities: either $\edge$ has two distinct incident triangles of order one, or a single incident triangle of order two, and all remaining incident triangles have order zero.  We simplify the notation in this case by colouring the opposite vertices where the orders are nonzero; see \autoref{fig:P4-one-edge} for the case $n=2$.
\end{example}

\begin{figure}
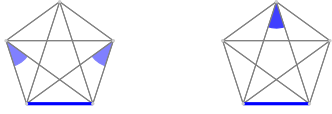

\[
\sagepentagon[0.5]{\smoothedge{v3}{v4}{v0}{v2}} \qquad \qquad
\sagepentagon[0.5]{\smoothedge{v3}{v4}{v1}{v1}}
\]
\caption{The possible orders of a smoothable edge for $\PP^4$, up to isomorphism.} \label{fig:P4-one-edge}
\end{figure} 

\subsection{Linear arrangements and smoothing diagrams}
\label{sec:arrangement}
Note that since the biresidues of $\omega$ are $(2 \pi \iu)^{-2}$ times the integral of $\omega$ over corresponding tori in $\Xo$, each of the conditions \eqref{eq:smoothable-nondegen}, \eqref{eq:bires-smoothable} and \eqref{eq:smooth-monodromy} for a stratum to be smoothable amounts to a $\ZZ$-linear relation on the periods of $\omega$ over classes in $\Hlgy[2]{\Xo;\ZZ}$.  In particular, they depend only on the cohomology class of $\omega$, and make sense even for degenerate forms.  This allows us to make the following definition.

\begin{definition}\label{def:lin-arr}
For an edge $\edge \subset \dcX$, we denote by
\[
\Se\subset \coH[2]{\Xo;\CC}
\]
the set of cohomology classes satisfying the conditions 1) through 3) of \autoref{prop:smoothable-bires}, and the identity \eqref{eq:smooth-c1} from \autoref{lem:smooth-c1}.
\end{definition}

Thus, by construction, a log symplectic form on $(\X,\bdX)$ has $\edge$ as a smoothable edge if and only if $[\omega] \in \Se$.  Note that $\Se$ is a smooth locally closed subset of $\coH[2]{\Xo;\CC}$, given by intersecting the hyperplane complement $\Be \ne 0$ with a countable union of disjoint linear subspaces in $\coH[2]{\Xo;\CC}$ defined by $\ZZ$-linear relations between periods.  Note furthermore that $\Se$ decomposes as a disjoint union $\Se = \sqcup_{m} \Semu$ of locally closed submanifolds in which the orders $\orddel \in \ZZ_{\ge 0}$ of all incident triangles take on fixed values.  

The collection $\{\Semu\}_{\edge,m}$ therefore defines a countable arrangement of linearly embedded submanifolds of $\coH[2]{\Xo;\CC}$.  By considering the mutual intersections of such submanifolds, we obtain a stratification of $\coH[2]{\Xo;\CC}$.  For each stratum $\W$ we obtain a collection $\Gamma$ of edges of $\dcX$, by declaring that $\edge \in \Gamma$ if and only if $\W \subset \Se$, so that
\begin{align}
\W = \cap_{\edge \in \Gamma} \Se \setminus \cup_{\edge \notin \Gamma} \Se. \label{eq:linear-statum}
\end{align}
Further remembering the order $\orddel$ for each incident triangle, we obtain from each stratum of the arrangement the following combinatorial datum:
\begin{definition}
A \defn{pre-smoothing diagram for $(\X,\bdX)$} is pair $(\Gamma,m)$ consisting of a collection $\Gamma$ of edges in the dual complex, and an assignment of a nonnegative integer $\orddel \in \ZZ_{\ge 0}$ to each pair $(\edge,\Delta)$ of an edge $\edge \in \Gamma$ and a triangle $\Delta \supset \edge$, such that the equation \eqref{eq:smooth-c1} is satisfied. 
\end{definition}

\begin{definition}
A pre-smoothing diagram $(\Gamma,m)$ is a \defn{smoothing diagram} if there exists a log symplectic form $\omega$ such that $[\omega]\in\coH[2]{\Xo;\CC}$ lies on a stratum whose pre-smoothing diagram is $(\Gamma,m)$. 
\end{definition}

\begin{remark}
The number of smoothable strata of a log symplectic form $\omega$ is the number of edges in its corresponding smoothing diagram. 
\end{remark}

\begin{remark}
  A pre-smoothing diagram may fail to be a smoothing diagram in two ways.  Firstly, the set $\W$ from \eqref{eq:linear-statum} may be empty due to linear dependence between the various subspaces $\Semu$ indexed by $(\Gamma,m)$.  Secondly, even if this set is nonempty, it may happen that $\W$ does not contain the class of a global nondegenerate logarithmic form.  
\end{remark}

\begin{remark}\label{rmk:subdiagrams}
There is a natural partial order on smoothing diagrams, defined by $\Gamma \le \Gamma'$ if $\Gamma$ is a smoothing sub-diagram of $\Gamma'$, i.e.~it can be obtained from $\Gamma'$ by removing some edges, leaving the orders $\orddel$ of the remaining triangles unchanged.  Note that this partial order is dual to the poset of inclusions of closures of strata in the arrangement. 
\end{remark}

The classification of smoothing diagrams for a given normal crossings divisor $(\X,\bdX)$ seems to be a subtle problem.  However, there are several important constraints that will be useful in what follows.  The first concerns the relation between smoothability and holonomicity:

\begin{lemma}\label{lem:nonhol-res}
If  $\omega$ is a log symplectic form and $\Delta \subset \dcX$ is a triangle corresponding to a nonholonomic codimension three stratum, then any edge of $\Delta$ with nonzero biresidue is resonant. In particular, no edge of $\Delta$ can be smoothable.
\end{lemma}

\begin{proof}
Orient the triangle $\Delta$ cyclically so that the biresidues are given by complex numbers $b_1,b_2,b_3 \in \CC$.    Since $\Delta$ is nonholonomic, we have the identity $b_1+b_2+b_3 = 0$ by  \autoref{ex:tri-holonomic}.  If some biresidue, say $b_3$, is nonzero, then we have $\tfrac{b_1+b_2}{b_3} = -1$, so that the edge corresponding to $b_3$ is resonant, and is therefore not smoothable.
\end{proof}

The second constraint classifies the possible decorations of a triangle with multiple smoothable edges, along with the corresponding biresidues of the log symplectic structure:
\begin{lemma}\label{prop:smooth-2chain}
  Let $\omega$ be a log symplectic structure, and let $\Delta \subset \dcX$ be a triangle with two  edges corresponding to strata that are smoothable near $\Xodel$ with orders  $m, n\in\ZZ_{\ge 0}$. Assume without loss of generality that $m \le n$.  Then the biresidue $B_\Delta$ has the following form:
\begin{enumerate}
\item If $m = 0$, then the third edge is resonant, and $\Bdel$ is a scalar multiple of the following, for some $n \ge 0$:
\begin{align}
\ObsTri{n} := \vcenter{\hbox{ \begin{tikzpicture}[scale=0.7,decoration={
    markings,
    mark=at position 0.55 with {\arrow{>}}}
    ] 
\draw[fill=gray] (-30:1) -- (90:1) -- (210:1) -- (-30:1);
\draw[thick,blue,postaction={decorate}] (-30:1) -- (90:1) ;
\draw[thick,blue,postaction={decorate}]  (90:1) -- (210:1) ;
\draw[thick,red,postaction={decorate}] (210:1) -- (-30:1);
\draw[blue] (30:1.2) node {$n+1$};
\draw[blue] (150:1.2) node {$1$};
\draw[red](-90:0.9) node {$-1$};
\draw[blue] (-30:0.6) node {$\scriptstyle n$};
\draw[blue] (-150:0.6) node {$\scriptstyle 0$};
\draw[red] (90:0.5) node {$ \scriptstyle -n$};
\draw[red] (90:0.25) node {$ \scriptstyle -2$};
\end{tikzpicture}}} \label{eq:2chain-special}
\end{align} 
\item If $(m,n) \in \{ (1,2),(1,3),(1,5),(2,2),(2,5),(3,3) \}$, then the remaining edge is smoothable near $\Xodel$, and $\Bdel$ is a scalar multiple of the following:
\begin{align}
 \widetilde E_6  := \vcenter{\hbox{ \begin{tikzpicture}[scale=0.7,decoration={
    markings,
    mark=at position 0.5 with {\arrow{>}}}
    ] 
    \draw[fill=gray] (-30:1) -- (90:1) -- (210:1) -- (-30:1);
\draw[thick,blue,postaction={decorate}] (-30:1) -- (90:1);
\draw[thick,blue,postaction={decorate}] (90:1) -- (210:1) ;
\draw[thick,blue,postaction={decorate}] (210:1) -- (-30:1);
\draw[blue] (30:0.9) node {$1$};
\draw[blue] (150:0.9) node {$1$};
\draw[blue] (-90:0.9) node {$1$};
\draw[blue] (-30:0.6) node {$\scriptstyle2$};
\draw[blue] (-150:0.6) node {$\scriptstyle2$};
\draw[blue] (90:0.6) node {$\scriptstyle2$};
\end{tikzpicture}}}
&&
\widetilde E_7 := \vcenter{\hbox{ \begin{tikzpicture}[scale=0.7,decoration={
    markings,
    mark=at position 0.5 with {\arrow{>}}}
    ] 
    \draw[fill=gray] (-30:1) -- (90:1) -- (210:1) -- (-30:1);
\draw[thick,blue,postaction={decorate}] (-30:1) -- (90:1);
\draw[thick,blue,postaction={decorate}] (90:1) -- (210:1) ;
\draw[thick,blue,postaction={decorate}] (210:1) -- (-30:1);
\draw[blue] (30:0.9) node {$1$};
\draw[blue] (150:0.9) node {$1$};
\draw[blue] (-90:0.9) node {$2$};
\draw[blue] (-30:0.6) node {$\scriptstyle3$};
\draw[blue] (-150:0.6) node {$\scriptstyle3$};
\draw[blue] (90:0.6) node {$\scriptstyle1$};
\end{tikzpicture}}}
&&
\widetilde E_8 := \vcenter{\hbox{ \begin{tikzpicture}[scale=0.7,decoration={
    markings,
    mark=at position 0.5 with {\arrow{>}}}
    ] 
    \draw[fill=gray] (-30:1) -- (90:1) -- (210:1) -- (-30:1);
\draw[thick,blue,postaction={decorate}] (-30:1) -- (90:1);
\draw[thick,blue,postaction={decorate}] (90:1) -- (210:1) ;
\draw[thick,blue,postaction={decorate}] (210:1) -- (-30:1);
\draw[blue] (30:0.9) node {$1$};
\draw[blue] (150:0.9) node {$2$};
\draw[blue] (-90:0.9) node {$3$};
\draw[blue] (-30:0.6) node {$\scriptstyle 2 $};
\draw[blue] (-150:0.6) node {$\scriptstyle 5$};
\draw[blue] (90:0.6) node {$\scriptstyle 1$};
\end{tikzpicture}}}
\label{eq:simple-elliptic-degen}
\end{align}
\item Otherwise, the remaining edge is neither resonant nor smoothable near $\Xodel$, and $\Bdel$ is a scalar multiple of the following:
\begin{align}
T_{\infty,m+1,n+1} := \vcenter{\hbox{ \begin{tikzpicture}[scale=0.7,decoration={
    markings,
    mark=at position 0.55 with {\arrow{>}}}
    ] 
\draw[fill=gray] (-30:1) -- (90:1) -- (210:1) -- (-30:1);
\draw[thick,blue,postaction={decorate}] (-30:1) -- (90:1) ;
\draw[thick,blue,postaction={decorate}]  (90:1) -- (210:1) ;
\draw[postaction={decorate}] (210:1) -- (-30:1);
\draw[blue] (30:1.2) node {$n+1$};
\draw[blue] (150:1.2) node {$m+1$};
\draw(-90:0.9) node {$nm-1$};
\draw[blue] (-30:0.6) node {$\scriptstyle n$};
\draw[blue] (-150:0.6) node {$\scriptstyle m$};
\end{tikzpicture}}} \label{eq:2chain-bires}
\end{align}
\end{enumerate}
\end{lemma}
Note that in the diagrams, the orders placed at vertices are invariant under scaling the biresidues (and in fact they determine the biresidues up to simultaneous scaling).
\begin{proof}
The form \eqref{eq:2chain-bires} follows immediately by solving the linear equations on biresidues \eqref{eq:bires-smoothable} with the given order $m,n$ for the smoothable edges, and the condition for a resonance follows immediately by requiring the order $\frac{(n+1)+(m+1)}{nm-1}$ on the third edge to be a negative integer.

To treat the case of three smoothable edges, let us orient the triangle cyclically so that the biresidues are given by nonzero numbers $(b_1,b_2,b_3)$. By \autoref{lem:nonhol-res}, we have $b_1+b_2+b_3 \ne 0$, and hence by rescaling, we may assume without loss of generality that $b_1+b_2+b_3=1$. If all three edges are smoothable, we may apply \eqref{eq:bires-smoothable} to each of them, giving three linear equations of the form $b_{i-1}+b_{i+1}=m_ib_i$ 
where the indices are taken modulo three and each $m_i \in \ZZ_{\ge 0}$.  These linear equations are equivalent to the equations $\frac{1}{m_i+1}=b_i$, which imply that $\sum_{i=1}^3 \frac{1}{m_i+1} = 1$. Up to permutation, this equation has only three solutions for which $m_1,m_2,m_3 \in \mathbb{Z}_{\ge0}$, namely $(2,2,2)$, $(1,3,3)$ and $(1,2,5)$.
\end{proof}

\begin{example}[\runex]
We continue with the example of the divisor $y_1y_2y_3=0$ in $\CC^4$, whose dual complex is a triangle.  The open stratum $\Xo \cong (\CC^*)^3 \times \CC$ is homotopic to a three-torus, so that $\coH[2]{\Xo;\CC} \cong \CC^3$, with basis given by the biresidues of logarithmic two-forms.  The codimension-two submanifolds $\Xe$ are planes, hence contractible, so that the Chern class constraint is vacuous, and $\Hlgy[1]{\Xe;\ZZ}=0$.  Thus to determine the arrangement $\{\Se\}$ we simply need to solve the relations \eqref{eq:smoothable-nondegen} and \eqref{eq:bires-smoothable} for each of the three edges of the triangle. This gives three countable sequences of planes in $\CC^3$, one sequence for each edge. Projecting from $\CC^3$ to $\PP^2$, we arrive at the subspace arrangement illustrated in \autoref{fig:3cpt-line-arr} in the introduction. The possible smoothing diagrams with at least two smoothable edges are completely determined by \autoref{prop:smooth-2chain}; they correspond to the double and triple points in the subspace arrangement.
\end{example}

Recall that a \defn{3-horn} in $\dcX$ is a subcomplex given by three triangles arranged cyclically around a common vertex.  We have the following:

\begin{lemma}\label{lem:horn}
Let $\Sigma$ be any 3-horn of $\dcX$.  Then at most two of the edges in $\Sigma$ are smoothable.
\end{lemma}

\begin{proof}
The horn $\Sigma$ has exactly three edges, all emanating from the central vertex; we must show that they cannot all be smoothable.  Suppose to the contrary that all three were smoothable.  Then we would have the following graphical representation of the biresidues: \begin{center}
\begin{tikzpicture}[scale=1,decoration={
    markings,
    mark=at position 0.55 with {\arrow{>}}}
    ] 
\coordinate (A) at (-1,0) {};
\coordinate (B) at (0.5,-1) {};
\coordinate (C) at (0.9,0)  {};
\coordinate (D) at (0,1.5) {};
\draw[gray,fill,opacity=0.8] (A) -- (C) -- (D) -- (C);
\draw[fill=gray,opacity=0.2] (A) -- (B) -- (D) -- (A);
\draw[fill=gray,opacity=0.4] (C) -- (B) -- (D) -- (C);
\draw[blue,thick,postaction={decorate}] (D)--(A);
\draw[blue,thick,postaction={decorate}] (D)--(B);
\draw[blue,thick,postaction={decorate}] (D)--(C);
\draw[blue] (0.8,0.8) node {$b_1$};
\draw[blue] (-0.8,1) node {$b_2$};
\draw[blue] (-0.1,0.2) node {$b_3$};
\end{tikzpicture}
\end{center}
where each $b_i \in \CC$ is nonzero.  Applying \autoref{prop:smooth-2chain} to each face of the horn, we find that
\[
b_1 = \lambda b_2 = \lambda \lambda' b_3 = \lambda \lambda' \lambda '' b_1
\]
for some strictly negative rational numbers $\lambda, \lambda',\lambda''$.  Hence $b_1$ is a negative multiple of itself, a contradiction.
\end{proof}

\subsection{Example: projective space}
\label{sec:P2n-diagrams}
As an extended example, consider $\X = \PP^{2n}$ equipped with the toric boundary divisor $\partial \PP^{2n}$ given by the union of the coordinate hyperplanes.  The open stratum is a torus $\Xo \cong (\CC^*)^{2n}$.  All global logarithmic forms are invariant under the torus action, and in particular a global logarithmic two-form can be written uniquely in any affine toric chart $\CC^{2n} \subset \PP^{2n}$ with coordinates $y_1,\ldots,y_{2n}$ in the form
\[
\omega := \sum_{i<j} B_{ij} \dlog{y_i}\wedge\dlog{y_j}.
\]
where $B := (B_{ij})_{ij} \in \CC^{2n\times 2n}$ is the matrix of biresidues in this chart.  The form $\omega$ is log symplectic if and only if the matrix  $B$  is invertible.  In this way the strata $\bS_e$ are identified with Zariski open sets in linear subspaces of the space of skew-symmetric matrices of size $2n$, defined over $\ZZ$.

We denote the vertex of $\dc{\PP^{2n}}$ corresponding to the hyperplane $y_{i}=0$ by $(i)$ for $1 \le i \le 2n$, and use $(0)$ for the vertex corresponding to the hyperplane at infinity.  We denote the unique oriented edge from $i$ to $j$ by $(ij)$.  With this notation understood, the biresidues along strata $(0i)$ at infinity in the chart are determined from the matrix $B$ by the formula
\[
B_{0i} := - \sum_{j=1}^{2n} B_{ji}.
\]

The $1$-skeleton of $\dc{\PP^{2n}}$ is the complete graph on $2n+1$ vertices, and there is a unique triangle joining every edge to every opposite vertex.  Hence  a pre-smoothing diagram consists of a choice of a subgraph $\Gamma$ of the complete graph, and an assignment of an order $\ord$ to every vertex opposite and edge of $\Gamma$.      By \autoref{prop:smoothable-bires}, an edge $\edge = (ij)$ of $\Gamma$ is smoothable for some given $B$ if only if the following relations holds:
\begin{align}
 B_{ij} \neq 0, \qquad B_{ik}+B_{kj} =  \orddel B_{ij}\textrm{ for all }k \neq i,j\label{eqn:matrix-smoothings}
\end{align}
These conditions correspond to conditions \ref{cnd:nondegen} and \ref{cnd:nonres} of \autoref{prop:smoothable-bires}; the remaining condition \ref{cnd:nomon} is automatic since the closed strata $\Xe \cong \PP^{2n}$ are simply connected.  Therefore a pair $(\Gamma,\ord)$ defines a smoothing diagram if and only if the following hold
\begin{itemize}
\item There exists a nondegenerate matrix $B$ solving the relations \eqref{eqn:matrix-smoothings} for all edges $\edge = (ij)$ of $\Gamma$ and all triangles $\Delta = (ijk)$ containing $\edge$.
\item If $\edge = (ij)$ is any edge not lying in $\Gamma$ then either $B_{ij} = 0$ or there exists an index $k$ such that $B_{ij}^{-1}(B_{ik}+B_{kj}) \notin \ZZ_{\ge 0}$.
\end{itemize}
It would be interesting to give a purely combinatorial characterization of such pairs $(\Gamma,m)$.  While we do not have a complete description, we can identify several constraints, as in the following proposition.  It constrains things enough to make it straightforward to program a computer to produce a complete enumeration of the possible smoothing diagrams up to isomorphism.  We have done so for $n=2$; see \autoref{tab:confs_P4}.

\begin{table}
\caption{Classification of smoothing diagrams $(\Gamma,m)$ for $\mathbb{P}^4$, and the dimensions of the corresponding strata $\bS_{\Gamma,m} \subset \coH[2]{(\CC^*)^4;\CC}\cong \CC^6$.  Blue edges correspond to smoothable strata, dashed edges correspond to codimension-two strata with zero biresidue, and the orders are indicated by colouring the opposite angles when the order is nonzero as in \autoref{ex:P2n-single-edge} and \autoref{fig:P4-one-edge}.}
\label{tab:confs_P4}
\begin{center}
\begin{tabular}{ccc}
\hline
Smoothing diagram $(\Gamma,m)$ & $\dim(\textrm{stratum})$ \\ 
\hline
\sagepentagon{\smoothedge{v3}{v4}{v1}{v1} 
\smoothedge{v2}{v3}{v0}{v0} 
\smoothedge{v1}{v2}{v4}{v4} 
\smoothedge{v0}{v1}{v3}{v3} 
\smoothedge{v4}{v0}{v2}{v2} } 
\sagepentagon{\smoothedge{v3}{v4}{v0}{v2} 
\smoothedge{v2}{v3}{v1}{v4} 
\smoothedge{v1}{v2}{v3}{v0} 
\smoothedge{v0}{v1}{v2}{v4} 
\smoothedge{v4}{v0}{v1}{v3} 
\zeroedge{v0}{v2}
\zeroedge{v0}{v3}
\zeroedge{v2}{v4}
\zeroedge{v1}{v4}
\zeroedge{v1}{v3}
} & 1
\\
\hline
\sagepentagon{\smoothedge{v0}{v1}{v2}{v3} \smoothedge{v0}{v4}{v3}{v3} \smoothedge{v1}{v2}{v0}{v4} \smoothedge{v2}{v3}{v4}{v4}
\zeroedge{v0}{v2}
\zeroedge{v0}{v3}
\zeroedge{v2}{v4}
}
\sagepentagon{ \smoothedge{v4}{v0}{v3}{v3} \smoothedge{v0}{v1}{v4}{v4} \smoothedge{v1}{v2}{v3}{v3} \smoothedge{v2}{v3}{v4}{v4}}
\sagepentagon{ \smoothedge{v4}{v0}{v1}{v2} \smoothedge{v0}{v1}{v4}{v4} \smoothedge{v1}{v2}{v3}{v3} \smoothedge{v2}{v3}{v1}{v0}} 
& 
1
\\
\hline
\sagepentagon{  \smoothedge{v2}{v3}{v1}{v1} \smoothedge{v3}{v4}{v1}{v1} \smoothedge{v4}{v0}{v1}{v1}}
\sagepentagon{  \smoothedge{v2}{v3}{v1}{v1} \smoothedge{v3}{v4}{v2}{v0} \smoothedge{v4}{v0}{v1}{v1}
\zeroedge{v1}{v3}
\zeroedge{v1}{v4}
}
\sagepentagon{  \smoothedge{v2}{v3}{v1}{v4} \smoothedge{v3}{v4}{v1}{v1} \smoothedge{v4}{v0}{v3}{v1}}
\sagepentagon{  \smoothedge{v2}{v3}{v4}{v4} \smoothedge{v3}{v4}{v1}{v1} \smoothedge{v4}{v0}{v3}{v3}}
\sagepentagon{  \smoothedge{v2}{v3}{v1}{v0} \smoothedge{v3}{v4}{v2}{v0} \smoothedge{v4}{v0}{v2}{v1}
\zeroedge{v1}{v3}
\zeroedge{v1}{v4}
} 
&
1
\\
\sagepentagon{  \smoothedge{v2}{v3}{v4}{v0} \smoothedge{v3}{v4}{v1}{v1} \smoothedge{v4}{v0}{v1}{v1}}
\sagepentagon{  \smoothedge{v2}{v3}{v1}{v1} \smoothedge{v3}{v4}{v1}{v0} \smoothedge{v4}{v0}{v1}{v1}
\zeroedge{v1}{v4}
}
\sagepentagon{  \smoothedge{v2}{v3}{v4}{v0} \smoothedge{v3}{v4}{v1}{v0} \smoothedge{v4}{v0}{v1}{v1}}
\sagepentagon{  \smoothedge{v2}{v3}{v4}{v0} \smoothedge{v3}{v4}{v0}{v0} \smoothedge{v4}{v0}{v1}{v1}}
\sagepentagon{  \smoothedge{v2}{v3}{v0}{v4} \smoothedge{v3}{v4}{v2}{v0} \smoothedge{v4}{v0}{v1}{v1}
\zeroedge{v0}{v2}
\zeroedge{v2}{v4}
}
&
1\\
\hline
\sagepentagon{  \smoothedge{v1}{v2}{v3}{v3}  \smoothedge{v1}{v0}{v4}{v4} \smoothedge{v3}{v4}{v2}{v0}  
\zeroedge{v2}{v4}
\zeroedge{v4}{v1}
\zeroedge{v1}{v3}
\zeroedge{v3}{v0}
}
\sagepentagon{  \smoothedge{v1}{v2}{v0}{v0}  \smoothedge{v1}{v0}{v4}{v4} \smoothedge{v3}{v4}{v1}{v2}  }
\sagepentagon{  \smoothedge{v1}{v2}{v4}{v3}  \smoothedge{v1}{v0}{v4}{v2} \smoothedge{v3}{v4}{v0}{v0}  }
\sagepentagon{  \smoothedge{v1}{v2}{v0}{v0}  \smoothedge{v1}{v0}{v4}{v2} \smoothedge{v3}{v4}{v0}{v0}  } 
&
1
\\
\hline 
\sagepentagon{\smoothedge{v0}{v1}{v2}{v2} \smoothedge{v0}{v2}{v1}{v1} \smoothedge{v1}{v2}{v0}{v0}} 
&
2
\\
\hline 
\sagepentagon{\smoothedge{v0}{v1}{v3}{v4} \smoothedge{v1}{v2}{v3}{v4}}
\sagepentagon{\smoothedge{v0}{v1}{v3}{v4} \smoothedge{v1}{v2}{v0}{v0}} 
&
\\
\sagepentagon{\smoothedge{v0}{v1}{v4}{v4} \smoothedge{v1}{v2}{v3}{v3}}
\sagepentagon{\smoothedge{v0}{v1}{v3}{v4} \smoothedge{v1}{v2}{v3}{v3}}
\sagepentagon{\smoothedge{v0}{v1}{v3}{v3} \smoothedge{v1}{v2}{v3}{v3}}
\sagepentagon{\smoothedge{v2}{v1}{v0}{v3} \smoothedge{v1}{v0}{v3}{v3}}
\sagepentagon{\smoothedge{v0}{v1}{v2}{v4} \smoothedge{v1}{v2}{v3}{v3}}
&
2
\\
\sagepentagon{\smoothedge{v2}{v1}{v0}{v3} \smoothedge{v1}{v0}{v3}{v4}}
\sagepentagon{\smoothedge{v0}{v1}{v2}{v3} \smoothedge{v1}{v2}{v0}{v3}
\zeroedge{v0}{v2}}
\sagepentagon{\smoothedge{v0}{v1}{v2}{v3} \smoothedge{v1}{v2}{v0}{v4}
\zeroedge{v0}{v2}
}
\sagepentagon{\smoothedge{v0}{v1}{v2}{v2} \smoothedge{v1}{v2}{v0}{v4}}
\sagepentagon{\smoothedge{v0}{v1}{v2}{v2} \smoothedge{v1}{v2}{v3}{v3}} 
&
\\
\hline
\sagepentagon{\smoothedge{v0}{v4}{v1}{v1} \smoothedge{v2}{v3}{v1}{v1}}
\sagepentagon{\smoothedge{v0}{v4}{v2}{v3} \smoothedge{v2}{v3}{v1}{v1}}
&
3\\
\sagepentagon{\smoothedge{v0}{v4}{v2}{v1} \smoothedge{v2}{v3}{v1}{v0}}
\sagepentagon{\smoothedge{v0}{v4}{v2}{v2} \smoothedge{v2}{v3}{v1}{v0}}
\sagepentagon{\smoothedge{v0}{v4}{v2}{v2} \smoothedge{v2}{v3}{v0}{v0}}
&
2
\\
\hline 
\sagepentagon{\smoothedge{v3}{v4}{v1}{v1}}
\sagepentagon{\smoothedge{v3}{v4}{v0}{v2}}
&
4
\\
\hline 
\sagepentagon{} &  6\\
\hline
\end{tabular}
\end{center}
\end{table}

\begin{proposition}\label{prop:smoothing-diagrams}
Suppose that $(\Gamma,m)$ is a smoothing diagram for $\PP^{2n}$ where $n \ge 2$.  Then the following statements hold:
\begin{enumerate}
\item Every vertex in the graph $\Gamma$ has valency at most two; hence $\Gamma$ is a disjoint union of connected chains and cycles.
\item\label{part:cycle-containment} If $C \subset \Gamma$ is a cycle, $\Delta \subset C$ is an edge, and $v$ is a vertex opposite $\Delta$ with nonzero order, then $v \in C$.
\item The graph $\Gamma$ contains no cycles of even length.
\end{enumerate}
\end{proposition}

\begin{proof}
Suppose that $\omega$ is a log symplectic form on $(\PP^{2n},\partial\PP^{2n})$.  We must show that the smoothing diagram associated to $\omega$ has the listed properties.  

For the first statement, suppose that  $v \in \dcP{2n}$ is a vertex.  We must show that it has at most two smoothable edges.  But if it had three, then they would span a three-horn in $\dcP{2n}$, which is impossible by \autoref{lem:horn}. Hence $\Gamma$ is at most bivalent, and the decomposition into chains and cycles follows immediately.  

For the second statement, observe that any adjacent pair of edges in $\dc{\PP^{2n}}$ spans a unique triangle.  If both edges are smoothable we may apply \autoref{prop:smooth-2chain} to deduce that the corresponding biresidues are positive rational multiplies of one another.  Hence if $C$ is a cycle, we may orient $C$ cyclically and apply this argument at each vertex to deduce that all biresidues along $C$ are positive rational multiplies of one another.  Then by rescaling $\omega$, we may assume without loss of generality that the biresidues along $C$ are positive numbers.  Now suppose that $v$ is a vertex not lying in $C$, and identify $C$ with $\ZZ/k\ZZ$ in cyclic order.  For $i \in \ZZ/k\ZZ$, let $a_i$ be the biresidue along the unique edge from the vertex $i \in C$ to $v$, oriented towards $v$.  Considering the triangle $\Delta'$ with vertices $i,i+1,v$, we obtain a relation
\[
a_{i+1}-a_i = \orddel B_\Delta \ge 0
\]
for every edge $\Delta = (i,i+1)$ in $C$.  Going around the cycle we conclude that $a_0 \le a_1 \le \cdots \le a_k \le a_0$, so that $a_{i}=a_{i+1}$ for all $i$, and hence $\orddel=0$ whenever $v \notin C$, as claimed.

Finally, for the third statement, suppose that $C$ is a cycle in $\Gamma$ whose length $l$ is even.  By part 2 and \autoref{ex:P2n-single-edge}, we have for any $i \in \ZZ/l\ZZ$ that
\[
\sum_{j \in \ZZ/l\ZZ}B_{i,j} = 2 =\sum_{j \in \ZZ/l\ZZ} B_{i+1,j}.
\]
In other words, if $B_C := (B_{i,j})_{i,j\in\ZZ/l\ZZ}$ is the matrix of biresidues along the simplex spanned by $C$, then the column sums of $B_C$ are all equal.  But since $B_C$ is skew-symmetric, its entries add up to zero.  Hence the individual columns must sum to zero, so that $B_C$ is degenerate, and since it is skew-symmetric of even size, its rank is at most $l-2$.  Now consider an arbitrary $(2n-1)$-simplex $\Delta \in \dcP[2n]{2n}$ that contains $C$.  (This is always possible since $\dcP{2n}$ has an odd number of vertices.)  The biresidue along $\Delta$ then has the form of a $2n\times 2n$ block matrix
\[
B_\Delta = \begin{pmatrix}
 B_C & A \\
 -A^T & D
\end{pmatrix}
\]
where the block $A$ encodes the biresidues from vertices in $C$ to vertices outside $C$.  Therefore statement 2 implies that all rows of $A$ are constant, and hence $A$ has rank at most one.  Therefore
\[
\rank(B_\Delta) \le \rank(B_C) + \rank(A) + (2n-2l) \le (2l-2) + 1 + (2n-2l) = 2n-1,
\]
so that $B_\Delta$ is degenerate.  But the stratum in $\PP^{2n}$ corresponding to $\Delta$ is a single point; hence it is necessarily a symplectic leaf.  This contradicts~\autoref{lem:bires-nondegen}.
\end{proof}

\section{Higher order deformations}

\label{sec:L-infinity}

\subsection{$L_\infty$ algebras and the deformation functor}

We now turn to the study of the obstructions to deformations. To this end, we make use of the general formalism of derived deformation theory via $L_\infty$ algebras. 

Since $\der{\X}$ is a sheaf of differential Gerstenhaber algebras, its hypercohomology carries a homotopy Gerstenhaber structure ($G_\infty$ structure), canonically defined up to gauge equivalence.  It is constructed in two steps.   In the first step, one chooses a suitable resolution of $\der{\X}$ by sheaves of dg Gerstenhaber algebras in order to compute the algebra $R\Gamma(\der{\X})$ of derived global sections, giving the latter the structure of dg Gerstenhaber algebra.  In the second step, one chooses a homotopy equivalence of complexes $\Hpi{\X} \cong R\Gamma(\der{\X})$ and uses it to transfer the $G_\infty$ structure to $\Hpi{\X}$, e.g.~using the general formulae in~\cite{Berglund2014} or \cite[Section 10.3]{Loday2012}. 

\begin{remark}
A suitable resolution can be obtained, for instance, using the Dolbeault resolution.  Alternatively, one can use Sullivan's Thom--Whitney normalization of the \v{C}ech resolution as reviewed, for instance, in \cite[Section 4]{Calaque2014} or \cite[Appendix B]{Dolgushev2015}.  The latter has the advantage of working equally well for sheaves that are not coherent, such as the constant sheaf $\CC_\X$.  More generally, it can be used to show that the derived direct image functor $Rf_*$ for any morphism $f : \X \to \Y$ is compatible with algebra structures on sheaves.  
\end{remark}

In particular, corresponding to the Schouten--Nijenhuis bracket on $\der{\X}$ is a sequence of brackets $\Hpi{\X}^{\otimes k} \to \Hpi{\X}[3-2k]$ of arity $k \ge 2$, which induce an $L_\infty$ structure on the graded vector space $\Hpi{\X}[1]$.

\begin{definition}
We denote by $\Def{\X}$, the simplicial set-valued deformation functor associated to the $L_\infty$ algebra underlying $\Hpi{\X}$.
\end{definition}
More precisely, $\Def{\X}$ is the functor sending an Artin ring with maximal ideal $\m$ to the simplicial set of Maurer--Cartan elements in $\Hpi[2]{\X}\otimes_\CC \m$ as in \cite{Hinich1997,Getzler2009}.  Since the underlying $L_\infty$ algebra $\Hpi{\X}[1]$ is concentrated in degrees $\ge -1$, this simplicial set is the nerve of the corresponding Deligne (weak-)2-groupoid~\cite{Bressler2015,Getzler2002}: objects are solutions of the Maurer--Cartan equation in $\Hpi[2]{\X} \otimes_\CC \m$, acted on by formal gauge transformations parameterized by $\Hpi[1]{\X} \otimes_\CC \m$ and with homotopies between gauge transformations parameterized by $\Hpi[0]{\X} \otimes_\CC \m$.  We remark that if $\X$ is compact or generically symplectic, then $\Hpi[0]{\X} = \CC$ is a strict central subalgebra, and hence it acts trivially on the space of gauge transformations.  In this sense it is not crucial for understanding the deformation functor per se, but it becomes more important when we consider natural operations, such as fibre products, which relate collections of deformation functors. 

This deformation functor has a natural geometric interpretation: it parametrizes formal deformations of $(\X,\pi)$ as a generalized complex manifold in the sense of Hitchin~\cite{Hitchin2003} and Gualtieri~\cite{Gualtieri2011}; see, in particular \cite[Section 5]{Gualtieri2011}.  Alternatively, it parametrizes Poisson deformations of $(\X,\pi)$ twisted by a gerbe as in \cite{VandenBergh2007,Yekutieli2015}.  Either way, if $h^1(\cO{\X})=h^2(\cO{\X}) = 0$, it reduces deformations of the pair $(\X,\pi)$ as a holomorphic Poisson manifold as in \cite{Ginzburg2004,Kim2014,Namikawa2008}.  Note that if, in addition, $(\X,\pi)$ is compact and log symplectic, then since nondegeneracy is an open condition, small deformations will remain log symplectic, although as we shall see the deformed anticanonical divisor need not remain normal crossings.  Similarly, one can show that holonomicity is an open condition in proper families, but this is more subtle, and since we will not need this statement in this paper, we shall omit the proof.

If $\U\subset \X$ is an open set, then the restriction on Poisson cohomology gives a $G_\infty$ morphism $\Hpi{\X} \to \Hpi{\U}$, canonically defined up to homotopy, and a corresponding morphism $\Def{\X} \to \Def{\U}$ of deformation functors.  We will require two basic lemmas about these restriction maps.  The first allows us to simplify the deformation functor by throwing out irrelevant subsets.

\begin{lemma}\label{lem:retrict-def}
Suppose that the linear restriction map $\Hpi[k]{\X} \to \Hpi[k]{\U}$ is an isomorphism for $k\le 2$ and injective for $k=3$.  Then the $G_\infty$ restriction map induces an isomorphism
\[
\Def{\X} \cong \Def{\U}
\]
of deformation functors.
\end{lemma} 
\begin{proof}
Passing to the corresponding $L_\infty$ algebras (which involves a shift in degree by one), this follows from a standard statement in deformation theory: an $L_\infty$ morphism whose linear part is an isomorphism in degrees $\le 1$ and injective in degree two induces an isomorphism of deformation functors.  For set-valued deformation functors, see~e.g.~\cite[Theorem V.51]{Manetti2004}.  For simplicial set-valued functors, this follows from \cite[Proposition 4.9]{Getzler2009}; see also \cite[Proposition 3.9]{Bressler2015} for more details on the inductive argument.   Note that the cited propositions assume that the linear part is an isomorphism in all degrees, but an examination of the proof reveals that  the weaker conditions listed above suffice, since the deformation functor only depends on the truncations of the $L_\infty$ algebras in degree $\le 2$.
\end{proof}

The second lemma allows for gluing over an open cover.  More precisely, observe that if $\V$ is another open set, we obtain a homotopy  commutative diagram
\begin{align}
\vcenter{\xymatrix{
& \Hpi{\U\cup\V} \ar[ld] \ar[rd] \\
\Hpi{\U}\ar[rd] && \Hpi{\V}\ar[ld] \\
& \Hpi{\U \cap \V}} }\label{eq:MV-square}
\end{align}
The following lemma then follows from Mayer--Vietoris sequence.  
\begin{lemma}\label{lem:MV}
The square \eqref{eq:MV-square} is  homotopy Cartesian, giving an equivalence
\[
\Def{\U\cup\V} \cong \Def{\U}\fibprod_{\Def{\U\cap\V}}\Def{\V}
\]
where the right hind side is the homotopy fibre product.  It reduces to the strict fibre product when the linear restriction map $\Hpi{\V} \to \Hpi{\U\cap \V}$ is surjective.
\end{lemma}
See, e.g., \cite[Section 7]{Bandiera2018} for generalities on fibre products of $L_\infty$ algebras, and \cite{Bandiera2017,Bressler2015,Hinich1997,Rogers2020} for their compatibility with deformation functors and descent over open covers.

Now observe that $\X$ itself may be built up by starting from the open set $\U = \Xo$, and inductively attaching tubular neighbourhoods $\V$ of all higher-codimension strata; this mirrors the cell attachments defining the dual  complex $\dcX$.  In this way, we may construct the entire $G_\infty$ algebra $\Hpi{\X}$ as a homotopy limit over the simplices in the dual complex.  At least in principle, this reduces the problem to describing $\Hpi{\Xo}$, and explaining what happens when we attach tubular neighbourhood using the fibre square above.  We focus here on the pieces relevant to our local Torelli theorem, starting with the open set $\Xo$. 

\subsection{Codimension zero: unobstructed log deformations}

\begin{proposition}\label{prop:weight-zero-abelian}
If $(\X,\pi)$ is any normal crossings log symplectic manifold, then the $L_\infty$ structure on $\Hpi{\X}[1]$ is equivalent to one for which the summand $\W_0\Hpi{\X} \cong \coH{\Xo} \subset \Hpi{\X}$ is a strict abelian $L_\infty$ subalgebra, and the linear restriction map $\Hpi{\X} \to \coH{\Xo}$ is a strict $L_\infty$ morphism. 
\end{proposition}

\begin{proof}
  Let $j : \Xo = \X\setminus \bdX \to \X$ be the inclusion of the open symplectic leaf.  Note that since $\Xo$ is a hypersurface complement, the inclusion $j$ is a Stein morphism, so that the higher direct images of any coherent sheaf vanish.  In particular, the natural map $ j_*\forms{\Xo} \to Rj_*\forms{\Xo}$ in $\dcat{\X}$ is an isomorphism.
  
Using the canonical isomorphism $j^*\der{\X} \cong \forms{\Xo}$ induced by the symplectic form, we obtain the following commutative diagram of sheaves of differential Gerstenhaber algebras on $\X$:
\[
\xymatrix{
\W_0{\der{\X}} \ar[r] & \der{\X} \ar[r]& j_*\forms{\Xo} \ar[r]^{\sim} & Rj_*\forms{\Xo} & Rj_*\CC_{\Xo}\ar[l]_{\sim} \\
& \logforms{\X} \ar[ul]^-{\sim} \ar[ur]_-{\sim}
}
\]
where the differential and bracket on $\CC_{\Xo}$ are identically zero, and the symbol $\sim$ indicates quasi-isomorphisms.  Taking hypercohomology and inverting weak equivalences, we see immediately that $\W_0\Hpi{\X}$ is equivalent to the strictly abelian algebra $\coH{\Xo;\CC}$, and we furthermore obtain a pair of $L_\infty$ morphisms $i = (i_1,i_2,...) : \coH{\Xo;\CC} \to \Hpi{\X}$ and $p=(p_1,p_2,\ldots) : \Hpi{\X} \to \coH{\Xo;\CC}$ such that $p_1i_1=1$.

By \cite[Lemma 1.5.4]{Schuhmacher2004} we may apply a gauge transformation to $\Hpi{\X}$ to make $i$ a strict $L_\infty$ morphism, so that $\W_0\Hpi{\X}\cong \coH{\Xo;\CC}$ is a strict abelian subalgebra.  Finally, by \cite[Lemma 1.5.4]{Schuhmacher2004} again (see also \cite[Lemma 7.2]{Bandiera2018}) we may apply a further gauge transformation to make the projection $p$ a strict morphism; from the explicit formula for this gauge transformation, we see that the restriction of the transformed $L_\infty$ structure to the subspace $\W_0\Hpi{\X}$ remains abelian, giving the result. 
\end{proof}

\begin{remark}
The proof shows that $\Hpi{\Xo}$ is $L_\infty$-equivalent to the minimal abelian algebra $\coH{\Xo;\CC}$, but does not give the explicit $L_\infty$ quasi-isomorphism.  However, as explained in \cite{Fiorenza2012,Gualtieri2020}, there is a canonical $L_\infty$ quasi-isomorphism from the sheaf $(\derlog{\X}[1],\dpi,[-,-])$ of dg Lie algebras to the corresponding abelian sheaf $(\derlog{\X}[1],\dpi,0)$, given by exponentiating the binary operation $R(\xi,\omega) = \hook{\omega}(\xi\wedge \eta) - \hook{\omega}\xi \wedge \eta - \xi \wedge \hook{\omega}\eta$ that measures the failure of contraction with $\omega$ to be a derivation of the wedge product. 
\end{remark}

Recalling that a homotopy abelian $L_\infty$ algebra presents a smooth deformation functor, we have the following immediate corollaries:
\begin{corollary}\label{cor:wt0-unobstr}
First-order deformations lying in the weight-zero  subspace $\W_0\Hpi[2]{\X} = \coH[2]{\Xo;\CC}$ are unobstructed to all orders; they correspond to  deformations for which $\bdX$ deforms locally trivially (i.e.~remains normal crossings with locally unchanged number of components).
\end{corollary}

\begin{corollary}\label{cor:no-smoothings}
If $(\X,\pi)$ has no smoothable strata, then $\Def{\X}$ is smooth, and the divisor $\bdX \subset \X$ deforms locally trivially under all deformations of $(\X,\pi)$.  In fact, $\Def{\Xo}$ is isomorphic to the formal completion of the 2-stack
\[
\coH[2]{\Xo;\CC} \times \B\coH[1]{\Xo;\CC} \times \B^2 \CC
\]
at the origin, where $\coH[1]{\Xo;\CC}$ is abelian. 
\end{corollary}

\begin{proof}
In this case, $\gr^\W_2 \Hpi[2]{\X} = 0$ so that $\W_0\Hpi[\le 2]{\X} = \Hpi[\le 2]{\X}$ and $\W_0\Hpi[3]{\X} \subset \Hpi[3]{\X}$.  Hence as in \autoref{lem:retrict-def}, the deformation functors of $\Hpi{\X}$ and its weight-zero part are equivalent.  Since $\W_0\Hpi{\X}$ is strictly abelian, the Deligne 2-groupoid $\Def{\Xo}(A)$ for an Artin ring $A$ with maximal ideal $\mathfrak{m}$ is the simplicial vector space associated to the graded vector space $\coH[\le 2]{\Xo;\CC} \otimes_\CC \mathfrak{m}$ by the Dold--Kan correspondence~\cite{Getzler2009}, which gives the description of the full 2-stack.
\end{proof}

\begin{remark}\label{rmk:period-map}
The corollary has the following geometric interpretation as a period mapping: given a topologically trivial normal crossings deformation $(\X_t,\omega_t)$, where $t$ is a formal parameter (or a parameter taking values in a contractible space), the Gauss--Manin connection on logarithmic cohomology gives a canonical identification $\phi_t : \coH[2]{\Xo_t;\CC} \to \coH[2]{\Xo;\CC}$ and the map $\Def{\X} \to \coH[2]{\Xo;\CC}$ sends $(\X_t,\omega_t)$ to $\phi_t([\omega_t]) \in \coH[2]{\Xo;\CC}$. 
\end{remark}

\begin{remark}\label{rmk:ran}
The corollaries can be viewed as extensions of Bogomolov's unobstructedness theorem for deformations of compact hyperk\"ahler manifolds~\cite{Bogomolov1978} to the log symplectic setting.  Similar criteria for unobstructedness were obtained by Ran~\cite{Ran2017,Ran2020,Ran2020a} in the slightly different context of Poisson (rather than generalized complex) deformations; they are related to \autoref{cor:no-smoothings} as follows.  The P-normality condition of \cite{Ran2017} is equivalent to the statement that all biresidues are identically zero.  In this case, there are no smoothable strata by part 1 of \autoref{prop:smoothable-bires}, so \autoref{cor:no-smoothings} applies.    Meanwhile, the ``2-very general position'' condition of \cite{Ran2020,Ran2020a} implies that the integer-linear condition \eqref{eq:bires-smoothable} on biresidues from \autoref{prop:smoothable-bires} will fail for all codimension-two strata, so that again none of them are smoothable and  \autoref{cor:no-smoothings} applies. 
\end{remark}

\subsection{Codimension two: unobstructed smoothings}
\label{sec:codim2}

We now consider what happens to the deformation functor when we add strata of codimension two.  Near such a stratum, the Poisson structure is locally stably equivalent to a normal crossings structure on $\CC^2$, so we treat that simple case first and then explain how to globalize the calculation.

\subsubsection{Case $\X = \CC^2$:}

Equip $\CC^2$ with the following log symplectic form $\omega$ and its inverse bivector $\pi$:
\[
\omega = B \dlog{y_0}\wedge \dlog{y_1} \qquad \qquad \pi = -\frac{1}{B} \logcvf{y_0}\wedge\logcvf{y_1}
\]
where $B \in \CC$ is nonzero.   Its dual complex consists of a single smoothable edge, with orientation and biresidue as follows:
  \[
\begin{tikzpicture}[scale=0.7,decoration={
    markings,
    mark=at position 0.55 with {\arrow{>}}}
    ] 
\draw[thick,blue,postaction={decorate}] (0,0) -- (1,0);
\draw[blue] (0.5,-0.4) node {$B$};
\draw (-0.3,0) node {$y_0$};
\draw (1.3,0) node {$y_1$};
\end{tikzpicture}
\]
The universal deformation of this structure is obtained by varying the biresidue $B$ and smoothing the nodal at the origin, as explained in the introduction.  This can be obtained from the following complete description of $\Hpi{\CC^2}$ as a $G_\infty$ algebra.

Let us denote by
\begin{align}
v_0 = \pis\dlog{y_0}  = - \frac{1}{B}\logcvf{y_1}  \qquad \qquad v_1 = \pis \dlog{y_1} = \frac{1}{B}\logcvf{y_0} \label{eq:v-classes}
\end{align}
weight zero vector fields dual to loops around the coordinate axes, and let
\begin{align}
s = \cvf{y_0}\wedge\cvf{y_1} \label{eq:s-classes}
\end{align}

Then $s$ has weight two and projects to a basis of $\gr_2\der{\CC^2} \cong i_*\sLe$, where $\sL_e = \det \tb[0]{\CC^2}$ is a one-dimensional vector space, viewed as a skyscraper sheaf supported at the origin.

It is straightforward to check that the elements $v_0,v_1,s$ generate a Gerstenhaber subalgebra of $\der{\CC^2}$ of the following form:

\begin{definition}\label{def:2dGerst}
For a one-dimensional vector space $\sL$ and a constant $B \in \CC$, we denote by
\[
\CC[v_0,v_1] \ltimes_B \sL
\]
the Gerstenhaber algebra generated by degree one elements $v_0,v_1$ and a copy of $\sL$ placed in degree two, with multiplication determined by the relations
\[
v_0\sL = v_1 \sL =\sL \cdot \sL = 0
\]
and with the only nonzero bracket given by
\[
[a_0v_0+a_1 v_1,s] = \frac{a_0-a_1}{B} s
\]
for $a_0,a_1\in\CC$ and $s \in \sL$.
\end{definition}

\begin{proposition}\label{prop:C2}
We have a $G_\infty$ equivalence
\[
\Hpi{\CC^2} \cong \CC[v_0,v_1] \ltimes_B \sLe
\]
and the $G_\infty$ restriction map $\Hpi{\CC^2} \to \coH{(\CC^2)^\circ}$ is homotopic to the strict quotient map $\CC[v_0,v_1] \ltimes_B \sLe \to \CC[v_0,v_1]$ that kills the ideal $\sLe$.
\end{proposition}

\begin{proof}
That these elements form a Gerstenhaber subalgebra with trivial differential and the given multiplicative and bracket relations is immediate.  That this subalgebra projects isomorphically to the cohomology follows immediately from \autoref{thm:HPhol}, since $\W_0\Hpi{\CC^2} = \coH{(\CC^*)^2}$ is the cohomology of a two-torus and is therefore freely generated by the classes of the logarithmic forms $\dlog{y_i}$.  Meanwhile the unique codimension two stratum $\Xe$ is a point $\coH{\Xe;\sLe}$ is simply the one-dimensional space of global sections, for which $s$ is a basis.

For the statement about the restriction map, consider the square
\[
\xymatrix{
\CC[v_0,v_1] \ltimes_B \sLe \ar[r]^-{\phi} \ar[d]^-{p} & \der{\CC^2} \ar[d]^-{r} \\
\CC[v_0,v_1] \ar[r]^{\phi_0} & \der{(\CC^*)^2}
}
\]
induced by the embeddings $\phi$, $\phi_0$, the quotient map $p$ and the restriction map $r$.  This diagram is not commutative, because $\phi(s) = \cvf{y_0}\wedge\cvf{y_1}$ does not vanish identically on the open leaf $(\CC^*)^2$, but we  claim that it commutes up to homotopy, which suffices.  Indeed,  the operator $h$ given by contraction with the one-form $\dlog{y_0}$ on $\der{(\CC^*)^2}$ is a dg Gerstenhaber derivation of degree $-1$, satisfying the homotopy identities
\[
[\dpi,h]v_i = 0 \qquad [\dpi,h]s=s
\]
and therefore $h$ gives a homotopy of dg Gerstenhaber morphisms from $r \circ \phi$ to $\phi_0 \circ p$, as desired.
\end{proof}

\subsubsection{Neighbourhood of a single stratum}

Now let $(\X,\omega)$ be an arbitrary normal crossings log symplectic manifold.  We do not assume that the normal crossings divisor is simple.  Let $e$ be an edge of the dual complex.  It corresponds to a codimension-two stratum $\Xoe \subset \X$.  Let $\V \subset \X$ be a $C^\infty$ tubular neighbourhood of $\Xoe$ that is compatible with the normal crossings divisor.  We have $\Def{\V} \neq \Def{\Vo}$ if and only if $e$ is a smoothable edge for $(\V,\omega|_\V)$, or equivalently conditions \ref{cnd:nondegen} and \ref{cnd:nomon} of \autoref{prop:smoothable-bires} are satisfied, the remaining condition being vacuous since $\Xoe$ has no boundary strata in $\V$.  Note that in this case, $\partial \V$ is simple by  \autoref{rmk:non-simple-symplectic}.  We will describe the deformation functor of $\V$ in this case as follows. 

Choose a basepoint $p \in \Xoe$, let $\{\ell^j\}_{j \in \I}$ be a collection of loops based at $p$ that generate $\Hlgy[1]{\Xoe;\ZZ}$, giving an embedded graph
\[
\bouqe \cong  \bigvee_{j \in \I} \ell^j \hookrightarrow \Xoe
\]
\begin{lemma}\label{lem:skeltube}
The subspace $\bouqe \subset \Xoe$ admits a Stein tubular neighbourhood $\W \subset \Xoe$.  For any such $\W$, its normal bundle in $\X$ is holomorphically trivial.
\end{lemma}
\begin{proof}
We may thicken $\bouqe$ to a totally real two-dimensional submanifold of $\Xoe$, which then admits a Stein tubular neighbourhood in $\Xoe$ by the results of \cite[Section 3]{Grauert1958} (see also \cite[Section 3.5]{Forstneric2017}).   Since $\coH[2]{\V;\ZZ}\cong \coH[2]{\bouqe;\ZZ} = 0$, every line bundle on $\V$ is topologically trivial, and since $\V$ is Stein, every topologically trivial bundle is holomorphically trivial.  But the normal bundle $\cN{e}$ is a direct sum of line bundles, and therefore it is holomorphically trivial over $\V$ as claimed.
\end{proof}

Let $\W\subset \Xoe$ be a tubular neighbourhood of $\bouqe$ as in \autoref{lem:skeltube}.  Then $\W$ further admits a tubular neighbourhood  $\tV$ inside $\V$ by a theorem of Siu~\cite{Siu1976}. Homogenizing the Poisson structure on $\tV$ using \autoref{prop:homogenization}, we deduce that $\tV$ is holomorphic Poisson isomorphic to a tubular neighbourhood of the zero section in the trivial bundle $\W \times \CC^2$, equipped with a fibrewise homogeneous log symplectic structure and a possibly nontrivial symplectic connection.  Moreover we can choose the isomorphism so that the chosen orientation of $e$ matches the standard orientation for the edge of $\dc{\CC^2}$.  Since the local system $\sLe$ is trivial, the polyvectors $v_0,v_1,s$ on $\CC^2$ given by the formulae \eqref{eq:v-classes} and \eqref{eq:s-classes} define global holomorphic Poisson cocycles on $\W \times \CC^2$. Meanwhile, $\W$ is a symplectic manifold, so that $\Hpi{\W} \cong \coH{\W;\CC}$ is equipped with the abelian bracket, and since $\W$ is homotopy equivalent to a bouquet of circles, its commutative dg algebra of cochains is formal.   We therefore obtain the following strong form of the K\"unneth decomposition for the Poisson cohomology:
\begin{proposition}
The Poisson cohomology $\Hpi{\tV}$ is formal as a $G_\infty$ algebra; in fact
\[
\Hpi{\tV} \cong \coH{\W;\CC} \otimes \rbrac{\CC[v_0,v_1]\ltimes_{\Be} \coH[0]{\sLe}} 
\]
and
\[
\Hpi{\tVo} \cong \coH{\W;\CC} \otimes \CC[v_0,v_1],
\]
where $\coH{\W;\CC}$ is viewed as a Gerstenhaber algebra with the cup product and zero bracket. The restriction $\Hpi{\tV} \to \Hpi{\tVo}$ is homotopic to the quotient by the ideal generated by $\coH[0]{\sLe}$.
\end{proposition}

From the form of the brackets on $\CC[v_0,v_1]\ltimes_{\Be} \coH[0]{\sLe}$ from \autoref{def:2dGerst}, it follows that the bracket on $\Hpi{\tV}$ is completely determined by following components:
\[
\begin{array}{ccc}
\coH[1]{\tVo;\CC} \times \coH[0]{\sLe} &\to& \coH[0]{\sLe} \\
( \mu , s ) &\mapsto& \frac{1}{\Be} (\res_0\mu  -\res_1 \mu)  s
\end{array}
\]
where $\res_i : \coH[1]{\tVo;\CC} \to \CC$ are the residue maps for the two components of $\bdX$ intersecting along $\Xoe$, and
\[
\begin{array}{ccc}
\coH[2]{\tVo;\CC} \times \coH[0]{\sLe} &\to& \coH[1]{\W} \otimes \coH[0]{\sLe}\\
( \eta , s ) &\mapsto& \frac{1}{2 \pi \sqrt{-1}\Be}\sum_{j \in \I} \rbrac{\int_{\torcyc^j_0-\torcyc^j_1} \eta}\cdot \alpha^j \otimes s
\end{array}
\]
where $\{\alpha^j\}_{j \in \I} \subset \coH[1]{\W;\CC}$ are dual to the classes $\ell^j \in \Hlgy[1]{\W;\ZZ}$ and $\torcyc^j_{i}$ are the extensions of $\ell^j$ to tori in the open leaf $\tVo$ as in \autoref{sec:monodromy}. 

The elements $(\eta,s) \in \coH[2]{\tVo;\CC} \times \coH[0]{\sLe} = \Hpi[2]{\tV}$ that solve the Maurer--Cartan equation $[\eta+s,\eta+s] =0$ therefore  have the following form:
\begin{itemize}
\item if $s=0$, then $\eta \in \coH[2]{\tVo;\CC}$ is arbitrary
\item $s\ne 0$, then $\int_{\torcyc^j_1-\torcyc^j_0} \eta = 0$ for all $j \in \I$.
\end{itemize}

The deformation functor $\Def{\tV}$ is then the 2-stack quotient of this Maurer--Cartan locus by the linear action of $\coH[\le 1]{\tVo;\CC}$.  This explicit description has the following gauge-invariant geometric counterpart.
\begin{corollary}\label{cor:edge-constructible}
Let $\Def[\edge]{\tVo} \subset \Def{\tVo}$ be the closed subfunctor representing normal crossings deformations of $\tV$ for which the unique edge $\edge$ remains smoothable.  Then the projection $\Def{\tV} \to \Def{\tVo}$ is a constructible vector bundle, given by the union of the zero section and a line bundle over $\Def[\edge]{\tVo}$ with fibre $\coH[0]{\sLe}$.
\end{corollary}

\subsubsection{Case without strata of codimension three:}

Now let $(\X,\omega)$ be a normal crossings log symplectic manifold, and consider the open subset
\[
\Xtwo := \X \setminus \bdX[3]
\]
obtained by removing all strata of codimension greater than two.  We do not assume that $\bdX$ is simple normal crossings. We describe $\Def{\Xtwo}$ as follows.

Let $\Gamma_0$ be the smoothing diagram of $(\Xtwo,\omega|_{\Xtwo})$ and let $\Gamma < \Gamma_0$ be any smoothing subdiagram.  Let
\[
\Def[\Gamma]{\Xo} \subset \Def{\Xo}
\]
by the smooth closed subfunctor corresponding to the linear subspace consisting of classes $\eta \in \coH[2]{\Xo;\CC}$ such that for every edge $\edge$ of $\Gamma$ and every loop $\ell$ in $\Xoe$, we have $\int_{\torcyc_1-\torcyc_0}\eta = 0$.  Thus $\Def[\Gamma]{\Xo}$ corresponds to normal crossings deformations of $(\X,\omega)$ for which all edges of $\Gamma$ remain smoothable in $\Xtwo$.

Then the subfunctors $\Def[\Gamma]{\Xo}$ give a stratification of $\Def{\Xo}$ by linearly embedded substacks.  Recall also that we have an embedding $\Def{\Xo}\hookrightarrow \Def{\Xtwo}$ corresponding to the normal crossings deformations of $\Xtwo$, via the inclusion $\coH{\Xo;\CC} \hookrightarrow \Hpi{\Xtwo}$.

\begin{theorem}\label{thm:Xtwo}
$\Def{\Xtwo}$ is a constructible vector bundle over the stratified deformation functor $\Def{\Xo}$.  More precisely, for every smoothing subdiagram $\Gamma$ as above, there exists a unique irreducible component
\[
\Def[\Gamma]{\Xtwo} \subset \Def{\Xtwo}
\]
whose intersection with $\Def{\Xo}$ is $\Def[\Gamma]{\Xo}$.  This component is a vector bundle over $\Def[\Gamma]{\Xo}$ with fibre isomorphic to $\bigoplus_{\edge \in \Gamma} \coH[0]{\sLe}$.  
\end{theorem}

\begin{proof}
For each edge $\edge$ of $\dcX$, choose a bouquet of circles $\bouqe \subset \Xoe$ generating the first homology, a tubular neighbourhood $\We$ of $\bouqe$ in $\Xodel$ and a tubular neighbourhood $\tVe$ of $\We$ in $\X$ as above.  We may assume without loss of generality that the neighbourhoods $\tVe$ are pairwise disjoint.

Let $\tX = \Xo \cup \tV$ where $\tV = \bigsqcup_{\edge \in \dcX[1]} \tVe$.  We claim that the restriction map $\Def{\Xtwo} \to \Def{\tX}$ is an isomorphism.  Indeed, note that the open symplectic leaf in $\tX$ is $(\tX)^\circ = \Xo$.  Therefore the restriction
\[
\W_0\Hpi{\Xtwo} \to \W_0\Hpi{\tX}
\]
is an isomorphism.  Meanwhile the codimension two strata are given by the submanifolds $\We \subset \Xoe$.  Since each submanifold $\We$ is connected and generates the first homology of $\Xoe$, the  linear restriction map $\coH[\le 1]{\Xoe;\sLe} \to \coH[\le 1]{\We;\sLe}$ is an isomorphism in degree zero and injective in degree one.  We conclude that the linear restriction $\gr_2 \Hpi{\Xtwo} \to \gr_2\Hpi{\tX}$ is an isomorphism in degree two and injective in degree three.  Since $\Xtwo$ and $\tX$ have no strata of codimension greater than two, it follows that $\Hpi{\Xtwo} \to \Hpi{\tX}$ is bijective in degree two and injective in degree three, giving an isomorphism of deformation functors by \autoref{lem:retrict-def}.

Now apply \autoref{lem:MV} to the open cover $\tX = \Xo \cup \tV$; we conclude that
\[
\Def{\Xtwo} \cong \Def{\tX} \cong \Def{\Xo}\fibprod_{\Def{\tVo}} \Def{\tV}.
\]
But the restriction morphism $\Def{\tV} \to \Def{\tVo}$ is the direct product of the individual restriction maps $r_e : \Def{\tVe} \to \Def{\tVe^\circ}$ where $\edge$ is an edge of $\dcX$.  If the edge $\edge$ is not smoothable, then $r_e$ is an isomorphism.  Otherwise, $r_e$ is a constructible vector bundle $\Def{\tVe} \to \Def{\tVe^\circ}$ as in \autoref{cor:edge-constructible}.  Therefore $\Def{\tV}$ is a constructible bundle whose strata are indexed by collections of smoothable edges of $\dcX$, and the desired description of $\Def{\Xtwo}$ follows by pulling this constructible bundle back along the map $\Def{\Xo} \to \Def{\tVo}$.  Indeed, that the pullback is a constructible vector  bundle over some stratification of $\Def{\Xo}$ is immediate.   That  this stratification agrees with the stratification defined above follows from the fact that the morphism $\Def{\Xo}\to \Def{\tVo}$ is induced by the linear restriction $\coH{\Xo;\CC} \to \coH{\tVo;\CC}$ of abelian dg Lie algebras, and the smoothability of an edge is determined by linear conditions on periods of $\omega$ in $\tVo$. 
\end{proof}

\begin{remark}\label{rmk:param}
In the situation of the theorem, the deformations of $\X$ all still have normal crossings singularities; indeed, for each codimension-two stratum there are only two possibilities: either it is smoothed or it is not. The generic deformation in $\Def[\Gamma]{\X}$ will furthermore have no smoothable strata, so that it satisfies the hypotheses of \autoref{cor:no-smoothings}.  This indicates that nearby irreducible components are also described by period maps as in \autoref{rmk:period-map}.  
\end{remark}

\subsection{Higher codimension: the full deformation functor}
We now describe the deformation functor $\Def{\X}$ for an arbitrary normal crossings log symplectic manifold, by relating it to the deformation functor of the open set
\[
\Xs \subset \Xtwo
\]
obtained from $\Xtwo$ by removing all codimension-two strata that are not smoothable in $\X$. Note that a stratum of $\Xtwo$ is smoothable in $\X$ if and only if it is smoothable in $\Xtwo$ and nonresonant; hence by \autoref{thm:Xtwo},  $\Def{\Xs} \to \Def{\Xo}$ is isomorphic to the quotient of the constructible vector bundle $\Def{\Xtwo} \to \Def{\Xo}$ that kills the subspaces $\coH[0]{\sLe}$, for every resonant edge $\edge$. The irreducible components of $\Def{\Xs}$ are then indexed by smoothing subdiagrams of the smoothing diagram for $(\X,\omega)$.  
 
Moreover, in light of \autoref{thm:hol-complex}, the restriction $\Hpi{\X} \to \Hpi{\Xs}$ is an isomorphism in degrees at most two, and surjective in degree three.  Hence the restriction $\Def{\X} \to \Def{\Xs}$ in an embedding of a full closed substack with the same Zariski tangent space.  We will describe the image of this embedding, and in particular give  necessary and sufficient conditions for it to be an isomorphism, thus establishing \autoref{thm:torelli-bundle} from the introduction.

\subsubsection{Nonresonant case}  The simplest case is when all codimension-two symplectic strata are nonresonant.  Note that this is an open condition on the class $[\omega] \in \coH[2]{\Xo;\CC}$.  In this case we have the following: 

\begin{lemma}\label{lem:nonres-def}
Suppose that $(\X,\pi)$ is a normal crossings log symplectic manifold and that all codimension-two characteristic symplectic leaves are nonresonant.  Then the restriction maps $\Def{\X} \to \Def{\Xtwo} \to \Def{\Xs}$ are isomorphisms.
\end{lemma}

\begin{proof}
By \autoref{ex:tri-holonomic} and \autoref{ex:res23}, every codimension-two stratum with nonholonomic boundary is resonant.  Therefore the assumptions imply that $\X$ is holonomic away from a codimension four subset.  In this case, it follows as in the proof of \autoref{thm:hol-complex}, part 3, that the restriction to $\Hpi{\X} \to \Hpi{\Xtwo}$ splits the weight filtration in weights $\le 2$, and in particular the restriction $\Hpi[\le 3]{\X} \to \Hpi[\le 3]{\Xtwo}$ is an isomorphism, so that $\Def{\X}\cong\Def{\Xtwo}$.  Now let $\V \subset \Xtwo$ be any open set containing $\Xo$ such that $\Xtwo = \Xs \cup \V$, and $\V \cap \Xs$ has no codimension-two strata.  Then $\V$  and $\V \cap \Xs$ have no smoothable strata, so that $\Def{\V\cap\Xs} \cong \Def{\Xo} \cong \Def{\V}$ by \autoref{cor:no-smoothings}.   Hence using the Mayer--Vietoris square of \autoref{lem:MV}, we obtain isomorphisms
\begin{align*}
\Def{\Xtwo} &\cong \Def{\Xs} \fibprod_{\Def{\V\cap \Xs}} \Def{\V} \\
&\cong \Def{\Xs}\fibprod_{\Def{\Xo}}\Def{\Xo} \\
&\cong \Def{\Xs}. \qedhere
\end{align*}
\end{proof}

\subsubsection{Running example}\label{sec:runex-def} We now consider the special case of the manifold  $\U := \CC^4$, with the log symplectic structure of our running example:
\begin{align*}
\omega = b_1 \dlog{y_2}\wedge \dlog{y_3} + b_2 \dlog{y_3}\wedge \dlog{y_1} + b_3 \dlog{y_1}\wedge \dlog{y_2}+ \dd z \wedge \sum_i \dlog{y_i} 
\end{align*}
As an immediate consequence of our explicit calculation of the Poisson cohomology in \autoref{ex:runex-cohomology}, we see that the following conclusion of \autoref{lem:nonres-def} holds, even though some strata may be resonant:

\begin{lemma}\label{lem:runex-Us}
The restriction $\Hpi{\U} \to \Hpi{\Us}$ is an isomorphism; hence so is the map $\Def{\U} \to \Def{\Us}$.
\end{lemma}

The deformation space can be described concretely  with the aid of \autoref{fig:3cpt-line-arr} from the introduction.  Namely, the smoothable edges of $\omega$ are in bijection with the collection of $k \le 3$ grey lines passing through the point $[b_1:b_2:b_3] \in \PP(\coH[2]{\Uo;\CC})$; the corresponding two-dimensional subspaces of $\coH[2]{\Uo;\CC}$ generate the strata in the base of the constructible vector bundle  $\Def{\Us} \to \Def{\Uo}$.  If $k=0$, then $\Def{\Us}\cong\Def{\Uo}$ is smooth and parameterized by the biresidues $b_i$.  Otherwise, the deformation space has either $2$ components, if $k=1$ (just the zero section together with the rank-one bundle over the grey line), or $k+2$ components (the zero section, the rank-one bundles over the $k$ grey lines, and a rank $k > 1$ bundle over the point), 
some of which may be identified in the full moduli space by the action of the symmetric group $S_3$.  

Using the explicit formulae for the corresponding cocycles from \autoref{ex:runex-strata-contrib}, we see that that the weight two classes in $\Hpi[2]{\U}$ are given by monomial bivectors of the form $y_i^{m_i}\cvf{y_{i+1}}\wedge\cvf{y_{i+2}}$ where $m_i \in \ZZ_{\ge0}$.  One then readily checks that the following bivector is Poisson, and represents the universal deformation of $\pi$ over $\Def{\U$}:
\begin{align}
\begin{split}
\tilde \pi &= \pi + \cvf{z} \wedge \sum_i \tilde b_i \logcvf{y_i} + \sum_i \tau_i y_i^{m_i} \cvf{y_{i+1}}\wedge\cvf{y_{i+2}}  \label{eq:runex-MC}
\end{split}
\end{align}
where the indices are taken modulo three, and $\tilde b_i,\tau_i,m_i \in \CC$ are parameters such that $\tau_i = 0$ unless $m_i \in \ZZ_{\ge 0}$ and $\tilde b_{i+1}=\tilde b_{i+2}$; these parameters give coordinates on the components on the miniversal deformation space.   Computing the Pfaffian of $\tilde \pi$, we see that up to rescaling the variables, the degeneracy divisor has the form
\[
 y_1y_2y_3 + \sum_i \tau_i y_i^{m_i+1} = 0
\]
Note that action of $\Hpi[1]{\U} \cong \coH[1]{\Uo;\CC}\cong \CC^3$ integrates to the action of the torus $(\CC^*)^3$ by dilation of the variables $y_i$.  Using this action, we may rescale the parameters $\tau_i$ so that they are all equal to some fixed constant $\lambda$; if the number $k$ of smoothable strata is less then three, then we may arrange that $\lambda = 1$.   In this way, we obtain the singularities listed in \autoref{tab:3cpt-smoothings} in the introduction. 

In order to globalize this calculation below, we will need to characterize the image of the map $\Def{\U} \to \Def{\Utwo}$.  In the presence of resonances, this map fails to be an isomorphism, and can instead by interpreted as defining a compatibility relation amongst the smoothing of the smoothable and resonant strata, as follows. 

 Let us denote by $\Ur \subset \Utwo$ the open set obtained by removing the smoothable strata of $\U$. Then the smoothable strata of $\Ur$ are exactly the resonant strata of $\U$, and $\Def{\Ur}\to \Def{\Uo}$ is a constructible vector bundle over $\Def{\Uo}$.  We define a map
\[
\phi : \Def{\Us} \to \Def{\Ur}
\]
of fibrations over $\Def{\Uo}$ by composing the isomorphism $\Def{\Us}\cong\Def{\U}$ from \autoref{lem:runex-Us} with the restriction $\Def{\U} \to \Def{\Ur}$.  The deformation-theoretic interpretation of this map is as follows: a deformation $\tau \in \Def{\Us}$ extends uniquely to a deformation of $\U$ as above, and this could in turn result in a nontrivial smoothing of the resonant strata, which is captured by $\phi(\tau)$.  Thus  $\phi(\tau)$ lies in the zero section $\Def{\Uo}\subset \Def{\Ur}$ if and only if the resonant strata are \emph{not} smoothed by this process. Note that we have the isomorphism of constructible vector bundles
\[
\Def{\Utwo} \cong \Def{\Us} \times_{\Def{\Uo}} \Def{\Ur}
\]
with fibres given by direct sums of the spaces $\coH[0]{\sLe}$ where $\edge$ is an edge, and by construction $\Def{\U}$ is identified with the substack of this constructible vector bundle given by the graph of $\phi$.  Hence to characterize the image of $\Def{\U} \to \Def{\Utwo}$, it suffices to compute $\phi$, which we do in the following:
\begin{proposition}\label{prop:runex-smoothings}
The map $\phi$ is the zero map of constructible vector bundles unless two edges of $\Delta$ are smoothable and the third is resonant, i.e.~the biresidue has the form $\ObsTri{n}$ from \eqref{eq:2chain-special} for some $n \ge 0$.  In this case, up to permutation of the coordinates, the biresidue has the form
\[
\Bdel = (b_1,b_2,b_3) = b_\Delta (n+1,1,-1)
\]
for a unique constant $b_\Delta \in \CC$, and $\phi$ is a fibrewise monomial map, which acts on weight-two Maurer--Cartan elements by the formula
\[
\phi( u_1 \eta_1 + u_2 \eta_2 ) =  (-(n+1)b_\Delta)^{n+1} u_1^{n+1}u_2 \cdot \eta_3
\]
where $u_1,u_2 \in \CC$ and
\[
\eta_1 = [\cvf{y_2}\wedge\cvf{y_3}] \quad \eta_2 = [y_2^n\cvf{y_3}\wedge\cvf{y_1}] \quad \eta_3 = [y_3^{-(n+2)}\cvf{y_1}\wedge\cvf{y_2}]
\]
are a basis of cohomology classes corresponding to the codimension-two strata.
\end{proposition}

\begin{proof}
Consider the action of the torus $\G := (\CC^\times)^3$ by dilations of the coordinates $y_i$.  It preserves the Poisson structure, and therefore acts on $\der{\CC^4}$ by dg Gerstenhaber automorphisms.   By homotopy transfer, we obtain  $\G$-invariant $L_\infty$ structures on $\Hpi{\U}$, $\Hpi{\Utwo}$, $\Hpi{\Us}$ and $\Hpi{\Ur}$, and $\G$-equivariant $L_\infty$ morphisms $\Hpi{\Us} \cong \Hpi{\U}\to\Hpi{\Utwo} \to \Hpi{\Ur}$ homotopic to the restriction maps, so that $\Def{\Us}\to\Def{\Ur}$ is the induced map on Maurer--Cartan elements.  

Note that the weight-two parts of these graded vector spaces are generated over $\coH{\Xo}$ by bivectors of the form $y_i^{k_i}\cvf{y_{i+1}}\wedge\cvf{y_{i+2}}$ where $k_i$ is the order of the corresponding smoothable/resonant edge.  It follows that the decomposition of the Poisson cohomologies into summands corresponding to strata matches the decomposition of $\Hpi{\Utwo}$ into weight spaces for the action of $\G$, and that the possible nonzero weights are of the form $(k_1,-1,-1),(-1,k_2,-1)$ and $(-1,-1,k_3)$ where $k_i \in \ZZ$. A component of the $L_\infty$ morphism corresponds to a multilinear map between these weight spaces, of total weight zero.   Hence the only way that the action of the $L_\infty$ map $\Hpi{\Us} \to \Hpi{\Ur}$ on elements of nonzero weight can be nontrivial is if a resonant weight vector can be written as a $\ZZ_{>0}$-linear combination of smoothable weight vectors.  

By inspection, such a linear combination can only occur if two edges are smoothable and the third is resonant, i.e.~the triangle is of type $\ObsTri{n}$, in which case up to permutation, the biresidue is a multiple of $(n+1,1,-1)$ and the linear combination of weight vectors is
\[
(n+1)\cdot (0,-1,-1) + 1 \cdot (-1,n,-1) = (-1,-1,-n-2).
\]
The induced map on Maurer--Cartan elements must then have the form $\phi(u_1\eta_1+u_2\eta_2) = u_3 \eta_3$, where $u_3 = c u_1^{n+1}u_2$ for some $c \in \CC$ depending only on $b_\Delta$ and $n$. One then verifies by direct calculation that the Poisson bivectors $\pi + u_1\eta_1+u_2\eta_2$ and $\pi + u_3\eta_3$ can only be isomorphic on the complement of the hyperplane $y_3=0$  if $c = (-(n+1)b_\Delta)^{n+1}$.
\end{proof}

\subsubsection{Gluing together the full deformation functor} 
Now consider an arbitrary normal crossing log symplectic manifold $(\X,\omega)$.  We do not assume that the normal crossings divisor $\bdX$ is simple.  We will show that the only obstructions to deformations of $\X$ that are not captured already by $\Def{\Xtwo}$ are localized near points on the codimension-three strata, where they are described by the local models above.  To this end, choose a collection of pairwise disjoint open set $\Udel$, one for each triangle $\Delta$ of $\dcX$, such that $\Udel$ forms a tubular neighbourhood of an open ball in the corresponding codimension three stratum $\Xodel$.  

There are two possibilities for such a ball:
\begin{enumerate}
\item $\Udel$ is holonomic, in which case it is stably equivalent to the running example, and the image of $\Def{\Udel} \to \Def{\Udeltwo}$ is described as the graph of a map as in \autoref{prop:runex-smoothings};
\item $\Udel$ is nonholonomic, in which case it case its codimension two strata with nonzero biresidue are all resonant by \autoref{lem:nonhol-res}, so that $\Def{\U} \cong \Def{\Uo}$ is the zero section in $\Def{\Udeltwo}$ by \autoref{cor:no-smoothings}.
\end{enumerate}

Let $\U = \bigsqcup_\Delta \Udel$ be the union of these balls.  We establish the following result at the end of this subsection:

\begin{proposition}\label{lem:codim3-balls-suffice}
The restriction map
\[
\xymatrix{
\Def{\X} \ar[r]\ar[d] &  \Def{\U} \cong \prod_{\Delta} \Def{\Udel} \ar[d] \\
\Def{\Xtwo} \ar[r] & \Def{\Utwo}\cong \prod_\Delta \Def{\Udeltwo}
}
\]
induces an isomorphism
\[
\Def{\X} \cong \Def{\Xtwo} \fibprod_{\Def{\Utwo}} \Def{\U}
\]  
\end{proposition}

 \autoref{lem:codim3-balls-suffice} implies that $\Def{\X}$ is the closed embedded substack of $\Def{\Xtwo}$ consisting of elements whose restriction to every ball $\Udel$ lies in the image of $\Def{\Udel} \subset \Def{\Udeltwo}$.  Note that this amounts to a system of fibrewise monomial equations in the constructible vector bundle $\Def{\Xtwo}$, indexed by pairs $(\Delta,\edge)$ of a triangle $\Delta$ and a resonant edge $\edge$ of $\Delta$.  If $\Delta$ is not of type $\ObsTri{n}$ the equation is simply that every element of $\Def{\X}$ projects to zero in the summand $\coH[0]{\sLe}$ in the fibres.  Otherwise, the equation states that the projection to $\coH[0]{\sLe}$ is a monomial function of the projections to the summands $\coH[0]{\sL_{\edge'}}$ and $\coH[0]{\sL_{\edge''}}$ corresponding to the remaining edges $\edge',\edge''$ of $\Delta$ as in \autoref{prop:runex-smoothings}; this is the only possible source of new higher-order obstructions. 
 
\begin{definition}
An \defn{obstruction triangle} for $\omega$ is a triangle of $\dcX$ whose biresidue has type $\ObsTri{n}$ for some $n \ge 0$.
\end{definition} 
 
  Note that the obstructions involving a fixed edge $\edge$ are local in a neighbourhood of the corresponding stratum, and thus they take values in the corresponding summand $\coH[1]{\Xe;\sLenr}$ in $\Hpi[3]{\X}$.  We therefore have the following generalization of the nonresonant case (\autoref{lem:nonres-def}). 
\begin{corollary}\label{cor:X=Xs}
Suppose that either one of the following conditions hold:
\begin{enumerate}
\item\label{cond:no-special-triangles} $(\X,\omega)$ has no obstruction triangles.
\item\label{cond:weak-nonres} $\coH[1]{\Xe;\sLenr} = 0$ for every resonant edge $\edge$.
\end{enumerate}
Then $\Def{\X} \cong \Def{\Xs}$, and in particular all fibres of $\Def{\X} \to \Def{\Xo}$ are smooth.
\end{corollary}

\begin{example}\label{ex:P2n-moduli}
Suppose that $\X = \PP^{2n}$ is a projective space with boundary divisor given by the union of the coordinate hyperplanes. If $\edge$ is an edge then $(\Xe,\bdXe)\cong (\PP^{2n-2},\partial\PP^{2n-2})$ is the standard toric pair.  If $\Y \subset \bdXe$ is any nonempty subdivisor and $\sL$ is a rank-one local system on $\Xoe\cup\Y$ it is straightforward to verify that $\coH{\Xoe \cup \Y,\Y;\sL} = 0$. Hence for any log symplectic structure on $\PP^{2n}$, condition \ref{cond:weak-nonres} of \autoref{cor:X=Xs} is automatically satisfied.  If follows that, in a formal neighbourhood of the toric log symplectic structures, the irreducible components of the moduli space of Poisson structures on $\PP^{2n}$ are in bijection with smoothing diagrams.  

A priori it could happen that some global irreducible components become reducible in this formal neighbourhood.  However we believe that this cannot occur; it should be possible to rule it out based on torus weight considerations, or better, by determining the full singularity structure of the deformed Poisson structure from the smoothing diagram.  We hope to explore this in more detail in future work.  For $n=2$, we confirmed by direct calculation that the 40 families of Poisson structures corresponding to the classification of smoothing diagrams in \autoref{tab:confs_P4} are truly distinct, which yields the classification stated in \autoref{thm:P4} in the introduction.  
\end{example}

The conditions of \autoref{cor:X=Xs} are not necessary for the equality $\Def{\X}\cong \Def{\Xs}$.  For instance, if there is an obstruction triangle which has two smoothable edges and a resonant edge that becomes smoothable after removing $\Xdel$, then the resonant smoothing presents no new obstructions.  However, when there are more obstruction triangles, or there is an obstruction triangle  whose resonant edge has nontrivial monodromy, there could be new nontrivial obstructions which produce singularities of the fibres.  Here  is a simple example.

\begin{example}\label{ex:singular-fibre}
Let $\Z = \CC^3$ equipped with the divisor given by the union of the hyperplanes $y_1 = \pm 1$, $y_2 = 0$ and $y_3 = 0$ and let $\X = \logctb{\Z}$.   The dual complex is a pair of triangles corresponding to the codimension three strata $y_1\mp 1 = y_2 = y_3=0$ and by the K\"unneth theorem $\coH[2]{\Xo;\CC}$ is isomorphic to the space of all possible biresidues of a magnetic term as in \autoref{ex:magnetic}.  For $m,n \ge 0$, consider a log symplectic structure with the following biresidues:
\[
\begin{tikzpicture}[scale=0.7,decoration={
    markings,
    mark=at position 0.55 with {\arrow{>}}}
    ] 
\coordinate (A) at (0,1);
\coordinate (B) at (-1.5,0);
\coordinate (C) at (0,-1);
\coordinate (D) at (1.5,0);
\draw[fill=gray] (A) -- (B) -- (C) -- (D) -- (A);
\draw[thick,blue,postaction={decorate}] (A) -- (B) ;
\draw[thick,blue,postaction={decorate}] (B) -- (C) ;
\draw[thick,blue,postaction={decorate}] (A) -- (D) ;
\draw[thick,blue,postaction={decorate}] (D) -- (C) ;
\draw[thick,red,postaction={decorate}] (C) -- (A) ;
\draw[blue] (-1.4,0.7) node {$n+1$};
\draw[blue] (1.5,0.7) node {$m+1$};
\draw[blue] (1.1,-0.75) node {$1$};
\draw[blue] (-1.1,-0.75) node {$1$};
\draw[red] (0.5,0) node {$-1$};
\end{tikzpicture}
\]
Let $u , v \in \gr_2 \Hpi[2]{\X}$  be nonzero classes giving bases for the spaces of smoothings of the left edges with biresidues $1, n+1$, respectively, and similarly consider classes $\tilde u,\tilde v$ corresponding to the right edges.  The triangles are of type $\ObsTri{n}$ and $\ObsTri{m}$, so smoothings of the left or right pair of edges induce corresponding resonant smoothings of the red edge, which requires that the other pair of edges also smooth correspondingly---so the divisor smooths completely. The fibre of $\Def{\X}$ over the class of the log symplectic structure is defined by the equality of these resonant smoothings; thus, up to rescaling the variables, it is given by the formal completion of the space of linear combinations $a u+ bv+\tilde a \tilde u + \tilde b\tilde v$ such that
\[
ab^{n+1} = \tilde  a\tilde b^{m+1}.
\]
This gives a singular three-fold as the fibre of $\Def{\X}$.  For instance, if $n=m=0$ we obtain the equation $ab=\tilde a \tilde b$ for a quadric cone (conifold).

If we modify this example by changing the biresidue $m+1$ to $m+1+\varepsilon$, for a small non-zero $\varepsilon$, then the local system on the stratum with biresidue $-1$ will acquire a monodromy around the codimension three stratum given by the right triangle. Additionally, such a change will render both the right edges unsmoothable, so the cohomology classes $\tilde u$, $\tilde v$ will no longer exist.
In that case the fibre of $\Def{\X}$ will be cut out by the equation
$
a b^{n+1} = 0
$, and it will be impossible to smooth out both the left edges simultaneously. The smoothing diagram here consists of just the two left smoothable edges, forming a triangle of type $\ObsTri{n}$, where now the resonant stratum has nontrivial monodromy.
\end{example}

\begin{proof}[Proof of \autoref{lem:codim3-balls-suffice}]
Applying Mayer--Vietoris  (\autoref{lem:MV}) to the open cover of the open set $\V := \Xtwo \cup \U$ by the subsets $\Xtwo$ and $\U$ with intersection $\Xtwo \cap \U = \Utwo$, we see by \autoref{lem:retrict-def} that it suffices to prove that the restriction $\Hpi{\X} \to \Hpi{\V}$ is an isomorphism in degrees at most two, and injective in degree three.  Since the open leaves in $\X$ and $\V$ are both equal to $\Xo$, the restriction on weight zero classes $\W_0\Hpi{\X} \to \W_0\Hpi{\V}$ is an isomorphism.  Meanwhile,  since $\dcX$ and $\dc{\V}$ have the same 2-skeleton, it follows from \autoref{prop:smoothable-bires} that their smoothable strata are the same.  Hence the restriction $\Hpi[2]{\X} \to \Hpi[2]{\V}$ is also an isomorphism.

It remains to show that the map $\Hpi[3]{\X}\to \Hpi[3]{\V}$ is injective. For this, consider the exact sequence
\begin{align}
\xymatrix{
0 \ar[r] & \W_2 \der{\X} \ar[r]& \der{\X} \ar[r] & \sQ \ar[r] & 0.
}\label{eq:wt3-seq}
\end{align}
where the weight two part is given by
\[
\W_2 \der{\X} \cong Rj_*\CC_{\Xo} \oplus \bigoplus_{\textrm{edges }\edge \subset \dcX} \sLdelnr[-2],
\]
and $\sQ := \der{\X}/\W_2\der{\X}$ is the quotient, which is concentrated in degrees at least three.  The long exact sequence in hypercohomology induced by \eqref{eq:wt3-seq} then gives the following commutative diagram with exact rows:
\[
\xymatrixcolsep{1.5pc}
\xymatrix{
0 \ar[r] & \coH[3]{\Xo;\CC} \oplus \bigoplus_{\edge \in \dcX[2]} \coH[1]{\sLenr} \ar[r]\ar[d] & \Hpi[3]{\X} \ar[r] \ar[d] & \coH[3]{\sQ} \ar[d] \\
0 \ar[r] & \coH[3]{\Xo;\CC} \oplus \bigoplus_{\edge \in \dc[2]{\V}} \coH[1]{\sLenr|_\V} \ar[r] & \Hpi[3]{\V} \ar[r] & \coH[3]{\sQ|_\V}
}
\]
The maps $\coH[1]{\sLenr} \to \coH[1]{\sLenr|_\V}$ are the restrictions on cohomology of local systems relative to certain boundary components and are easily seen to be injective.  Meanwhile elements in $\coH[3]{\sQ}$ all have weight three, and are given by global sections of line bundles over nonholonomic codimension-three strata. Since $\V$ contains an open set in each such stratum, the restriction $\coH[3]{\sQ}\to\coH[3]{\sQ|_\V}$ is injective by analytic continuation, and hence the restriction $\Hpi[3]{\X} \to \Hpi[3]{\V}$ is injective by the four lemma.
\end{proof}

\subsubsection{Description of the fibres}
\label{sec:irred-decomp}

We now explain how to describe the fibre of $\Def{\X}$ over the base point of $\Def{\Xo}$ corresponding to the trivial deformation in terms of the cohomology class $[\omega] \in \coH[2]{\Xo;\CC}$.  More precisely, the fibre is the formal completion of the Maurer--Cartan variety
\[
\F \subset \prod_{\edge \in \Smth{\pi}} \coH[0]{\Xoe;\sLe}
\]
which as explained above is defined by monomial equations.  We will show that $\F$ has a stratification by tori indexed by collections of edges of $\dcX$ that are simultaneously smoothed, and the structure of the stratification can be determined from the linear algebra of the periods of $\omega$.   We will assume for simplicity that $\bdX$ is simple and has finitely many irreducible components.  We may thus choose a cyclic order of the vertices of $\dcX$ and thereby obtain consistent orientations of all edges and triangles.  We fix such an orientation from now on, so that we may speak unambiguously of the biresidue $\Be \in \CC$ for each edge $\edge \in \tilde\Gamma$.

Given a collection $\Gamma$ of smoothable edges, we introduce the following objects:
\begin{itemize}
\item $\F_\Gamma \subset \F$ denotes the subvariety corresponding to deformations in which all smoothable edges in $\Gamma$ are smoothed, and all other smoothable edges are left unsmoothed:
\[
\F_\Gamma := \F \cap \set{ (s_e)_e \in \prod_{\edge \in \Smth{\pi}} \coH[0]{\sLe} }{ s_e \neq 0 \textrm{ if and only if }e \in \Gamma}
\]
\item $\tilde \Gamma \supset \Gamma$ denotes the set of all edges that must be smoothed if the edges of $\Gamma$ are smoothed, due to the presence of obstruction triangles, i.e. $\tilde \Gamma$ is the smallest collection of edges containing $\Gamma$ such that for every obstruction triangle $\Delta$, the resonant edge of $\Delta$ lies in $\tilde \Gamma$ if and only if both nonresonant edges of $\Delta$ lie in $\tilde \Gamma$.
\item $\STri{\Gamma}$ denotes the set of obstruction triangles whose edges all lie in $\tilde \Gamma$
\end{itemize}

Clearly $\F_{\Gamma}$ is a Zariski-locally closed subvariety of $\F$, and we have
\[
\F = \coprod_{\Gamma} \F_{\Gamma}.
\]
Note that for a fixed $\Gamma$, all smoothings of $\tilde \Gamma$ are determined by the smoothings of $\Gamma$, via resonance along obstruction triangles.  Thus we have an embedding
\[
\F_\Gamma \hookrightarrow \prod_{\edge \in \tilde \Gamma} (\coH[0]{\sLe}\setminus \{0\})
\]
In particular, $\F_\Gamma$ is empty unless all local systems $\sLe$, $\edge \in \tilde \Gamma$ are trivial.  In the latter case, we can describe $\F_\Gamma$ as the set of solutions of a system of monomial equations, as follows.

Choose a trivialization of each local system $\sLe$, $\edge \in \tilde \Gamma$, and let $u_\edge$ be the corresponding coordinate on $\coH[0]{\sLe}$.   If $\Delta \in \STri{\Gamma}$ is an obstruction triangle, let
\[
b_\Delta \in \CC
\]
be the unique number such that the biresidues on $\Delta$ are given by $\pm b_\Delta$ and $(n+1)b_\Delta$, $n \in \ZZ_{\ge 0}$.  Then considering the degrees of the monomials from \autoref{prop:runex-smoothings}, we see that the monomial equations defining $\F$ must  have the form
\begin{align}
\prod_{\edge \subset \Delta} u_e^{k_{\edge,\Delta}} = C_\Delta \label{eqn:monomial}
\end{align}
for some constants $C_\Delta \in \CC^*$ (to be determined below),  where the exponents are given by the ratios of biresidues:
\begin{align}
k_{\edge,\Delta} := \frac{B_e}{b_\Delta} \in \ZZ. \label{eqn:k-def}
\end{align}
Note that the collections of integers $(k_{\edge,\Delta})_{k,\Delta}$ depends only on the underlying smoothing diagram of $\omega$, and is independent of the choice of orientations.  It therefore determines a canonical $\ZZ$-linear map
\[
\kappa : \ZZ^{\STri{\Gamma}} \to \ZZ^{\tilde \Gamma}
\]
Similarly the constants $C_\Delta$ determine a homomorphism
\[
C : \ZZ^{\STri{\Gamma}}\to\CC^*
\]
Note that if we change the coordinates by a transformation $u_e \mapsto \tau_e u_e$ for some constants $\tau_e \in \CC^*$, then $C$ changes by $C \mapsto (\tau \circ \kappa)C$, and hence the restriction
\[
\tilde C := C|_{\ker \kappa} : \ker \kappa \to \CC^*
\]
is independent of the coordinates, i.e.~it depends only on $\omega$ and the discrete choice of orientation of the edges.  We will give an explicit expression for $\tilde C$ in terms of periods of $\omega$ in \autoref{prop:fibre-nonempty} below.

Finally observe that a solution of the system \eqref{eqn:monomial} with $u_e = \lambda_e \in \CC^*$ for $e \in \tilde \Gamma$ determines a homomorphism
\[
\lambda:\ZZ^{\tilde \Gamma} \to \CC^*
\]
which immediately gives the following:

\begin{lemma}\label{lem:sol-homs}
The set of solutions to the system \eqref{eqn:monomial} is canonically isomorphic to the set of homomorphisms $\lambda : \ZZ^{\tilde \Gamma} \to \CC^*$ making the following diagram commute:
\[
\xymatrix{
\ZZ^{\STri{\Gamma}} \ar[rr]^-{\kappa} \ar[rd]_-{C}& & \ZZ^{\tilde \Gamma}\ar@{-->}[dl]^-{\lambda} \\
& \CC^*
}
\]
\end{lemma}

This in turn, gives the following structure of the fibre:
\begin{proposition}\label{prop:fibre-structure}
The variety $\F_\Gamma$ is nonempty if and only if the following conditions hold:
\begin{itemize}
\item For every $\edge \in \tilde \Gamma$, the local system $\sLe$ has trivial monodromy, and its only resonances occur along obstruction triangles.
\item $\tilde C = 1$
\end{itemize} 
In this case, $\F_\Gamma$ is a torsor for the abelian Lie group $\HomZ{\coker \kappa,\CC^*}$.  In particular, the connected components of $\F_\Gamma$ are in bijection with elements of the torsion subgroup of $\coker \kappa$, and each component is a torus of dimension equal to the rank of $\coker \kappa$.
\end{proposition}

\begin{proof} 
Note that if for some $\edge \in \tilde \Gamma$, the local system $\sLe$ is trivial or has a resonance along a non-obstruction triangle, then $\Def{\X}$ projects to zero  in $\coH[0]{\Xo;\sLe}$.  Assuming neither of those occur, the remaining equations for $\Def{\X}$ come from the monomial equations \eqref{eqn:monomial} so that $\F_{\Gamma}$ is identified with the set of homomorphisms $\lambda$ as in  \eqref{lem:sol-homs}.  Note that by exactness, $\tilde C = 1$  if and only if $C$ descends to a homomorphism from $\img(\kappa) < \ZZ^{\tilde \Gamma}$ to $\CC^*$.  But $\CC^*$ is an injective group, so any such homomorphism automatically extends to an arrow $\lambda : \ZZ^{\tilde \Gamma} \to \CC^*$ making the diagram commute, as desired.  

To see that $\F_{\Gamma}$ is a torsor, note that the group $\HomZ{\coker \kappa ,\CC^*}$ is identified with the set of homomorphism $\phi : \ZZ^{\tilde \Gamma}\to \CC^*$ such that $\phi\circ \kappa = 1$.  This group acts freely and transitively on the set of arrows $\lambda$ by scalar multiplication $\lambda \mapsto\phi\lambda$.  

Finally, since $\coker \kappa $ is a finitely generated abelian group, we may write $\coker \kappa = A \oplus B$ where $A$ is a free abelian group of finite rank and $B$ is a product of finite subgroups of $\CC^*$.  Then $\HomZ{\coker \kappa,\CC^*}$ is the product of the torus $\HomZ{A,\CC^*}$ and the finite group $\HomZ{B,\CC^*}$, which is abstractly isomorphic to $B$, giving the description of the connected components of $\F_\Gamma$. 
\end{proof}

To derive a formula for the homomorphism $\tilde C$, choose balls $\Udel$ centred at points on the strata $\Xdel$, $\Delta \in \STri{\Gamma}$.  For each $\Udel$ we may write the log symplectic form in the standard coordinates in \autoref{sec:runex-def}, and since the Poisson structure is invariant under independent rescalings of the coordinates $y_i$ we may assume that the polydisk of radius one is contained in $\Udel$.  Let $x_{e,\Delta} \in \Xoe \cap \Udel$ be the base points corresponding to the points $(y_1,y_2,y_3,z) = (1,0,0,0),(0,1,0,0)$ and $(0,0,1,0)$, and let $u_{e,\Delta}$ be the coordinate on each $\coH[0]{\sLe|_{\Xoe\cal\Udel}}$ dual to the bivectors $y_{i}^{m_{i}}\cvf{y_{i+1}}\wedge\cvf{y_{i+2}}$, $i=1,2,3$; the latter give flat sections of $\sLe$ that agree with the coordinate trivializations at the base points $x_{e,\Delta}$.

By \autoref{prop:runex-smoothings}, the equations defining the fibre of $\Def{\Udel}\subset \Def{\Utwo}$ have the form
\[
\prod_{\edge \subset \Delta} u_{e,\Delta}^{k_{e,\Delta}} =  \theta_\Delta
\]
where for a triangle $\Delta$ of type $\Theta_n$, we have
\[
\theta_\Delta := \frac{1}{(-(n+1)b_\Delta)^{n+1}}. 
\]
We denote by 
\[
\theta : \ZZ^{\ObsTri{\Gamma}} \to \QQ^* \hookrightarrow \CC^*
\]
the induced homomorphism. 

Note that since the local system $\sLe$ has rank one, there exists unique constants $a_{e,\Delta} \in \CC^*$ such that $u_{e,\Delta} = a_{e,\Delta}u_{e}|_{\Udel}$.  Then
\begin{align}
C_{\Delta} = \frac{\theta_\Delta}{  \prod_{e \subset \Delta} a_{e,\Delta}^{k_{e,\Delta}}} \label{eqn:Cdelta}
\end{align}
Since the constants $a_{e,\Delta}$ measure the difference in the chosen trivialization of $\sLe$ over $\Xoe$ and $\Xoe \cap \Udel$, and they can be computed by the parallel transport formula \eqref{cor:no-monodromy}, applied to paths instead of loops, as follows.

Choose a base point $x_e$ in each $\Xoe$ and connect it to each $x_{e,\Delta}$ via a path $\gamma_{e,\Delta}$.  Since the resulting graph has a Stein neighbourhood, we may choose a holomorphic trivialization of the normal bundle $\cN{\Xoe}$ in a neighbourhood of this graph that is equal to the $u_e$ at $x_e$ and is equal to $u_{e,\Delta}$ at $x_{e,\Delta}$.  Then we have
\begin{align}
a_{e,\Delta} = \exp\rbrac{\tfrac{-1}{2\pi\sqrt{-1}\Be} \int_{\tilde \gamma_{e,\Delta}^0-\tilde \gamma_{e,\Delta}^1}\omega } \label{eqn:adelta}
\end{align}
where  $\tilde \gamma_{e,\Delta}^i \cong \gamma_{e,\Delta}\times S^1$ are lifts of the path to cylinders in $\Xo$ wrapping the irreducible components intersecting along $\Xoe$, constructed via the trivialization of $\cN{\Xoe}$.  By construction, the boundary circle of such a cylinder lying over a base point  $x_{e,\Delta}$ is given in local coordinates by one of the three circles $\{|y_i|=1,y_{i+1}=y_{i+2}=1, z=0\}$, and similarly all boundary circles that wrap the same irreducible component of $\bdX$ at the points $x_e$ are equal.   

Define a two-chain in $\Xo$ by
\begin{align}
\sigma_\Delta :=  \frac{1}{2 \pi \sqrt{-1}b_{\Delta}} \sum_{\edge \subset \Delta}(\tilde \gamma_{e,\Delta}^0 - \tilde \gamma_{e,\Delta}^1), \label{eq:2-chain-b}
\end{align}
and note that by \eqref{eqn:k-def} we may alternatively write
\begin{align}
\sigma_\Delta = \frac{1}{2 \pi \sqrt{-1}} \sum_{\edge \subset \Delta} \frac{k_{e,\Delta}}{\Be} (\tilde \gamma_{e,\Delta}^0 - \tilde \gamma_{e,\Delta}^1). \label{eq:2-chain-k}
\end{align}
We  extend this construction $\ZZ$-linearly to obtain a homomorphism $\sigma$ from $\ZZ^{\STri{\Gamma}}$ to the group of 2-chains on $\Xo$ with complex coefficients.  
\begin{proposition}\label{prop:fibre-nonempty}
If $w \in \ker \kappa \subset \ZZ^{\STri{\Gamma}}$, then $\sigma(w)$ is a cycle in $\Xo$, and the induced diagram
\[
\xymatrix{
\ker \kappa \ar[r]^-{\sigma} \ar[rd]_-{\tilde C/\theta} & \Hlgy[2]{\Xo;\CC} \ar[r]^-{\int\omega} &  \CC \ar[ld]^-{\exp} \\
& \CC^*
}
\]
is commutative.  In particular, $\tilde C = 1$ if and only if for every $w \in \ker \kappa$, we have
\begin{align}
\exp\rbrac{\int_{\sigma(w)}\omega} = \theta(w). \label{eq:fibre-nonempty-period}
\end{align}
\end{proposition}

\begin{proof}
Each $\sigma_\Delta$ is a sum of six cylinders so its boundary consists of twelve circles (counted with multiplicity).  In light of \eqref{eq:2-chain-b}, the part of $\partial \sigma_\Delta$ lying in $\Udel$ forms a 1-chain consisting of three circles, each taken once with coefficient $+1$ and once with coefficient $-1$, so that $\partial \sigma_\Delta$ has no boundary in $\Udel$.  Thus  $\partial \sigma_\Delta$ reduces to a sum of six circles, broken into pairs that form 1-chains $\rho_e$ lying  in fibres over each $x_e$, which are independent of $\Delta$ by construction.  Now if $v = \sum v_\Delta \Delta \in \ker \kappa$, we have $\sum_\Delta v_\Delta k_{e,\Delta}=0$ for all $e$, and therefore using \eqref{eq:2-chain-k}, the boundary of $\sigma(v)$ is given by
\begin{align*}
 \partial \sigma(v) = \sum_{\Delta} v_\Delta \partial \sigma_\Delta = \frac{1}{2 \pi \sqrt{-1}}\sum_\Delta \sum_{e} \frac{v_\Delta k_{e,\Delta}}{B_e} \rho_e  = 0.
\end{align*}
so that $\sigma(v)$ is a cycle. 

 The commutativity of the diagram follows similarly: using \eqref{eq:2-chain-b}, \eqref{eqn:adelta} and \eqref{eqn:Cdelta} in turn, we have
\begin{align*}
\exp \rbrac{\int_{\sigma(v)} \omega} &= \exp\rbrac{ \sum_{e,\Delta} \frac{v_\Delta k_{e,\Delta}}{2 \pi \sqrt{-1}B_e} \int_{\tilde \gamma_{e,\Delta}^0-\tilde \gamma_{e,\Delta}^1}\omega } \\
&= \prod_{e,\Delta} a_{e,\Delta}^{-v_\Delta k_{e,\Delta}} \\
&= \prod_\Delta (C_\Delta/\theta_\Delta)^{v_\Delta} \\
&=: (\tilde C/\theta)(v),
\end{align*}
as desired.
\end{proof}

Combining \autoref{prop:fibre-structure} and \autoref{prop:fibre-nonempty} we obtain a complete description of $\F$ as a set: it is partitioned into the subvarieties $\F_\Gamma$ which are disjoints unions of tori whose dimension and number are given by the integral linear algebra of the map $\kappa$ determined by the smoothing diagram and whose nonemptiness is determined by the  periods of $\omega$ along the cycles $\sigma(v) \in \Hlgy[2]{\Xo;\CC}$, for $v \in \ker \kappa$ via the condition \eqref{eq:fibre-nonempty-period}.  It follows that as we vary $\omega$ in $\coH[2]{\Xo;\CC} = \Def{\Xo}$, the subvarieties $\F_\Gamma$ assemble into  principal bundles over the subvarieties of $\coH[2]{\Xo;\CC}$ on which the smoothing subdiagram $(\Gamma,m|_{\Gamma})$ is constant, the monodromies of all $\sLe$, $e \in \tilde \Gamma$ remain trivial, and the period conditions \eqref{eq:fibre-nonempty-period} hold. Since the base of such a bundle is a formal germ, it is contractible by~\cite{Gilmartin1964}, and hence the bundle is trivial.  Note that $\int_{\sigma(v)}\omega$ is a linear combination of ratios of periods with integer coefficients, while $\theta(v)$ is a product of biresidues.  Taking the logarithm of both sides and clearing denominators, \eqref{eq:fibre-nonempty-period} becomes a polynomial equation on the periods of $\omega$ and their logarithms.

The irreducible decomposition of $\Def{\X}$ then has the following description.  A subvariety $\F_\Gamma$ lies in the closure of some $\F_{\Gamma'}$ if and only if $\Gamma \subset \Gamma'$ and $\F_{\Gamma'}$ is nonempty.  In this case, the bundle corresponding to $\Gamma$ lies in the closure of the bundle corresponding to $\Gamma'$ if and only if their images in $\coH[2]{\Xo;\CC}$ are the same.  Thus the irreducible components are in bijection with pairs $(\Gamma,\Z)$ where $\Gamma$  is a collection of smoothable edges that is maximal amongst all collections imposing the same conditions on periods, and $\Z$ is an irreducible component of $\F_\Gamma$, the latter being enumerated by torsion elements in $\coker \kappa$.

\bibliographystyle{hyperamsplain}
\bibliography{hol-log-can}

\end{document}